\documentclass[10pt]{amsart}
\usepackage{amssymb}
\usepackage{latexsym}
\usepackage{epsfig}
\usepackage{graphicx}
\usepackage{psfrag}
\usepackage{ amssymb, amsmath}
\usepackage{amsfonts}
\usepackage{amsthm}
\usepackage{enumerate}
\usepackage{color}
\newtheorem{theorem}{Theorem}[section]
\newtheorem{corollary}[theorem]{Corollary}
\newtheorem{lemma}[theorem]{Lemma}
\newtheorem{proposition}[theorem]{Proposition}

\newtheorem{definition}[theorem]{Definition}

\def\be{\begin{equation}}
\def\ee{\end{equation}}
\def\bes{\begin{equation*}}
\def\ees{\end{equation*}}
\def\bea{\begin{eqnarray}}
\def\eea{\end{eqnarray}}
\def\eeas{\end{eqnarray*}}
\def\beas{\begin{eqnarray*}}

\def\atxt{}

\numberwithin{equation}{section}

\begin{document}

\title[Resolution of singularities and applications]{A multi-dimensional resolution of singularities with applications to analysis}
\author{Tristan C.\ Collins, Allan Greenleaf and Malabika Pramanik}
\thanks{The second author was partially supported by NSF grants
DMS-0138167 and DMS-0551894. The third author was partially supported by NSERC grant 22R82900.}

\date{}

\maketitle
\begin{abstract}
We formulate a resolution of singularities algorithm for analyzing the zero sets of real-analytic functions in dimensions $\geq 3$. Rather than using the celebrated result of Hironaka, the algorithm is modeled on a more explicit and elementary approach used in the contemporary algebraic geometry literature. As an application, we {\atxt define a new notion of the height of  real-analytic  functions}, compute the critical integrability index, and obtain sharp growth rate of  sublevel sets. This also leads to a characterization of the oscillation index of scalar oscillatory integrals with real-analytic phases in all dimensions.  


\end{abstract} 
\setcounter{tocdepth}{1}
\tableofcontents 

\section{Introduction}\label{sec-intro}
The structure of the zero set of a multivariate analytic function is a topic of wide interest, in view of its ubiquity in problems of analysis, partial differential equations, probability and geometry. The study of such sets originated in the pioneering work of Jung \cite{Jung08}, Abhyankar \cite{Abh55} \cite{Abh64} and Hironaka \cite{Hir73}, \cite{Hir74}. Since then this field, known in algebraic geometry literature as resolution of singularities, has seen substantial advances, with contributions by Bierstone and Milman \cite{BieMil88} \cite{BieMil90}, Sussmann \cite{Suss90}, Parusi\'nski \cite{Par01} \cite{Par94}, among many others. The latter body of work has a strong analytical component that makes it adaptable, at least in principle, to a variety of problems in analysis where zero sets of polynomials or analytic functions play a crucial role. More recently, work by Phong, Stein and Sturm \cite{PhSt97, PhSteStu99, PhStu00, PhStu02}, Greenblatt \cite{Grblt04,Grblt08}, and Ikromov, Kempe and M\"uller \cite{IkrMul07, IkrKemMul10, IkrMul10} have explicitly addressed resolution of singularities with the goal of applications to such analytical questions, several of these in the bivariate setting. The purpose of this article is to formulate an algorithm for resolving singularities of a multivariate real-analytic function and apply it to a class of related problems for an $(n+1)$-variate real-analytic function, $n \geq 2$:
\begin{enumerate}[(i)]
\item computation and extremal characterization of the critical integrability index (see below for a definition), 
\item estimation of the size of sublevel sets for such functions, and 
\item description of the oscillation index of a class of scalar oscillatory integrals with real-analytic phase.  
\end{enumerate}    
These problems are of importance in a multitude of questions in harmonic analysis, PDE and geometry. For example, the scalar oscillatory integrals arise naturally as Fourier transforms of surface-carried measures (see \cite[Ch.\ 6]{AGV88} or \cite[Ch.\ 8]{Stein-book}, \cite{Var76, Vass77, Chr85, Chr98, Kar98}), while their operator analogues are central to the study of regularity properties of Fourier integral operators \cite{PhSt97,See98,Rych01,Grblt05, GrSee02}. The critical integrability index and sublevel set estimates, in the context of holomorphic functions in $\mathbb C^{n+1}$, provide important input in questions involving solvability of certain Monge-Amp\`ere equations. We will not expand on these well-established connections in this article. Instead we refer the reader to \cite{PhSt97}--\cite{PhStu02}, \cite{Grblt04}--\cite{Grblt11} and the bibliographies therein for an exhaustive treatment of the subject. 

\subsection{Description of the problems} Let $F$ be a $C^{\infty}$ function in a neighborhood $U_0$ of the origin in $\mathbb R^{n+1}$, $F(\mathbf 0, 0) = 0$. The {\em{critical integrability index}} of $F$, denoted $\mu_0(F)$, is defined as  
\begin{equation} \mu_0(F) := \sup \left\{\mu > 0 : \begin{aligned} &\text{ there exists an open set } U \subseteq U_0 \\ &\text{ such that } |F|^{-\mu} \in L^1(U) \end{aligned} \right\}. \label{cii-def} \end{equation} 
For a general $C^{\infty}$ function $F$ and in the absence of any ``finite type'' hypotheses, $\mu_0(F)$ may be infinite. However it is easily seen to be finite if $F$ is real-analytic. When $F$ is a bivariate real-analytic function, i.e., \ $n=1$, it has been proved by Phong, Stein and Sturm \cite{PhSteStu99} (see \cite{Grblt06} for a treatment of the $C^{\infty}$ case) that there exists an analytic coordinate system $\Phi_0$ such that \begin{equation} \mu_0(F) = \delta_0(F;\Phi_0), \label{def-adapted} \end{equation} where $\delta_0(F;\Phi)$ denotes the Newton {\atxt exponent} of $F$ in the coordinate system $\Phi$. (See \S \ref{sec-newton} for the definition of Newton {\atxt  exponent and the Newton distance $d_0(F;\Phi):=\delta_0(F;\Phi)^{-1}$}). It is also implicit in \cite{PhSteStu99} that the following characterization holds:
\begin{equation} \mu_0(F) = \inf \left\{ \delta_0(F;\Phi) \Bigl| \begin{aligned}  &\Phi \text{ is an analytic coordinate system with} \\ &\text{non-vanishing Jacobian}, \Phi(0,0) = (0,0) \end{aligned} \right\}. \label{cii-d2-char} \end{equation}  

Following the notation of \cite{Var76}, we will call a coordinate system $\Phi_0$ {\em{adapted to $F$}} if (\ref{def-adapted}) holds. For a bivariate real-analytic function $F$, the adapted coordinate system $\Phi_0$ takes one of two possible forms, 
\begin{equation} \label{phong-stein-adapted} (x,y) \mapsto (x, y - r(x)) \quad \text{ or } \quad (x, y) \mapsto (x - r(y), y), \end{equation} where $r$ is a univariate real-valued real-analytic function with $r(0) = 0$. Theorem 4 of \cite{PhSteStu99} also implies the following description of $\mu_0(F)$: if $\mathcal C_{\omega}$ denotes the class of all real-analytic coordinate transformations of the form (\ref{phong-stein-adapted}), then  
\begin{equation} \mu_0(F) =\delta_0(F):= \inf \left\{ \delta_0(F; \Phi): \Phi \in \mathcal C_{\omega}\right\}. \label{cii-ps-char} \end{equation} 
In fact, it follows from the proof of this theorem (see also \cite{Pra02,IkrMul07}) that there exists a finite subcollection $\mathcal C_{\omega}^{\ast}$ of $\mathcal C_{\omega}$ 
such that the infimum in (\ref{cii-ps-char}) is attained for a member of $\mathcal C_{\omega}^{\ast}$. A closer inspection of the analysis in \cite{PhSteStu99} shows that the functions $y = r(x)$ or $x = r(y)$ appearing in $\mathcal C_{\omega}^{\ast}$ are real-analytic functions that are related to the real parts of the roots of $F(x,y)$ in a neighborhood of the origin. Specifically, the functions $r$ are the ``principal root jets'' in the language of \cite{IkrMul07}.

On the other hand, it is also known that the characterization (\ref{cii-ps-char}) fails to hold for $n > 1$ (see \cite{Var76} and also \S \ref{sec-bivariate-comparison} for a simpler example). This gives rise to the following natural question: for $n > 1$, is there an appropriate generalization $\mathcal C$ of the class of analytic coordinate systems such that (i) for every $\Phi \in \mathcal C$, the concepts of Newton distance {\atxt and exponent} continue to be meaningful for $F \circ \Phi$; and (ii) for which (\ref{cii-ps-char}) holds with $\mathcal C_{\omega}$ replaced by $\mathcal C$? Our main result (Theorem \ref{main-thm1}) answers this question in the affirmative. In particular it provides an inductive algorithm (based on dimension) for constructing this class of {\atxt local} coordinates, using real-analytic functions in lesser number of variables.     

Closely related to the critical integrability index of a multivariate real-analytic function is its {\em{sublevel set growth rate}} $\nu_0(F)$, defined by 
\begin{equation} \label{def-nu0}
\begin{aligned} 
\nu_0(F) &:= \sup \left\{ \nu > 0 : \sup_{\epsilon > 0} \epsilon^{-\nu} \bigl|\mathcal E_{\epsilon}(F)\bigr| < \infty \right\}, \text{ where } \\ \mathcal E_{\epsilon}(F) &:= \bigl\{(\mathbf x, x_{n+1}) \in U_0 : |F(\mathbf x, x_{n+1})| < \epsilon \bigr\}. 
\end{aligned} 
\end{equation} 
An easy application of Chebyshev's inequality shows 
\[ |\mathcal E_{\epsilon}(F)| \leq C_{\mu} \epsilon^{\mu} \quad \text{ for any $\mu>0$ such that } \quad \int_{U_0}|F|^{-\mu} < \infty, \]
proving that $\mu_0(F) \leq \nu_0(F)$. Conversely, if $\mu, \nu > 0$ are such that
\begin{equation} \int_{U_0} |F|^{-\mu} = \infty \qquad \text{ and } \qquad \bigl| \mathcal E_{\epsilon}(F) \bigr| \leq C_{\nu} \epsilon^{\nu}, \label{mumu'}  \end{equation} 
then 
\begin{align*} 
\infty = \int_{U_0} |F|^{- \mu} &\leq \sum_{j=-C}^{\infty} \int_{|F| \sim 2^{-j}} |F|^{-\mu} \\ &\leq C_{\nu} \sum_{j=-C}^{\infty} 2^{j(\mu-\nu)}, \quad \text{ which implies } \quad \mu \geq \nu. 
\end{align*}    
Taking the infimum over $\mu$ and supremum over $\nu$ satisfying (\ref{mumu'}), we obtain $\mu_0(F) = \nu_0(F)$. Thus Theorem \ref{main-thm1} below specifies the sharp sublevel set growth rate as well. 

We now describe the problem of computation of the oscillation index of a scalar oscillatory integral. Let 
\[ \mathcal I(F, \phi; \lambda) := \int_{\mathbb R^{n+1}} e^{i \lambda F(\mathbf x, x_{n+1})} \phi(\mathbf x, x_{n+1}) \, d\mathbf x \, dx_{n+1}. \]
Here $\phi$ is a smooth real-valued function supported within the domain of definition of $F$, and $\lambda$ is large real parameter. We are interested in the behavior of $\mathcal I(F, \phi; \lambda)$ as $\lambda \rightarrow \infty$. It is known that (see \cite{Grblt11} \cite{Var76}) for $\phi$ supported in a sufficiently small neighborhood of the origin, $\mathcal I(F, \phi; \lambda)$ admits an asymptotic expansion 
\begin{equation} \label{oscind-asymp}
\mathcal I(F, \phi; \lambda) \sim \sum_{k=0}^{\infty} \sum_{j=0}^{n} a_{jk}(F, \phi) \bigl( \ln \lambda \bigr)^j \lambda^{- r_k},  \qquad \lambda \gg 1,     
\end{equation} 
where $\{ r_k \}$ is an increasing arithmetic progression of positive real numbers, depending only on the zero set of $F$ and independent of $\phi$. The {\em{oscillation index}} $\rho_0(F)$ is defined as 
\begin{equation} \label{def-rho} \rho_0(F) := \min \left\{r_0 : \begin{aligned} &\text{ there is an open set } U_0 \subseteq \mathbb R^{n+1}, (\mathbf 0, 0) \in U_0, \\ &\text{ and a smooth function } \phi \in C_0^{\infty}(U_0) \\ &\text{ such that } a_{j0}(F, \phi) \ne 0 \text{ for some } 0 \leq j \leq n  \end{aligned}  \right\}. \end{equation}   

The connection between the oscillation index and the geometry of the zero set of a real-analytic function was first observed by Arnold \cite{Arn72} \cite{Arn75} and studied by Varchenko \cite{Var76}. Subsequent developments in the scalar and operator theory may be found in the references mentioned earlier in the introduction. In particular, it has been shown in \cite{Grblt11}, \cite{Grblt10}, \cite{Var76} that $\mathcal I(F, \phi; \lambda)$ decays as fast or faster than the decay rate corresponding to $|\mathcal E_{\epsilon}(F)|$ with $\epsilon = \lambda^{-1}$. In fact the proof of Theorem 1.6 in \cite{Grblt10} shows that $\rho_0(F) = \nu_0(F)$ unless $\nu_0(F)$ happens to be an odd integer. 
 
\subsection{Statement of the results} 
Our goal is to develop a systematic resolution of singularities algorithm in several variables that allows effective computation as well as several extremal characterizations of all three indices described above.  
\begin{theorem}
Given a nonconstant real-analytic function $F$ with $F(\mathbf 0, 0) = 0$ defined in a neighborhood of the origin in $\mathbb R^{n+1}$, there exists $\epsilon_0 > 0$ and an orthogonal change of coordinates $(\mathbf x, x_{n+1})$ such that $F$ expressed in these coordinates satisfies the following property:
 
Let $\mathcal C = \mathcal C(F)$ denote the class of all coordinate transformations of the form $\Phi = \Phi(\varphi, V, r)$, 
\begin{equation} (\mathbf x, x_{n+1}) = \Phi(\mathbf y, y_{n+1}) = (\varphi(\mathbf y), y_{n+1} + r(\mathbf y)) \label{coordclass-def} \end{equation} 
where
\begin{enumerate}[(a)] 
\item The set $V \subseteq (-\epsilon, \epsilon)^n \setminus \cup_{j=1}^{n} \{x_j = 0 \}$ is open and connected with $\mathbf 0 \in \overline{V}$, for some $0 < \epsilon < \epsilon_0$. \label{mainthm1-parta}
\item The vector-valued function $\varphi$, initially defined on an open neighborhood of the origin containing $[0,1]^n$, is a system of coordinates on $V$ when restricted to $(0,1)^n$. More precisely, \label{mainthm1-partb}
\begin{enumerate}[$-$]
\item each entry of $\varphi$ admits a multivariate convergent Puiseux expansion at the origin and is a fractional power series on $(0,1)^n$, 
\item $\varphi(\mathbf 0) = \mathbf 0$, 
\item $\varphi : (0,1)^n \rightarrow V$ is a $C^1$ bijection onto $V$ with nonvanishing Jacobian on $[0,1]^n$. 
\end{enumerate}
\item The function $r$ is initially defined as a scalar-valued function on an open neighborhood of the origin containing $[0,1]^n$. It admits a multivariate convergent Puiseux expansion at the origin, with $r(\mathbf 0) = 0$, and is an $n$-variate real-valued fractional power series on $(0,1)^n$. \label{mainthm1-partc}  
\item If $\{ r_i(\mathbf y) : 1 \leq i \leq N \}$ is the collection of roots (in $y_{n+1}$) of $F \circ \Phi(\mathbf y, y_{n+1})$ with $r_i(\mathbf 0) = 0$, then every element of $\mathcal A = \{r_i, \text{Re}(r_i) : 1 \leq i \leq N \}$ and the difference set $\mathcal A - \mathcal A$ is either identically zero or fractional normal crossings for $\mathbf y \in (0,1)^n$. \label{mainthm1-partd}
\end{enumerate}
Then 
\[\mu_0(F) =  \inf \left\{ \delta_0(F;\Phi) : \Phi \in \mathcal C \right\}. \] 
  
The infimum in the equation above is attained; in fact, one can specify an algorithm that produces for any $F$ a finite subcollection $\mathcal C^{\ast} \subseteq \mathcal C$ such that 
\begin{equation} \mu_0(F) = \min \left\{ \delta_0(F;\Phi) : \Phi \in \mathcal C^{\ast} \right\}. \label{coordclass-finite} \end{equation}  
Every $\Phi = \Phi(\varphi, V, r) \in \mathcal C^{\ast}$ satisfies the property that the fractional power series $r(\mathbf y)$ is the real part of a root of $F \circ \Phi(\mathbf y, y_{n+1})$ with respect to $y_{n+1}$. 
\label{main-thm1}
\end{theorem}
{\bf{Notation and terminology:}} 
\begin{enumerate}[1.]
\item {\atxt The words ``multivariate Puiseux expansion'', ``fractional power series'', ``fractional normal crossings'' and ``coordinate systems'' used in the statement of the theorem have specific meanings in the context of this paper. For precise definitions see \S\S \ref{subsec-defns} and \S\S \ref{subsec-coords}.
\item The notation $A-A$ stands for algebraic difference, not set difference; i.e., $A-A = \{x-y : x, y \in A \}$.}
\end{enumerate} 
{\bf{Remarks:}}
\begin{enumerate}[1.]
\item It is important to note that the open set $V$ in the definition of $\Phi(\varphi, V, r) \in \mathcal C$ need not contain the origin, even though the origin lies in its closure. For instance, $V$ could be shaped like a horn, with the origin at its cusp. As a consequence, the transformation $\varphi$ need not in general admit an extension as a coordinate system on an open neighborhood of the origin. {\atxt The necessity of using such ``local" adapted coordinates, rather than ``global'' ones, is perhaps the most significant new feature of the higher dimensional problem,} in contrast with the bivariate situation.    
\item \label{remark-Cstar} The construction of $\mathcal C^{\ast}$ is explicit and involves resolving the singularities of a finite number of real-analytic functions in lesser number of variables. More precisely, $\mathcal C^{\ast}$ is built in at most two steps. 
\begin{itemize}
\item In the first step, we define an auxiliary $n$-variate real-analytic function $\Lambda$ that is constructible using the coefficients (in $\mathbf x$) of $F$ viewed as a polynomial in $x_{n+1}$, in a sense made precise in \S\S \ref{Construction-Lambda}. The function $\Lambda$ is a multiple of the discriminant of $F$. Monomialization of $\Lambda$ via the inductive mechanism gives rise to a finite collection of set-coordinate pairs $\{(U, \zeta_U) : U \in \mathcal U \}$, $\zeta_U:(0,1)^n \rightarrow U$, $\mathbf x = \zeta_U(\mathbf u)$,  
such that the transformed function $\Lambda \circ \zeta_U(\mathbf u)$ is fractional normal crossings in $\mathbf u$. 
\item For every $U \in \mathcal U$ and depending on the nature of the roots of $F(\zeta_U(\mathbf u), x_{n+1})$ in $x_{n+1}$, we may then need to construct a second $n$-variate real-analytic function $H_U(\mathbf u)$. The process of monomializing $H_{U}$ generates another finite collection of sets and coordinates $\{(W, \eta_W) : W \in \mathcal W \}$ with $\cup \{W : W \in \mathcal W \} \subseteq (0,1)^n$. 
\end{itemize} 
We will see that any $(V, \varphi)$ with $\Phi(\varphi, V, r) \in \mathcal C^{\ast}$ must be of the form $V = \varphi(0,1)^n$, where $\varphi$ is a composition of $\zeta_U$, $\eta_W$ and appropriate power transformations. In particular, $\Lambda$ and $H_U$, when expressed in the final coordinate system, will be fractional normal crossings.   

\item Further details about the structure of the subsets $V$ and the associated coordinate systems $\varphi$ have been summarized in Theorem \ref{mainthm-resolution}, with some background material in \S \ref{sec-resolution-preliminaries}. It follows that the coordinate transformations $\varphi$ are compositions of elementary transformations, see \S\S \ref{examples}. 
\item A key technical tool is that the Newton {\atxt polyhedra} of $F \circ \Phi$, for $\Phi \in \mathcal C$, have a very special structure described in \S \ref{sec-newton}. The details are in Theorem \ref{mainthm-resolution}. 
\item We observe that when $n = 1$, the class $\mathcal C$ contains $\mathcal C_{\omega}$.    
\item Greenblatt\cite{Grblt08} has obtained similar results using somewhat different methods, with an argument involving induction on both the dimension and the order of vanishing at the origin.
\item { Sets like $V$ in Theorem \ref{main-thm1} (termed ``towers'' in this paper, see \S \ref{sec-resolution-preliminaries}) and coordinates of the form (\ref{coordclass-def}) given by nonlinear shears involving fractional normal crossings have already appeared in analytical problems involving resolution of singularities in the bivariate setting. For instance, they have been used by Phong and Stein \cite{PhSt97} in the study of oscillatory integral operators in one variable and by Ikromov, Kempe and M\"uller \cite{IkrKemMul10} in the context of maximal operators associated with two-dimensional surfaces in $\mathbb R^3$.} 
\item We anticipate that the method of resolution of singularities described in this article will have applications beyond those addressed here. In particular, this method can be used in certain operator-theoretic problems, such as the sharp $L^2$-decay exponent for oscillatory integral operators with multivariate real-analytic phases \cite{GrePra11-pre}. 
\end{enumerate} 

Using analytical tools from convex geometry developed in \cite{Nagel-Pramanik09} in conjunction with the theorem above, 
it is possible to strengthen the first conclusion of Theorem \ref{main-thm1} by expanding the class $\mathcal C$, so that the characterization of the critical integrability index for $(n+1)$-variate real-analytic functions bears a closer resemblance to the corresponding statements (\ref{cii-d2-char}) and (\ref{phong-stein-adapted}) in dimension 2. Let $\mathcal C_{\infty}$ denote the universal class of all coordinate systems in the sense of definitions \ref{def-coord-tf} and \ref{def-coord-tf-2}, i.e., 
\[ \mathcal C_{\infty} := \left\{\Phi \Biggl| \;\begin{aligned} &\Phi \text{ is a vector-valued function on an open set containing } [0,1]^{n+1}, \\ &\Phi \text{ admits a convergent Puiseux expansion at the origin}, \\ & \Phi(\mathbf 0, 0) = (\mathbf 0, 0), \\&\Phi \text{ is a fractional power series on } (0,1)^{n+1},  \\  &\Phi  \text{ maps $(0,1)^n$ bijectively onto some open set $\mathcal O \subseteq \mathbb R^{n+1},\; (\mathbf 0,0) \in \overline{\mathcal O}$.} \end{aligned}   \right\}. \]  
Also, let $\mathcal C_1$ denote the superset of $\mathcal C$ consisting of transformations of the form (\ref{coordclass-def}) where $\varphi, V, r$ satisfy (\ref{mainthm1-parta}), (\ref{mainthm1-partb}), (\ref{mainthm1-partc}) of Theorem \ref{main-thm1} but not necessarily (\ref{mainthm1-partd}). Note that $\mathcal C \subseteq \mathcal C_1 \subseteq \mathcal C_{\infty}$.  
\begin{theorem} \label{main-thm2}
Let $F$ be as in Theorem \ref{main-thm1}. Then 
\[ \mu_0(F) = \inf\{\delta_0(F, \Phi) : \Phi \in \mathcal C_{\infty} \} = \inf \left\{\delta_0(F, \Phi) : \Phi \in \mathcal C_1 \right\}. \] 
\end{theorem}
Thus for the critical integrability index, the class $\mathcal C_{\infty}$ (respectively $\mathcal C_1$) plays the same role that the class of all analytic coordinate systems (respectively $\mathcal C_{\omega}$) does in two dimensions. {\atxt Thus, the quantities 
\begin{equation} \label{new-height}
d_0(F) := \sup\{d_0(F, \Phi) : \Phi \in \mathcal C_{\infty} \},\quad \delta_0(F):=\inf \{\delta_0(F;\Phi):  \Phi \in \mathcal C_{\infty} \}
\end{equation} 
may be viewed as the appropriate  generalizations to several variables of the notions of the ``height''   of a real-analytic function,  introduced in \cite{Var76}, and the Newton exponent, resp. }

As a consequence of Theorem \ref{main-thm1} and the discussion preceding it, we are able to estimate the size of sublevel sets for real-analytic functions and compute the oscillation index for scalar oscillatory integrals with real-analytic phases. 
\begin{corollary}
Let $F$ be as in Theorem \ref{main-thm1}, and $\nu_0(F)$, $\rho_0(F)$ and $\delta_0(F)$ be as in (\ref{def-nu0}), (\ref{def-rho}) and (\ref{new-height}) respectively. Then $\nu_0(F) = \delta_0(F)$. Further, $\rho_0(F) = \delta_0(F)$ if $\delta_0(F)$ is not an odd integer. 
\end{corollary} 

The specification of critical integrability, sublevel set growth and oscillation indices translate to estimates for $\int |F|^{-\delta}$, $|\mathcal E_{\epsilon}(F)|$ and $\mathcal I(F, \phi; \lambda)$ that are sharp excluding the endpoint. It is possible to refine these questions further by asking, for instance, about the order of the pole of the meromorphic functions $\delta \mapsto \int F_{\pm}^{-\delta}$ at $\delta = \mu_0(F)$, or the possible occurrence of logarithmic terms in the estimates for $|\mathcal E_{\epsilon}(F)|$ and in the leading term of the asymptotic expansion (\ref{oscind-asymp}). We hope to return to these finer issues in a future paper.

\subsection{Layout} The paper is arranged as follows. \S \ref{sec-resolution} contains the main resolution of singularities algorithm that lies at the heart of the analysis. More precisely, given a real-analytic function $F$ near the origin in $\mathbb R^{n+1}$ with $F(\mathbf 0, 0) = 0$, we formulate a recursive process that leads to the desingularization or monomialization of its roots. This is in slight contrast with the standard nomenclature of resolution of singularities where the function $F$ itself is monomialized. However, we will see that the former implies the latter (see Proposition \ref{lemma-block2}), along with good control on the Newton polyhedron. The goal of our (local) resolution process is to obtain 
\begin{enumerate}[1.]
\item a small open parallelepiped $(-\epsilon, \epsilon)^{n+1}$ centered at the origin, 
\item a finite collection of open subsets $\{V_i : i \in \mathcal I \}$ constructed inductively based on dimension whose union covers $(-\epsilon, \epsilon)^{n+1}$ except possibly a lower dimensional subset,  
\item a corresponding collection of fractional power series $\{\varphi_i : i \in \mathcal I \}$, with $\varphi_i \in C^1$ mapping $(0,1)^n$ bijectively onto $V_i$, and $\det(D \varphi_i)$ uniformly bounded away from zero,   
\end{enumerate} 
with the property that the roots of the function $F(\varphi_i(\mathbf y), x_{n+1})$ with respect to $x_{n+1}$ are ``ordered'' fractional normal crossings in the variables $\mathbf y$.
That this can be done (and indeed in much greater generality than the setup described above) is the foundational result of Hironaka \cite{Hir73} \cite{Hir74}. More elementary and explicit algorithms for resolution of singularities have subsequently been studied in great detail in the seminal work of Bierstone and Milman \cite{BieMil88} \cite{BieMil90}, Sussmann \cite{Suss90} and Parusi\'nski \cite{Par01} among others. In a more analytical framework, we also mention the work of Phong, Stein and Sturm \cite{PhSt97}--\cite{PhStu02}, Greenblatt \cite{Grblt04}--\cite{Grblt11}, Ikromov, Kempe and M\"uller \cite{IkrMul07}--\cite{IkrMul10}. In \S \ref{sec-resolution} we present a simple and self-contained exposition of resolution of singularities modeled on \cite{BieMil90} \cite{Suss90} \cite{Par01}, with special attention to those aspects of the algorithm that are relevant for the computation of the critical integrability index.

Theorems \ref{main-thm1} and \ref{main-thm2} have been proved in \S \ref{sec-proof-mainthm1}. \S \ref{sec-preliminaries} and \S \ref{sec-newton} include a collection of algebraic, geometric and analytic tools that are needed in the proof of Theorem \ref{mainthm-resolution}. \S \ref{sec-resolution-preliminaries} contains a description of the sets and coordinates that occur in the resolution process. In \S \ref{sec-example}, the resolution algorithm has been carried out in the context of an example. Further examples and counterexamples comparing and contrasting the bivariate situation with higher dimensions are in \S \ref{sec-bivariate-comparison}. Supplementary proofs whose techniques are not directly connected to the main content of the article have been relegated to \S \ref{sec-appendix} and \S \ref{sec-appendix2}. 

\section{A brief review of analytic sets} \label{sec-preliminaries}
The resolution of singularities scheme described in \S \ref{sec-resolution} draws heavily on certain fundamental results that lie at the interface of analysis and algebraic geometry. We summarize these facts in this section, without proof but with appropriate references, and record a list of standard definitions and notation that will be used extensively in the remainder of the paper.  
\subsection{Notation}
All the vectors in this article will be of dimension $(n-1)$, $n$ or $(n+1)$. We will use $\mathbf x$ to denote a vector in $\mathbb R^n$, and $\mathbf x'$ for its projection onto the first $(n-1)$ coordinates, so that $\mathbf x = (\mathbf x', x_n)$. An $(n+1)$-dimensional vector will be denoted by $(\mathbf x, x_{n+1})$. 

Let us denote by ``Log'' the branch of the logarithm defined on the slit complex plane with the positive imaginary axis removed, i.e., 
\[ \text{Log}(z) = \log|z| + i \arg(z), \quad z \in \mathbb C \setminus \{iy : y \geq 0 \}, \; -\frac{3 \pi}{2} < \arg(z) < \frac{\pi}{2}. \] For any $r \in \mathbb R$, the power function $z^r$ is then defined to be 
\begin{equation} z^r := \text{exp}(r \text{Log}(z)). \label{def-power}\end{equation} The map $x \mapsto x^r$ is therefore well-defined on $\mathbb R \setminus \{0\}$, satisfies the consistency conditions \begin{equation} (x^r)^s = (x^{s})^r = x^{rs} \text{ for } x > 0 \text{ and } \text{ for all } r, s \in \mathbb R, \label{power-consistency}  \end{equation} and admits a continuous extension to $\mathbb R$ for $r \geq 0$.

Given $\mathbf x = (x_1, \cdots, x_n) \in \mathbb R^n$ with $x_1 x_2 \cdots x_n \ne 0$, $\pmb{\alpha} = (\alpha_1, \cdots, \alpha_n) \in [0, \infty)^n$, $\pmb{\beta} = (\beta_1, \cdots, \beta_n) \in (0,\infty)^n$ and $m \geq 0$, we write \begin{align*} \mathbf x^{\pmb{\alpha}} &= x_1^{\alpha_1} \cdots x_n^{\alpha_n}, \qquad \mathbf x^{\frac{\pmb{\alpha}}{\pmb{\beta}}} = x_1^{\frac{\alpha_1}{\beta_1}} \cdots x_n^{\frac{\alpha_n}{\beta_n}}, \qquad \mathbf x^m = x_1^m \cdots x_n^m, \\ \Phi_m(\mathbf x) &= (x_1^m, \cdots, x_n^m), \qquad \Phi_{\pmb{\alpha}}(x_1, \cdots, x_n) = (x_1^{\alpha_1}, \cdots, x_n^{\alpha_n}). \end{align*}  

An ordering of multi-indices will be repeatedly used. Given two multi-indices $\pmb{\alpha} = (\alpha_1, \cdots, \alpha_n)$ and $\pmb{\beta} = (\beta_1, \cdots, \beta_n)$, we say that $\pmb{\alpha}$ is less than or equal to $\pmb{\beta}$ and write $\pmb{\alpha} \leq \pmb{\beta}$ if
\begin{equation}
\alpha_j \leq \beta_j \text{ for all $1 \leq j \leq n$ }. \label{order1}
\end{equation} 
A finite collection of multi-indices is said to be {\em{totally ordered}} if any two elements in the collection can be given an ordering as above. 
Needless to say, two arbitrary multi-indices need not be ordered in the sense of (\ref{order1}). If however there is a total ordering of a finite set of exponent vectors, this induces a reverse ordering of the corresponding functions in a neighborhood of the origin, namely \[ \mathbf x^{\pmb{\alpha}} \leq \mathbf x^{\pmb{\beta}} \qquad \text{ if } {\pmb{\alpha}} \geq {\pmb{\beta}} \; \text{ and } \; \mathbf x \in (0,1)^n. \]   
We say that $\pmb{\alpha}$ is strictly smaller than $\pmb{\beta}$ and write $\pmb{\alpha} < \pmb{\beta}$ if \[\pmb{\alpha} \leq \pmb{\beta} \text{ and there exists } 1 \leq j \leq n \text{ such that } \alpha_j < \beta_j.\]
  
\subsection{Definitions} \label{subsec-defns} 
The goal of the resolution process is to reduce the roots of an arbitrary multi-variate real-analytic function into a standard format, specifically re-express them as functions of monomial type. The definitions below describe the model functions of interest.   
\begin{definition}
A formal series $S(\mathbf x)$ given by \begin{equation} S(\mathbf x) = \sum_{\pmb{\kappa}= (\kappa_1, \cdots, \kappa_n)} a_{\pmb{\kappa}} \mathbf x^{\pmb{\kappa}}, \quad \mathbf x \in \mathbb R^n_{>0} = \{\mathbf t \in \mathbb R^n : t_j > 0, 1 \leq j \leq n \}  \label{fract-pow-series} \end{equation}
with complex coefficients $\{a_{\pmb{\kappa}}\}$ will be called a {\em{(multivariate) convergent Puiseux expansion}} in $\mathbf x$ if there is a positive integer $N$ for which 
\begin{equation} S(\Phi_N(\mathbf x)) = \sum_{\pmb{\kappa}} a_{\pmb{\kappa}} \mathbf x^{N \pmb{\kappa}} \label{regular-pow-series} \end{equation}
 is a convergent power series (in the usual sense) centered at $\mathbf 0$ with a nontrivial radius of convergence. In other words, the series $S$ is a convergent Puiseux expansion in the positive orthant if $N\pmb{\kappa}$ has non-negative integer entries for all multi-indices $\pmb{\kappa}$ with $a_{\pmb{\kappa}} \ne 0$ and there is an open neighborhood of the origin in $\mathbb R^n$ such that the series in (\ref{regular-pow-series}) converges absolutely and uniformly on all compact subsets of this neighborhood.

If $\pmb{\nu}$ is an $n$-dimensional vector with entries either $+1$ or $-1$, then a series \[ S(\mathbf x) = \sum_{\pmb{\kappa}= (\kappa_1, \cdots, \kappa_n)} a_{\pmb{\kappa}} \mathbf x^{\pmb{\kappa}} \; \text{ on } \{\mathbf x \in \mathbb R^n: \nu_j x_j > 0, 1 \leq j \leq n \} \] is a {\em{(multivariate) convergent Puiseux expansion}} if  $S(\nu_1x_1, \cdots, \nu_n x_n)$ is a convergent Puiseux expansion on $\mathbb R^n_{>0}$ in the above sense.    
\end{definition} 
{\bf{Remarks: }}
\begin{enumerate}[1.]
\item If $\mathbf r = (r_1, \cdots, r_n) \in (0, \infty)^n$ is a vector such that the power series $S \circ \Phi_N(\mathbf r)$ converges absolutely, then the multivariate Puiseux expansion (\ref{fract-pow-series}) is absolutely and uniformly convergent on compact subsets of the open parallelepiped $\prod_{j=1}^{n}(-r_j^{1/N}, r_j^{1/N})$. Unlike a regular power series that is real-analytic and hence $C^{\infty}$ on an open set containing the origin, a function given by the multivariate Puiseux expansion (\ref{fract-pow-series}) is in general only $C^{\infty}$ and real-analytic on an open set {\em{not containing the coordinate hyperplanes}}.

\item All multivariate Puiseux expansions considered in this article will be centered at the origin, i.e., of the form (\ref{fract-pow-series}) and hence absolutely and uniformly convergent on some open set containing the origin. On the other hand, we will often need to consider the extensions of these functions to certain open sets whose closures contain the origin (though the sets themselves may not) and on all of which the Puiseux expansions need not converge. With this in mind, we introduce certain variations of the definitions above. 
\end{enumerate} 
\begin{definition}
\begin{enumerate}[1.]
\item Given an open, connected set $V \subseteq \mathbb R^n_{>0}$ such that $\mathbf 0 \in \overline{V}$, we say that $f$ is a {\em{fractional power series on $V$}} if there is an open set $\mathcal O \subseteq \mathbb R^n$ containing $\overline{V}$, a positive integer $N \geq 1$ and a real-analytic function $g$ on $\Phi_N^{-1}(\mathcal O)$ such that $f = g \circ \Phi_{1/N}$ on $V$.   
\item A fractional power series $f$ on $V \subseteq \mathbb R^n_{>0}$ is said to be a {\em{unit on $V$}} if $\inf \{|f(\mathbf x)| : \mathbf x \in V \} > 0$. 
\item As in the case of multivariate Puiseux expansions, the above definitions extend naturally to the situation when $V$ is a subset of any orthant.   
\end{enumerate} 
\end{definition}
Note that since $\Phi_N^{-1}(\mathcal O)$ is an open set containing the origin, the real-analytic function $f \circ \Phi_N$ admits a power series representation centered at the origin with a nontrivial domain of convergence, even though the power series need not converge on all of $\Phi_N^{-1}(\mathcal O)$. Thus $f$ equals a multivariate Puiseux expansion of the form (\ref{fract-pow-series}) in a sufficiently small open subset of $V$. We claim that this Puiseux representation uniquely identifies $f$. 
\begin{lemma}
Let $V$ be an open connected subset of the positive orthant. A fractional power series on $V$ is uniquely identified by its multivariate Puiseux expansion at the origin. 
\end{lemma} 
\begin{proof}
Let $f$ and $g$ be two fractional power series on $V$ sharing the same multivariate Puiseux expansion at the origin. Since the Puiseux expansion is convergent, it represents a continuous function on some open neighborhood $U_0$ of the origin. Intersecting $U_0$ with $V$, we obtain an open set $U \subseteq V$ with $\mathbf 0 \in \overline{U}$ with the property that $f = g$ on $U$. In other words, $f \circ \Phi_N = g \circ \Phi_N$ on $\Phi_N^{-1}(U)$ for any positive integer $N$. On the other hand, it follows from the definition of fractional power series that for some $N$ sufficiently large, the functions $f \circ \Phi_N$ and $g \circ \Phi_N$ are both real-analytic on $\Phi_N^{-1}(V)$. Thus, if these functions agree on the nonempty open set $\Phi_N^{-1}(U)$, they must agree on every connected component of $\Phi_N^{-1}(V)$ as well. This implies that $f \equiv g$ on $V$.   
\end{proof} 

{\bf{Remarks:}}
\begin{enumerate}[1.]
\item We emphasize that contrary to certain conventions, our definition of a fractional power series $f$ on $V$ does not merely mean that for every $\mathbf a \in V$, $f$ admits a representation of the form (\ref{fract-pow-series}) with $\mathbf x$ replaced by $\mathbf x - \mathbf a$. Rather, it guarantees the existence of a single multivariate Puiseux expansion centered at $\mathbf 0$ (which does not lie in $V$) that converges absolutely and uniformly on some open neighborhood of the origin. While the domain of convergence of this specific Puiseux expansion may not cover all of $V$, the fractional power series $f$ is uniquely specified on $V$ by this expansion and coincides with it on their common domain of definition. 

\item At the same time, the existence of the multivariate Puiseux expansion at $\mathbf 0$ permits the fractional power series on $V$ to extend as a continuous function to an open set containing $\overline{V}$, and in particular assigns it an unambiguous value at the origin.  

\item In this paper, we will not need to deal with the actual domain of convergence of any multivariate Puiseux expansion. This is in general a nontrivial issue, but largely irrelevant in our present analysis. Our resolution of singularities procedure is local, in the following sense. Given a real-analytic function $F(\mathbf x, x_{n+1})$ on a neighborhood of the origin in $\mathbb R^{n+1}$, our goal is to obtain a small constant $\epsilon > 0$ and decompose the slit parallelepiped $(-\epsilon, \epsilon)^{n+1} \cap \{\mathbf x : x_1 x_2 \cdots x_{n+1} \ne 0 \}$ into a finite number of regions such that on each region the roots of $F$ with respect to $x_{n+1}$ (possibly after certain coordinate changes in $\mathbf x$) can be represented as fractional power series on $(0,1)^n$ in the new variables. This will be shown to be possible if $\epsilon$ is sufficiently small. The smallness of the ambient domain of definition of $F$ will be used without further reference in the sequel. Under this assumption, addressing the question of exact domains of convergence of the {\em{Puiseux representation}} of a root (in the new variables) will not be necessary, as long as the domain of definition of the corresponding fractional power series is large enough to contain the parallelepiped $(0,1)^n$.   

\item A significant feature of the class of fractional power series, in contrast with that of the regular ones, is ``non-closure under composition''. More precisely, if $f$ is a multivariate fractional power series and $g$ a fractional power series in a single variable, then $h = g \circ f$ need not be a fractional power series. An easy example is the function $h(x_1, x_2) = \sqrt{x_1^2+x_2^2}$, which is not a fractional power series, even though the individual components $f(x_1, x_2) = x_1^2 + x_2^2$ and $g(x) = \sqrt{x}$ are. However, the unit square in the $(x_1, x_2)$ space can be decomposed (excluding possibly a lower dimensional set) into regions on each of which $h(x_1, x_2)$ admits a fractional power series representation after a change of variables. For instance, on $\{ 0 < x_2 < x_1/2 \}$ and under the change of variables $x_1 = u_1$ and $x_2 = u_1 u_2/2$, $0<u_1, u_2 < 1$, the function $h$ reduces to $u_1 \sqrt{1 + u_2^2/4}$, which is a fractional power series in $(u_1, u_2) \in (0,1)^2$. Decompositions of domains and changes of variables such as these will be key points in our analysis.         
\end{enumerate} 

The following is a standard algebraic fact concerning a ring of fractional power series (see \cite[Lemma 1.19]{KiyVic04} or \cite[Ch.\ 4 and 5]{KiyVic-book}). Let $\mathcal R_n = \mathbb C[\langle \mathbf x \rangle]$ denote the ring of all power series with complex coefficients that admit a nontrivial radius of convergence about the origin in $\mathbb R^n$. Let $\mathcal Q_n$ be the field of quotients of $\mathcal R_n$. Fix an integer vector $\mathbf d = (d_1, \cdots, d_n) \in \mathbb N^n$, and let $\mathcal Q_n(\Phi_{{\mathbf 1}/{\mathbf d}}(\mathbf x))$ denote the splitting field of the polynomial $(t^{d_1} - x_1) \cdots (t^{d_n} - x_n)$ over $\mathcal Q_n$.
\begin{lemma} \label{lemma-integral-closure} 
 The ring $\mathbb C[\langle \Phi_{{\mathbf 1}/{\mathbf d}}(\mathbf x) \rangle]$ of convergent fractional power series in the variables $\{ x_i^{1/d_i} : 1 \leq i \leq n \}$ is the integral closure of $\mathcal R_n$ in $\mathcal Q_n(\Phi_{{\mathbf 1}/{\mathbf d}}(\mathbf x))$. In other words, every fractional power series in $\mathbb C[\langle \Phi_{{\mathbf 1}/{\mathbf d}}(\mathbf x) \rangle]$ is the root of some monic polynomial in $x_{n+1}$ with coefficients in $\mathcal R_n$.    
\end{lemma}  
\begin{proof} 
The statement of the lemma is a special case of the algebraic fact that $\mathbb C[\langle \Phi_{{\mathbf 1}/{\mathbf d}}(\mathbf x) \rangle]$ is a finite $\mathcal R_n$-algebra, i.e., every element in $\mathbb C[\langle \Phi_{{\mathbf 1}/{\mathbf d}}(\mathbf x) \rangle]$ is of the form \[ \sum_{\mathbf r} \mathbf x^{\mathbf r/\mathbf d} g_{\mathbf r}(\mathbf x) \quad \text{ for some  collection of functions } \{g_{\mathbf r}\} \subseteq \mathcal R_n, \] 
where the index $\mathbf r = (r_1, \cdots, r_n)$ ranges over the finite subset of $\mathbb Z^n$ such that $0 \leq r_j < d_j$ for all $1 \leq j \leq n$. The desired result now follows from \cite[Lemma 1.19]{Hulek-book}. 
\end{proof}
\begin{definition}
\begin{enumerate}[1.]
\item We say that $f(\mathbf x)$ is a {\em{monomial}} if it is of the form 
\begin{equation} f(\mathbf x) = a \mathbf x^{\pmb{\kappa}} =  a \prod_{j=1}^{n} x_j^{\kappa_j}, \qquad a \ne 0,\; \pmb{\kappa}> \mathbf 0  \label{def-monomial} \end{equation} and $\kappa_j \in \mathbb Z \cap [0, \infty)$ for all $1 \leq j \leq n$. A monomial is a well-defined analytic function on $\mathbb C^n$. 
\item A function is said to be {\em{normal crossings on a set $V \subseteq \mathbb R^{n}_{>0}$}} if it is the product of a monomial and a real-analytic function on $V$ that is uniformly bounded away from zero on $\overline{V}$. 
\end{enumerate}
\end{definition}
\begin{definition}
\begin{enumerate}[1.]
\item If the multi-exponent $\pmb{\kappa} > \mathbf 0$ has non-negative rational (not necessarily integer) entries, the corresponding function in (\ref{def-monomial}) is called a {\em{fractional monomial}}. A fractional monomial is well-defined on $\mathbb R^n$ away from the coordinate hyperplanes.  
\item A function is said to be {\em{fractional normal crossings on $V \subseteq \mathbb R^n_{>0}$}} if it is the product of a fractional monomial with a unit on $V$.
\end{enumerate}    
\end{definition} 
The following is an easy but frequently useful lemma (see \cite[Lemma 4.7]{BieMil88} or \cite[6.VI]{Suss90} or \cite[Lemma 4.3]{Par01} for a proof) concerning the ordering of fractional normal crossings and their exponents.  
\begin{lemma} \label{monomial-ordering-lemma}
Let $\pmb{\alpha}$, $\pmb{\beta}$ and $\pmb{\gamma}$ be multi-indices and let $a(\mathbf x)$, $b(\mathbf x)$ and $c(\mathbf x)$ be units. If 
\[ a(\mathbf x) \mathbf x^{\pmb{\alpha}} - b(\mathbf x) \mathbf x^{\pmb{\beta}} = c(\mathbf x) \mathbf x^{\pmb{\gamma}}, \]
then either $\pmb{\alpha} \leq \pmb{\beta}$ or $\pmb{\beta} \leq \pmb{\alpha}$.     
\end{lemma}

\subsection{The Weierstrass Preparation Theorem, Weierstrass polynomials} 
The Weierstrass preparation theorem is a classical result in the analysis of several complex variables that relates an arbitrary holomorphic (i.e., complex-analytic) function with one of polynomial structure. We use it in our analysis in the same way it was used in \cite{PhSteStu99} or \cite{PhSt97}.  
\begin{theorem}\cite[Theorem 6.4.5]{Krantz-SCVbook} \cite[p123]{Shabat-book} \label{W-prep}
Suppose that the function $F$ is holomorphic in some neighborhood $U_0$ of the origin in $\mathbb C^{n+1}$, with $F(\mathbf 0, 0) = 0$ and $F(\mathbf 0, z_{n+1}) \not\equiv 0$. Then in some neighborhood $U$ of the origin, $U \subseteq U_0$, the function $F$ may be written as  
\[ F(\mathbf z, z_{n+1}) = u(\mathbf z, z_{n+1}) \left[ z_{n+1}^d + c_1(\mathbf z) z_{n+1}^{d-1} + \cdots + c_d(\mathbf z) \right],  \]
where $d \geq 1$ is the order of the zero of $F(\mathbf 0, z_{n+1})$ at $z_{n+1} = 0$, the functions $c_\nu$ are holomorphic in $U$ with $c_{\nu}(\mathbf 0) = 0$, while $u$ is holomorphic and non-vanishing in $U$.  
\end{theorem} 
Given any holomorphic $F$ with $F(\mathbf 0, 0) = 0$, the condition $F(\mathbf 0, z_{n+1}) \not\equiv 0$ required by Theorem \ref{W-prep} can always be realized by a nonsingular linear change of coordinates. A {\em{Weierstrass polynomial}} is defined to be a monic polynomial in $z_{n+1}$ whose non-leading coefficients are holomorphic functions in $\mathbf z$ vanishing at $\mathbf 0$. Thus every holomorphic function vanishing at the origin in $\mathbb C^{n+1}$ is the product of a Weierstrass polynomial with a holomorphic unit, in suitable coordinates.    
\subsection{Factorization} \label{subsec-factor}
Let $\mathcal P$ denote the ring of polynomials in $z_{n+1}$ with holomorphic coefficients in $\mathbf z$. A polynomial $P \in \mathcal P$ is said to be {\em{irreducible}} if it cannot be written as the product of two polynomials $P_1, P_2 \in \mathcal P$ of degree at least one each. By a standard result in algebra, any polynomial $P \in \mathcal P$ can be represented in the form  
\begin{equation} P = P_1^{m_1} P_2^{m_2} \cdots P_L^{m_L}, \label{polynomial-factors}\end{equation} 
where the $P_{\ell}$-s are irreducible polynomials in $\mathcal P$, $m_{\ell} \in \mathbb N$ and the decomposition is unique upto factors that are holomorphic and nonzero in a neighborhood of the origin.
\begin{lemma} \cite[Lemma 6.4.8]{Krantz-SCVbook}\label{lemma-Wpoly-factor}  
Suppose that $P, P_1, \cdots, P_L \in \mathcal P$, and that each of them is monic in $z_{n+1}$. Suppose further that $P$ is a Weierstrass polynomial and that (\ref{polynomial-factors}) holds. Then the factors $P_1, P_2, \cdots P_L$ are Weierstrass polynomials.  
\end{lemma}   
\subsection{Discriminants} \label{subsec-discriminant}
The notion of the discriminant of a polynomial is crucial to many problems in singularity theory; an important case in point being the Jung-Abhyankar theorem stated later in this section. In our analysis, it plays a critical role in determining the regions where the desired monomialization may be achieved after a given change of coordinates. Here we recall the definition of a discriminant, referring the reader to \cite[Ch.\ 7]{Mishra-book} \cite[Ch.\ 3]{CLO-book} for a more detailed treatment of its properties.  
\begin{definition}
Let $P$ and $Q$ be univariate polynomials with complex coefficients 
\begin{align}
\label{def-P} P(t) &= a \prod_{i=1}^{M} (t - p_i) = \sum_{i=0}^{M} a_i t^i, \quad a=a_M \ne 0, \\ \label{def-Q} Q(t) &= b \prod_{j=1}^{N} (t - q_j) = \sum_{j=0}^{N} b_j t^j, \quad b = b_N \ne 0.   
\end{align}   
Then the {\em{resultant}} Res$(P,Q)$ is given by one of the following equivalent formulas  
\begin{align*}
\text{Res}(P,Q) = a^N \prod_{i=1}^{M} Q(p_i) = (-1)^{MN} b^{N} \prod_{j=1}^{N} P(q_i) = a^N b^M \prod_{\begin{subarray}{c} 1 \leq i \leq M \\ 1 \leq j \leq N\end{subarray}} (p_i - q_j).   
\end{align*}
\end{definition}
An alternative description of the resultant \cite{CLO-book} is that it equals the determinant of the $(N+M) \times (N+M)$ Sylvester matrix 
\begin{equation*}
\left(
\begin{matrix}
a_M & a_{M-1} & a_{M-2} & \cdots & a_1 & a_0 & 0 & 0 & \cdots & 0 \\ 
0 & a_M & a_{M-1} & a_{M-2} & \cdots & a_1 & a_0 & 0 & \cdots & 0 \\ 
\vdots & \vdots & \ddots & \ddots & \ddots & \ddots & \ddots & \ddots & \ddots & \vdots \\
0 & 0 & \cdots & 0 & a_M & a_{M-1} & a_{M-2} & \cdots & a_1 & a_0 \\ \\ \hline \\
b_N & b_{N-1} & \cdots & \cdots & b_2 & b_1 & b_0 & 0 & 0 & \cdots \\ 
0 & b_N & b_{N-1} & \cdots & \cdots & b_2 & b_1 & b_0 & 0 & \cdots \\
0 & 0 & b_N & b_{N-1} & \cdots & \cdots & b_2 & b_1 & b_0 & \cdots \\ 
\vdots & \vdots & \ddots & \ddots & \ddots & \ddots & \ddots & \ddots & \ddots & \vdots \\
0 & 0 & \cdots & 0 & b_N & b_{N-1} & \cdots & b_2 & b_1 & b_0  
\end{matrix} 
\right)
\end{equation*}
where the coefficients of $P$ are repeated on $N = \text{deg}(Q)$ rows and the coefficients of $Q$ are repeated on $M = \text{deg}(P)$ rows. This latter characterization combined with the definition shows that Res$(P,Q)$ is a homogeneous polynomial of degree $N+M$ in the coefficients of $P$, $Q$ that vanishes if and only if $P$ and $Q$ have a common root.

\begin{definition} \label{def-discriminant} 
Let $P$ be as in (\ref{def-P}). The {\em{discriminant}} of $P$, denoted $\Delta_P$, is defined to be
\[ \Delta_P = (-1)^{\frac{M(M-1)}{2}} a^{-1} \text{Res}(P,P') = a^{2M-2} \prod_{1 \leq i < j \leq M}(p_i - p_j)^2, \]
where $P'$ denotes the derivative of $P$. 
\end{definition}
In view of the preceding discussion on resultants, we conclude that $a\Delta_P$ is a polynomial in the coefficients of $P$ that vanishes if and only if $P$ and $P'$ share a root, or in other words $P$ has a multiple root.    

Suppose now that $P$ is a multivariate polynomial and that $P \in \mathcal P$, where $\mathcal P$ is as in \S\S \ref{subsec-factor}. Then $\Delta_P$ is a meromorphic function in $\mathbf z$ that vanishes exactly at those values of $\mathbf z$ for which $P$ has multiple roots with respect to $z_{n+1}$. In particular,
 $\Delta_P$ is a holomorphic function in a neighborhood of the origin in $\mathbb C^n$.

\begin{lemma} \label{Disc-product-lemma}
Let $P = P_1 P_2$, where $P, P_1, P_2$ are all monic polynomials in $\mathcal P$. Then $\Delta_{P}$ is divisible by $\Delta_{P_1}$ and $\Delta_{P_2}$.   
\end{lemma} 
\begin{proof}
This is an easy consequence of Definition \ref{def-discriminant}.  
\end{proof} 
A useful result connecting the discriminant with the reducibility of a polynomial is the following:  
\begin{lemma} \cite[p128, Theorem 3]{Shabat-book} \label{disc-red-lemma} 
Let $P \in \mathcal P$. If all the factors in the decomposition $P = P_1 \cdots P_L$ of a polynomial $P$ into irreducible polynomials in $\mathcal P$ are distinct, then the discriminant of $P$ is not identically equal to zero.  
\end{lemma} 
\subsection{Jung-Abhyankar Theorem} \label{subsec-JA}
We conclude this section with a fundamental result of Jung \cite{Jung08} and Abhyankar \cite{Abh55} that constitutes the building block of the resolution of singularities algorithm in \S \ref{sec-resolution}. The Jung-Abhyankar theorem has several algebraic and analytic formulations \cite{Luengo83} \cite{Suss90} \cite{Zurro93} \cite{Par01} \cite{KiyVic04}. The version we use is very similar to the ones proved in \cite{Par01}, \cite{Suss90}, though the statements in these references have to be modified slightly for our applications. In view of the critical role that the Jung-Abhyankar theorem plays in the subsequent steps, we state it below in the form that we need and include a proof in the appendix for completeness.    
\begin{theorem}[Jung-Abhyankar, complex version \cite{Par01}, \cite{Suss90}] \label{thm-JA-complex} 
Let $\mathcal O \subseteq \mathbb C^n$ be an open connected set containing the origin and let 
\[ F(\mathbf z, z_{n+1}) = z_{n+1}^d + \sum_{\nu = 1}^{d} c_{\nu}(\mathbf z) z_{n+1}^{d - \nu} \] 
be a complex-analytic function defined on $\mathcal O \times \mathbb C$ such that $F$ is a polynomial in $z_{n+1}$ with coefficients $c_{\nu}$ that are bounded, complex-analytic functions on $\mathcal O$. Suppose that 
\begin{equation} \Delta_F \not\equiv 0 \quad \text{ and } \quad \Delta_F \text{ is normal crossings on } \mathcal O. \label{disc-norm-cross} \end{equation}
Then there exists an integer $s \geq 1$ with the following property: for any collection of simply connected open sets $\{U_j : 1 \leq j \leq n \}$, $0 \in U_j \subseteq \mathbb C$, such that the power map \begin{equation} \Phi_s : \mathbf z = (z_1, \cdots, z_n) \mapsto (z_1^s, \cdots, z_n^s) \label{power-map} \end{equation} maps $\prod_{j=1}^{n} U_j$ into $\mathcal O$, there exist $d$ complex-analytic functions $\{b_i : 1 \leq i \leq d \}$ defined on $\prod_{j=1}^{n} U_j$ such that 
\[ F \bigl( \Phi_s(\mathbf z), z_{n+1}\bigr) = \prod_{i=1}^{d} \left( z_{n+1} - b_i(\mathbf z) \right), \qquad \mathbf z \in \prod_{j=1}^{n} U_j .  \] 
Moreover for every $i \ne i'$, the difference $b_i - b_{i'}$ is normal crossings on $\prod_{j=1}^{n} U_j$.     
\end{theorem} 
\begin{proof}
See \S \ref{sec-appendix}. 
\end{proof} 
\begin{corollary}[Jung-Abhyankar, real version] \label{thm-JA}
Let $U \subseteq \mathbb R^n$ be an open connected neighborhood of the origin containing $[0,1]^n$, and let \[ F(\mathbf x, x_{n+1}) = x_{n+1}^d + \sum_{\nu=1}^{d} c_{\nu}(\mathbf x) x_{n+1}^{d - \nu} \]
be a real-analytic function defined on $U \times \mathbb R$ such that $F$ is a polynomial in $x_{n+1}$ whose coefficients $c_{\nu}$ are bounded and real-analytic on $U$. Assume that the discriminant $\Delta_{F}(\mathbf x)$ (considering $F$ as a polynomial in $x_{n+1}$) is $\not\equiv 0$ and is normal crossings on $U$. Then there exist 
\begin{itemize}
\item a positive integer $s$, 
\item small positive constants $\epsilon_j$, $1 \leq j \leq n$, and 
\item real-analytic functions $\{ r_i : 1 \leq i \leq d \}$ 
on $\prod_{j=1}^{n}(-\epsilon_j, 1 + \epsilon_j)$ 
\end{itemize}
such that the power map $\Phi_s$ defined in (\ref{power-map}) maps $\prod_{j=1}^{n} (- \epsilon_j, 1 + \epsilon_j)$ into $U$ and $F$ admits the factorization    
\begin{equation*} F(\Phi_s(\mathbf y), x_{n+1}) = \prod_{i=1}^{d} (x_{n+1} - r_i(\mathbf y)), \quad (\mathbf y, x_{n+1}) \in \Bigl[ \prod_{j=1}^{n} (- \epsilon_j, 1 + \epsilon_j) \Bigr] \times \mathbb R.  \end{equation*} 
Moreover, for $i \ne i'$, the difference $r_i - r_{i'}$ is normal crossings on $\prod_{j=1}^{n}(-\epsilon_j, 1 + \epsilon_j)$.  
\end{corollary} 
\begin{proof}
Since $F$ is a real-analytic Weierstrass polynomial on $U \times \mathbb R$, there is an open connected set $\mathcal O \subseteq \mathbb C^n$ containing the origin such that $F$ admits a holomorphic extension to $\mathcal O \times \mathbb C$ and $\Delta_F$ continues to be normal crossings on $\mathcal O$. Let $s$ be as in Theorem \ref{thm-JA-complex}. We can then choose positive constants $\epsilon_j, \delta_j$ sufficiently small such that $\Phi_s$ maps $\prod_{j=1}^{n} U_j$ into $\mathcal O$ where 
\[ U_j = \{z \in \mathbb C : - \epsilon_j < \text{Re}(z) < 1 + \epsilon_j, -\delta_j < \text{Im}(z) < \delta_j \}, \quad 1 \leq j \leq n \]
is a simply connected open neighborhood of the origin in $\mathbb C$. Applying Theorem \ref{thm-JA-complex} with this choice of $U_j$ and restricting the resulting factorization back to $\mathbb R^{n+1}$, the conclusions of the corollary follow.    
\end{proof} 
We will repeatedly use the following consequence of the Jung-Abhyankar theorem. 
\begin{lemma} \label{JA-cor}
Let $U$ be as in the statement of Corollary \ref{thm-JA}. Let $G(\mathbf x, x_{n+1})$ and $P(\mathbf x, x_{n+1})$ be real-analytic functions on $U \times \mathbb R$ admitting the factorizations \[ G = \prod_{\ell=1}^{L} G_{\ell}^{m_{\ell}} \quad \text{ and } \quad P = \prod_{\ell=1}^{L} G_{\ell} \quad \text{ for } (\mathbf x, x_{n+1}) \in U \times \mathbb R, \] where $\mathbf m = (m_1, \cdots, m_n) \in \mathbb N^n$, and $\{G_{\ell} : 1 \leq \ell \leq L \}$ is a collection of distinct Weierstrass polynomials in $x_{n+1}$ with bounded real-analytic coefficients on $U$ with the property that \[ \Delta_{P}(\mathbf x) \not\equiv 0, \quad \text{ and } \quad c(\mathbf x) = \prod_{\ell=1}^L \left[G_{\ell}(\mathbf x, 0)\right]^{m_{\ell}} \not\equiv 0. \] Assume further that $\Delta_P$ and $c$ are both normal crossings on $U$. Then there exists 
\begin{itemize} 
\item a positive integer $s$ and 
\item roots $\{r_i : 1 \leq i \leq d \}$ that are real-analytic on an open neighborhood of the origin in $\mathbb R^n$ containing $[0,1]^n$ 
\end{itemize} 
such that   
\begin{equation} G(\Phi_s(\mathbf y), x_{n+1}) = \prod_{i=1}^{d} \left(x_{n+1} - r_i(\mathbf y) \right) \quad \text{ for } \quad (\mathbf y, x_{n+1}) \in  (0,1)^n \times \mathbb R.  \label{factor-id} \end{equation} 
The roots $r_i$ are normal crossings on $(0,1)^n$. The same conclusion holds for all the differences $\{r_i - r_{i'} : i \ne i' \}$ that are not identically zero. In particular, this means that the roots of $G$ and their differences are either identically zero of fractional normal crossings on $(0,1)^n$.
\end{lemma} 

\begin{proof} 
In view of our discussion on discriminants in \S \ref{subsec-discriminant} and Lemma \ref{Disc-product-lemma}, each $\Delta_{G_{\ell}}$ is a real-analytic function on $U$ that divides $\Delta_P$. Further $G_{\ell}$ is a Weierstrass polynomial, hence $\Delta_{G_{\ell}}$ vanishes at the origin. Since $\Delta_P$ is normal crossings on $U$ by assumption, so is $\Delta_{G_{\ell}}$. Let $d_{\ell}$ and $d$ denote the degree of $G_{\ell}$ and $G$ respectively, so that $d_{\ell} \leq d$ and $\sum_{\ell=1}^L d_{\ell} m_{\ell} = d$. By Corollary \ref{thm-JA} applied to $G_{\ell}$, there exist a positive integer $s_\ell$, small positive constants $\epsilon_j$ (independent of $\ell$ by shrinking if necessary) and real-analytic functions $\{ r_{i,\ell} : 1 \leq i \leq d_{\ell} \}$ on $\prod_{j=1}^{n}(-\epsilon_j, 1+\epsilon_j)$ such that  
\begin{equation*} G_{\ell}(\Phi_{s_\ell}(\mathbf y), x_{n+1}) = \prod_{i=1}^{d_{\ell}} \left(x_{n+1} - r_{i, \ell}(\mathbf y) \right) \quad \text{ for } \quad (\mathbf y, x_{n+1}) \in \bigl( \prod_{j=1}^{n}(- \epsilon_j, 1+\epsilon_j) \bigr) \times \mathbb R. \end{equation*} 
Since $\Delta_P \not\equiv 0$, we conclude that $r_{i,\ell} \not\equiv r_{i', \ell'}$ for $(i, \ell) \ne (i', \ell')$. Let $s$ be the lowest common multiple of the integers $\{s_\ell : 1 \leq \ell \leq L \}$. Set  
\[ f_{i, \ell}(\mathbf y) := r_{i, \ell} \circ \Phi_{\frac{s}{s_{\ell}}}(\mathbf y),   \]
so that each $f_{i, \ell}$ is real-analytic on an open set containing $[0,1]^n$, and the factorization \[ G_{\ell}(\Phi_s(\mathbf y), x_{n+1}) = G_{\ell}(\Phi_{s_\ell} \circ \Phi_{s/s_{\ell}}(\mathbf y), x_{n+1}) = \prod_{i=1}^{n} \bigl(x_{n+1} - f_{i, \ell}(\mathbf y) \bigr) \] holds for $\Phi_{s/s_{\ell}} (\mathbf y) \in \prod_{j=1}^{n} (-\epsilon_j, 1+\epsilon_j)$ and therefore certainly for $\mathbf y \in [0,1]^n$. Constructing a collection of functions $\{r_i : 1 \leq i \leq d \}$ in which each $f_{i, \ell}$ is repeated exactly $m_{\ell}$ times immediately leads to the factorization (\ref{factor-id}).   
       
Since the functions $c \circ \Phi_s$ and $\Delta_P \circ \Phi_s$ are both normal crossings on $(0,1)^n$ by hypothesis, with 
\begin{align*} 
c \circ \Phi_s (\mathbf y) &= (-1)^d\prod_{\ell=1}^{L} \prod_{i=1}^{d_{\ell}} \left[f_{i, \ell} (\mathbf y) \right]^{m_{\ell}} = (-1)^d \prod_{i=1}^{d} r_i(\mathbf y) \not\equiv 0,\; \text{ and }\\ \Delta_P \circ \Phi_s (\mathbf y) &= \prod_{(i, \ell) \ne (i', \ell')} \bigl((f_{i, \ell} - f_{i', \ell'})(\mathbf y) \bigr)^2 \not\equiv 0,
\end{align*} 
and the individual factors $r_i$ and $f_{i, \ell} - f_{i',\ell'}$ have just been shown to be nontrivial real-analytic functions on $[0,1]^n$, we conclude that each of these factors is also normal crossings on $(0,1)^n$, as claimed.  
\end{proof}
It is important to observe that the assumption (\ref{disc-norm-cross}) of a global normal crossings discriminant is an extremely restrictive one and not satisfied for generic analytic functions $F$. Nonetheless, it is possible to decompose an open neighborhood of the origin in $\mathbb R^n$ into subsets such that the normal crossings structure of the discriminant $\Delta_F$ is achieved on each of these subsets and in certain carefully chosen coordinates that may vary from one subset to the next. This approach to analyzing the structure of an analytic set (in the absence of the assumption that $\Delta_F$ is normal crossings) has been studied extensively in algebraic geometry literature, notably in \cite{BieMil88,BieMil90,Suss90,Par01}. This is also the primary focus of this paper and we expand on the theme in \S \ref{sec-resolution}.

\section{Newton polyhedron, distance and exponent} \label{sec-newton} 
\subsection{Newton data} As indicated in \S \ref{sec-intro}, the notions of {\atxt Newton polyhedron, distance and exponent} have played a critical role in problems involving the zero set of a bivariate real-analytic function. We recall these concepts and introduce certain variants that will be useful in our study of real-analytic functions in higher dimensions.  
\begin{definition} 
\begin{enumerate}[1.]
\item Given a fractional power series $F$ in $(n+1)$ variables vanishing at the origin,
\begin{equation} F(\mathbf x, x_{n+1}) = \sum_{\begin{subarray}{c} (\pmb{\kappa}, \kappa_{n+1}) > (\mathbf 0,0) \\ (\pmb{\kappa}, \kappa_{n+1}) \in \mathbb Q^{n+1} \end{subarray}} a_{\pmb{\kappa}, \kappa_{n+1}} \mathbf x^{\pmb{\kappa}} x_{n+1}^{\kappa_{n+1}}, \label{fractpow-n+1} \end{equation} 
 we define its {\em{Newton polyhedron}} to be \begin{equation*} 
\text{NP}(F) := \text{ convex hull} \bigcup_{\begin{subarray}{c} (\pmb{\kappa}, \kappa_{n+1}) > (\mathbf 0,0) \\ a_{{\pmb{\kappa}}, \kappa_{n+1}} \ne 0 \end{subarray}} \left[(\pmb{\kappa}, \kappa_{n+1}) + \mathbb R^{n+1}_{\geq 0} \right],
\end{equation*}
where $\mathbb R^{n+1}_{\geq 0}$ denotes the closure of the positive orthant in $\mathbb R^{n+1}$.
\medskip

\item The {\em{Newton distance}} of $F$ is defined to be 
\[ d_0(F) = \min \left\{ d : (d, \cdots, d) \in \text{NP}(F) \right\}. \]
\medskip
\item The \emph{Newton exponent} of $F$ is $\delta_0(F):= d(F)^{-1}$. 
\end{enumerate} 
\end{definition} 
{\atxt As discussed below, these concepts are coordinate-dependent, and what is denoted by $d_0(F)$, resp., $\delta_0(F)$, above would later be denoted by $d_0(F;I)$, $\delta_0(F;I)$, resp., where $I$ is the identity map. We also point out that the above definition of Newton distance is the same as the one used in
\cite{Var76},\cite{Grblt04}-\cite{Grblt11}, while what we call the Newton exponent is called the Newton distance in
 \cite{PhSt97,PhSteStu99}. }

The proof of the following easy observation is left to the reader. 
\begin{lemma} \label{lemma-prelim1}
The Newton polyhedron is closed under addition by non-negative vectors, i.e., if $(\pmb{\kappa}_0, \nu_0) \in \text{NP}(F)$, then $(\pmb{\kappa}_0, \nu_0) + (\pmb{\tau}, \tau_{n+1}) \in \text{NP}(F)$ for any $(\pmb{\tau}, \tau_{n+1}) \geq (\mathbf 0, 0)$. 
\end{lemma}

The Newton {\atxt exponent} of a fractional power series provides, in any dimension, an upper bound for the critical integrability index. 
\begin{theorem}\label{prop-cii-nd-ndim}
If $F$ is a fractional power series vanishing at the origin, then $\mu_0(F) \leq \delta_0(F)$. 
\end{theorem}  
\begin{proof}
See \S \ref{sec-appendix2}. 
\end{proof} 

An attractive feature of the Newton polyhedron, distance and exponent is that they make no reference to the roots of $F$ and can be computed based on the power series expansion of $F$ alone, in particular without appealing to any resolution of singularities algorithm. On the other hand both these quantities 
depend on the choice of the ambient coordinate system $(\mathbf x, x_{n+1})$ in which $F$ is expressed. Suppose that $F \mapsto \delta(F)$ is a mapping that produces for every real-analytic function $F$ a scalar $\delta(F)$ that depends on the zero variety of $F$ and is invariant under analytic coordinate changes  (for example, $\delta(F)$ could be the critical integrability index or the oscillation index). If $\delta(F)$ can be expressed in terms of geometric data contained in the Newton polyhedron, then such a characterization renders $\delta(F)$ computationally amenable, bypassing the numerical complexities of resolving singularities. However any such characterization can only hold provided that $F$ is in suitable coordinates. In such an event, a complete description of $\delta(F)$ must necessarily involve a characterization of the good choices of coordinates. 

In the sequel, we will often need to keep track of the Newton {\atxt data} for different coordinate choices. Suppose that $U, V$ are open subsets of $\mathbb R^{n+1}$ and $\Phi : U \rightarrow V$ a coordinate transformation such that $(\mathbf x, x_{n+1}) = \Phi(\mathbf y, y_{n+1})$. If $F$, originally a fractional power series in the variables $(\mathbf x, x_{n+1})$, can also be represented as a fractional power series in the variables $(\mathbf y, y_{n+1})$, then the Newton polyhedron, {\atxt distance  and  exponent} of $F$ in the new coordinates $(\mathbf y, y_{n+1})$ will be denoted by NP$(F;\Phi)$, $d_0(F;\Phi)$ and $\delta_0(F; \Phi)$ respectively. 
\begin{definition}
\begin{enumerate}[1.]
\item For $ 1 \leq j \leq n$, let $\pi_j :  \mathbb R^{n+1} \rightarrow \mathbb R^2$ denote the projection 
\[ \pi_j(\mathbf x, x_{n+1}) = (x_j, x_{n+1}). \]
For $F$ as in (\ref{fractpow-n+1}), the {\em{$j$th projected Newton {\atxt exponent}}} of $F$ is defined as
\[ \delta_j = \max \{t : (t^{-1}, t^{-1}) \in \pi_j \left( \text{NP}(F) \right) \}. \]
\item The {\em{generalized Newton {\atxt exponent}}} is defined to be the smallest of the $n$ projected Newton exponents.   
\end{enumerate} 
\end{definition} 

\subsection{Newton polyhedra with monotone edge paths}
We now describe a class of $(n+1)$-dimensional Newton polyhedra of a very special structure but which will be central to our analysis. 
\begin{definition}Let $\{(\pmb{\mu}_i, \nu_i) : 1 \leq i \leq m \} \subseteq \mathbb R^{n+1}_{\geq 0}$ be a finite collection of points obeying
\begin{equation}  \pmb{\mu}_i < \pmb{\mu}_{i+1} \; \text{ and } \; \nu_i > \nu_{i+1} \; \text{ for all } 1 \leq i \leq m-1. \label{order} \end{equation}  
Then the connected chain of line segments 
\begin{equation} \Gamma = \bigcup_{i=1}^{m-1} \bigl\{t (\pmb{\mu}_i, \nu_i) + (1 - t)(\pmb{\mu}_{i+1}, \nu_{i+1}): 0 \leq t \leq 1   \bigr\}  \label{mep} \end{equation} 
is called a {\em{monotone edge path}} in $\mathbb R^{n+1}$. 
\end{definition}
It is easy to see that the points on a monotone edge path satisfy an ordering relation. Namely, 
\begin{lemma}
If $(\pmb{\kappa}, \kappa_{n+1}), (\pmb{\gamma}, \gamma_{n+1})$ are two distinct points in $\Gamma$, then one of the following relations must hold:   
\begin{equation} \text{either } \pmb{\kappa} < \pmb{\gamma},\; \kappa_{n+1} > \gamma_{n+1}, \quad \text{ or } \quad  \pmb{\kappa} > \pmb{\gamma},\; \kappa_{n+1} < \gamma_{n+1}. \label{order-mep} \end{equation} 
\end{lemma} 
\begin{definition} \label{def-poly-mep}
We will say that an unbounded convex polyhedron $\mathfrak C \subseteq \mathbb R^{n+1}_{\geq 0}$ is defined by a monotone edge path $\Gamma \subseteq \mathbb R^{n+1}_{\geq 0}$ given by (\ref{mep}) if 
\begin{equation} \mathfrak C = \bigcup_{(\pmb{\kappa}, \kappa_{n+1}) \in \Gamma} \left[ (\pmb{\kappa}, \kappa_{n+1}) + \mathbb R^{n+1}_{\geq 0} \right]. \label{conv-poly-mep} \end{equation}
\end{definition} 
We record a few elementary facts concerning (Newton) polyhedra defined by monotone edge paths.
\begin{lemma} \label{mep-lemma1} 
If $\mathfrak C$ is a convex polyhedron defined by the monotone edge path $\Gamma$ in (\ref{mep}), then
\[ \mathfrak C = \text{ convex hull } \bigcup_{i=1}^{m} \left\{(\pmb{\mu}_i, \nu_i) + \mathbb R^{n+1}_{\geq 0}: 1 \leq i \leq m \right\}. \]   
\end{lemma} 
\begin{proof}
Since $\mathfrak C$ is convex and contains the translated orthants $(\pmb{\mu}_i, \nu_i) + \mathbb R^{n+1}_{\geq 0}$ for each $1 \leq i \leq m$, the convex hull of the union of these sets is trivially contained in $\mathfrak C$. For the reverse inclusion, let us pick $(\pmb{\gamma}, \gamma_{n+1}) \in \mathfrak C$. It follows from (\ref{mep}) and (\ref{conv-poly-mep}) that there exist $(\pmb{\kappa}, \kappa_{n+1}) = t (\pmb{\mu}_i, \nu_i) + (1 - t) (\pmb{\mu}_{i+1}, \nu_{i+1}) \in \Gamma$ and $(\pmb{\tau}, \tau_{n+1}) \geq (\mathbf 0, 0)$ such that \begin{align*} (\pmb{\gamma}, \gamma_{n+1}) &= (\pmb{\kappa}, \kappa_{n+1}) + (\pmb{\tau}, \tau_{n+1}) \\ &= t (\pmb{\mu}_i + \pmb{\tau}, \nu_i + \tau_{n+1}) + (1-t) (\pmb{\mu}_{i+1} + \pmb{\tau}, \nu_{i+1} + \tau_{n+1}). \end{align*}      
But the last expression is a convex combination of two points lying in $(\pmb{\mu}_i, \nu_i) + \mathbb R^{n+1}_{\geq 0}$ and $(\pmb{\mu}_{i+1}, \nu_{i+1}) + \mathbb R^{n+1}_{\geq 0}$ respectively, and hence $(\pmb{\gamma}, \gamma_{n+1})$ is in the desired convex hull.  
\end{proof} 
\begin{lemma} \label{mep-lemma2} Suppose that a convex polyhedron $\mathfrak C$ is defined by a monotone edge path. Then the intersection of $\mathfrak C$ with any horizontal hyperplane $\{ \kappa_{n+1} = c\}$ will either be empty or a translate of the $n$-dimensional horizontal positive  orthant, i.e., of the form $(\pmb{\mu}(c), c) + (\mathbb R^n_{\geq 0}, 0)$. 
\end{lemma} 
 \begin{proof}
Let us assume that $\mathfrak C$ is defined by the monotone edge path $\Gamma$ in (\ref{mep}). Since the smallest $\kappa_{n+1}$ coordinate occurring in $\Gamma$ is $\nu_m$, it follows from the definition (\ref{conv-poly-mep}) of $\mathfrak C$ that $\kappa_{n+1} \geq \nu_{m}$ for any $(\pmb{\kappa}, \kappa_{n+1}) \in \mathfrak C$. Thus the intersection of $\mathfrak C$ with the hyperplane $\{ \kappa_{n+1} = c \}$ is empty if $c < \nu_m$. On the other hand, the point in $\Gamma$ with the largest $\kappa_{n+1}$ coordinate, hence by (\ref{order}) the smallest $\pmb{\kappa}$ coordinate,  is $(\pmb{\mu}_1, \nu_1)$. So for any $c \geq \nu_1$, the intersection of $\mathfrak C$ with $\{\kappa_{n+1} = c \}$ is of the form $(\pmb{\mu}_1 + \mathbb R^n_{\geq 0}, c)$.       

It remains to prove the statement of the lemma when $\nu_m \leq c < \nu_1$. Given such $c$ and because of the monotonicity of the edge path, there is a unique index $1 \leq i \leq m-1$ and and a unique scalar $t \in [0,1)$ such that $c = t \nu_i + (1-t) \nu_{i+1}$. We define $\pmb{\mu}(c) = t \pmb{\mu}_i + (1-t) \pmb{\mu}_{i+1}$. Then $(\pmb{\mu}(c), c)$ clearly lies in $\Gamma$, and it follows from (\ref{conv-poly-mep}) that \[(\pmb{\mu}(c), c) + (\mathbb R^n_{\geq 0}, 0) \subseteq \mathfrak C \cap \{\kappa_{n+1} = c \}.\] To prove the reverse containment, let $(\pmb{\kappa}, c) \in \mathfrak C$. By (\ref{conv-poly-mep}) there exists $(\pmb{\gamma}, \gamma_{n+1}) \in \Gamma$ such that $(\pmb{\kappa}, c) \in (\pmb{\gamma}, \gamma_{n+1}) + \mathbb R^{n+1}_{\geq 0}$. In particular $c \geq \gamma_{n+1}$. Since $(\pmb{\mu}(c), c)$ and $(\pmb{\gamma}, \gamma_{n+1})$ both lie in $\Gamma$, it follows from the ordering property (\ref{order-mep}) that $\pmb{\mu}(c) \leq \pmb{\gamma} \leq \pmb{\kappa}$. But this implies that $(\pmb{\kappa}, c) \in (\pmb{\mu}(c), c) + (\mathbb R^{n}_{\geq 0}, 0)$, completing the proof.    
\end{proof} 

The following lemma identifies a class of functions whose Newton polyhedra have the special structure described in Definition \ref{def-poly-mep}. 
\begin{lemma} \label{mep-lemma3} 
Suppose that $F$ is a fractional power series of the form 
\begin{equation} F(\mathbf x, x_{n+1}) = \mathbf x^{\pmb{\beta}} x_{n+1}^{\beta_{n+1}} \prod_{i=1}^{M} \left(x_{n+1} - u_i(\mathbf x) \mathbf x^{\pmb{\gamma}_i} \right), \label{formula-F} \end{equation} 
where the set of exponents $\{ \pmb{\gamma}_i : 1 \leq i \leq M \}$ is totally ordered with $\pmb{\gamma}_i \leq \pmb{\gamma}_{i+1}$, and the functions $\{u_i : 1 \leq i \leq M \}$ are units. Then NP$(F)$ is defined by a monotone edge path.   
\end{lemma} 
\begin{proof}
Let $\{ \pmb{\alpha}_{\ell} : 1 \leq \ell \leq L \}$ be the collection of distinct multi-exponents occurring in $\{ \pmb{\gamma}_i : 1 \leq i \leq M \}$, arranged so that
\[ \pmb{\alpha}_1 < \pmb{\alpha}_2 < \cdots < \pmb{\alpha}_{L}. \]
For $0 \leq \ell \leq L$, we define  
\begin{equation} \begin{aligned} \mathcal L_{\ell} &= \bigl\{1 \leq i \leq M : \pmb{\gamma}_i = \pmb{\alpha}_{\ell} \bigr\},  \qquad m_{\ell} = \#(\mathcal L_{\ell}) \text{ so that } \sum_{\ell=1}^{L} m_{\ell} = M,  \\ \mathbf A_{\ell} &= \pmb{\beta} + \sum_{k \leq \ell} m_k \pmb{\alpha}_k,  \hskip0.83in B_{\ell} = \beta_{n+1} + \sum_{k > \ell} m_k. \end{aligned} \label{AlBl-old} \end{equation}
Multiplying out the factors in (\ref{formula-F}) and expanding $F$ as a Puiseux expansion 
\begin{equation} \label{Puiseux-exp} 
\begin{aligned}  F(\mathbf x, x_{n+1}) &= \mathbf x^{\pmb{\beta}} x_{n+1}^{\beta_{n+1}} \sum_{r=0}^{M} x_{n+1}^{M-r} (-1)^r \sum_{\begin{subarray}{c}  \mathbf I = (i_1, \cdots, i_r) \\ 1 \leq i_1  < \cdots < i_r \leq M \end{subarray}} \prod_{j=1}^{r} u_{i_j}(\mathbf x) \mathbf x^{\pmb{\gamma}_{i_j}} \\  &= \sum_{(\pmb{\kappa}, \kappa_{n+1}) > (\mathbf 0, 0)} a_{\pmb{\kappa}, \kappa_{n+1}} \mathbf x^{\pmb{\kappa}} x_{n+1}^{\kappa_{n+1}}, \end{aligned} 
\end{equation} 
we note the following two points:
\begin{enumerate}[1.]
\item The coefficient of $\mathbf x^{\mathbf A_{\ell}}x_{n+1}^{B_{\ell}}$ in the expansion is 
\[ \prod_{\begin{subarray}{c} i \in \mathcal L_{k} \\ k \leq \ell   \end{subarray}} \bigl[- u_i(\mathbf 0) \bigr] \ne 0.  \]  
Thus the points $\{(\mathbf A_{\ell}, B_{\ell}) : 0 \leq \ell \leq L \}$ occur in the Newton polyhedron of $F$. \item For every $(\pmb{\kappa}, \kappa_{n+1})$ with $a_{\pmb{\kappa}, \kappa_{n+1}} \ne 0$, there exists $0 \leq \ell \leq L-1$ such that \[ (\pmb{\kappa}, \kappa_{n+1}) \geq (A, B) \] for some point $(A, B)$ lying on the line segment connecting $(\mathbf A_{\ell}, B_{\ell})$ and $(\mathbf A_{\ell+1}, B_{\ell+1})$.
In order to verify this claim, it suffices to show that the points \begin{equation} \Bigl\{ \bigl(\pmb{\beta} + \pmb{\gamma}_1 + \cdots + \pmb{\gamma}_r, M - r + \pmb{\beta}_{n+1} \bigr): 0 \leq r \leq M \Bigr\} \label{edge-points} \end{equation}  lie on the segment joining $(\mathbf A_{\ell}, B_{\ell})$ and $(\mathbf A_{\ell+1}, B_{\ell+1})$ for some $\ell$. Any other $(\pmb{\kappa}, \kappa_{n+1})$ occurring in the expansion (\ref{Puiseux-exp}) has to be larger than some exponent in (\ref{edge-points}).  

Fix $r$. In view of the total ordering of the indices $\{ \pmb{\gamma}_i : 1 \leq i \leq M \}$, there is a unique index $0 \leq \ell = \ell(r) \leq L$ such that \[ \mathbf A_{\ell} \leq \pmb{\beta}+ \pmb{\gamma}_1 + \cdots + \pmb{\gamma}_r < \mathbf A_{\ell+1}. \] This implies that 
\begin{align*}  \pmb{\beta} + \pmb{\gamma}_1 + \cdots + \pmb{\gamma}_r &= \mathbf A_{\ell} + t \pmb{\alpha}_{\ell+1}, \\ M - r  + \beta_{n+1} &= M - \Bigl( \sum_{k \leq \ell} m_k + t \Bigr) + \beta_{n+1} = B_{\ell} - t    \end{align*}
for some non-negative integer $t < m_{\ell + 1}$, verifying that the point indeed lies on the stated line segment.  
\end{enumerate} 
Combining the two observations above, we obtain 
\[ \text{NP}(F) = \text{ convex hull} \bigcup_{\ell=1}^{L} \left[(\mathbf A_{\ell}, B_{\ell}) + \mathbb R^{n+1}_{\geq 0} \right]. \]
Since the connected chain of line segments joining $\{(\mathbf A_{\ell}, B_{\ell}) : 0 \leq \ell \leq L \}$ is a monotone edge path, the desired conclusion follows from Lemma \ref{mep-lemma1}.     
\end{proof} 
{\bf{Remark: }} The quantities introduced in (\ref{AlBl-old}) will return in \S \ref{sec-resolution} to play a vital role in one of the main steps in the resolution algorithm, namely Proposition \ref{lemma-block2}. 
\begin{lemma} \label{lemma-nd=gen}
If $F$ is a fractional power series in $(n+1)$ variables whose Newton polyhedron is defined by a monotone edge path, then the Newton {\atxt exponent} of $F$ equals the generalized Newton {\atxt exponent} . 
\end{lemma} 
\begin{proof}
Let $j_1, j_2, \cdots, j_n$ be a permutation of $\{ 1, 2, \cdots, n \}$ such that the projected Newton {\atxt exponents}  of $F$ obey 
\begin{equation} \delta_{j_1} \geq \delta_{j_2} \geq \cdots \geq \delta_{j_n}. \label{deltaj-ordering}  \end{equation} 
Thus the generalized Newton {\atxt exponent}  of $F$ is $\delta_{j_n}$. Let $\delta_0$ denote the standard Newton {\atxt exponent} . Since $(\delta_0^{-1}, \cdots, \delta_0^{-1}) \in \text{NP}(F)$, it  is clear that $(\delta_0^{-1}, \delta_0^{-1}) \in \pi_j(NP(F))$, and hence $\delta_j \geq \delta_0$ for every $1 \leq j \leq n$. In particular $\delta_{j_n} \geq \delta_0$. 

We now turn to the converse inequality. Since $(\delta_j^{-1}, \delta_{j}^{-1}) \in \pi_j(\text{NP}(F))$ for all $1 \leq j \leq n$, there exist scalars $\{ a_{ij} \}$ such that the points 
\[ (\delta_1^{-1}, a_{12}, \cdots, a_{1n}, \delta_1^{-1}), \; (a_{21}, \delta_2^{-1}, a_{23}, \cdots, a_{2n}, \delta_2^{-1}), \cdots, (a_{n1}, \cdots, a_{n,n-1}, \delta_n^{-1}, \delta_n^{-1}) \]
all lie in NP$(F)$. By Lemma \ref{lemma-prelim1}, a vector that is larger than or equal to any of the multi-indices above (in the sense of (\ref{order1})) also lies in NP$(F)$, therefore the ordering in (\ref{deltaj-ordering}) implies that all the points 
\begin{equation} (\delta_{j_n}^{-1}, a_{12}, \cdots, a_{1n}, \delta_{j_n}^{-1}), \; \cdots, \; (a_{n1}, \cdots, a_{n,n-1}, \delta_{j_n}^{-1}, \delta_{j_n}^{-1}) \label{points} \end{equation}
lie in NP$(F)$, and in fact lie in the intersection of NP$(F)$ with the hyperplane $\{ \kappa_{n+1} =\delta_{j_n}^{-1} \}$. But NP$(F)$ is defined by a monotone edge path by hypothesis, therefore by Lemma \ref{mep-lemma2}, there exists $\pmb{\kappa}_0$ such that 
\begin{equation} \text{NP}(F) \cap \{\kappa_{n+1} = \delta_{j_n}^{-1} \} = (\pmb{\kappa}_0, \delta_{j_n}^{-1}) + (\mathbb R^{n}_{\geq 0},0). \label{int-set} 
\end{equation} Comparing each of the points in (\ref{points}) above with $(\pmb{\kappa}_0, \delta_{j_n}^{-1})$, we conclude that $(\pmb{\kappa}_0, \delta_{j_n}^{-1}) \leq (\delta_{j_n}^{-1}, \cdots, \delta_{j_n}^{-1})$.
In view of (\ref{int-set}), this implies that $(\delta_{j_n}^{-1}, \cdots, \delta_{j_n}^{-1}) \in \text{NP}(F) \cap \{\kappa_{n+1} = \delta_{j_n}^{-1} \} \subseteq \text{NP}(F)$. By definition of the Newton {\atxt exponent} , $\delta_0 \geq \delta_{j_n}$, completing the proof.         
\end{proof}

\section{Domains and coordinates} \label{sec-resolution-preliminaries}
We now proceed to collect the definitions needed for desingularization or rectilinearization of the roots of an arbitrary real-analytic function. Our goal in this section is to describe the structure of the open subsets and provide an explicit construction of the coordinate transformations that occur in this recursive process. 

\subsection{Coordinate transformations} \label{subsec-coords}The singularities of $F$ will be resolved one orthant  at a time, so that each of the sets in the final decomposition will be a subset of an orthant. Without loss of generality, we henceforth restrict attention to the positive orthant $\mathbb R^n_{>0} = \{\mathbf x : x_j > 0, 1 \leq j \leq n \}$. 
\begin{definition} \label{def-coord-tf} 
Let $V$ be an open subset of the positive orthant whose closure contains the origin. A {\em{generalized coordinate transformation}} or a {\em{generalized coordinate system}} on $V$ is a vector-valued function $\sigma$ such that 
\begin{itemize}
\item each entry of $\sigma(\mathbf y)$ is a fractional power series in $\mathbf y$ on $(0,1)^n$, 
\item $\sigma(\mathbf 0) = \mathbf 0$, $\sigma \in C^1((0,1)^n)$,  
\item $\sigma$ is a bijection from $(0,1)^n$ onto $V$.  
\end{itemize} 
\end{definition}
In view of the discussion in \S\S \ref{subsec-defns}, $\sigma$ is well-defined as a continuous function on an open neighborhood of the origin containing $\overline{V}$, so $\sigma(\mathbf 0)$ can be defined uniquely. In particular the Jacobian of such a transformation is integrable on $[-1,1]^n$. 

\begin{definition} \label{def-coord-tf-2} 
\begin{enumerate}[1.]
\item A generalized coordinate transformation $\sigma$ on $V$ will be called a {\em{coordinate transformation or a system of coordinates on $V$}} if its Jacobian is a unit on $(0,1)^n$. 
\item If $\sigma$ is a (generalized) coordinate transformation and $W \subseteq (0,1)^n$, then $\sigma(W) \subseteq V$ is said to be a {\em{ (generalized) coordinate image }} of $W$.   
\end{enumerate} 
\end{definition}  

\subsubsection{Examples} \label{examples} 
\begin{enumerate}[1.]
\item Let $\sigma$ be a function such that after possibly a permutation of variables 
\[ \sigma(y_1, \cdots, y_n) = (y_1, \cdots, y_{n-1}, y_n u(\mathbf y')), \]
where $u$ is a unit on $(0,1)^n$. Then $\det(D \sigma)(\mathbf y) = u(\mathbf y')$. A coordinate transformation of this type will be referred to as {\em{scaling by a unit.}} 
\item A coordinate transformation $\sigma$ is called {\em{a shift}} if, after possibly a permutation of variables
\[ \sigma(\mathbf y) = (y_1, \cdots, y_{n-1}, y_n - f(\mathbf y')), \]
where $f$ is a fractional power series. Here $\det(D \sigma) \equiv 1$.  
\item Fix an index $1 \leq k \leq n-2$. Let $\sigma$ be a function such that after possibly a permutation of variables \[ \sigma(\mathbf y) = (y_1, \cdots, y_k, y_{k+1}y_n, \cdots, y_{n-1}y_n, y_n). \]
Then $\det(D \sigma(\mathbf y)) = y_n^{n-k-1}$, so that $\sigma$ is a generalized coordinate transformation. A transformation of this form will be referred to as a {\em{blow-down}}, in keeping with the standard nomenclature of these maps in the algebraic geometry and microlocal analysis literature \cite{Shafarevich-book}. 
\item A mapping $\Phi_{\mathbf r}:(0,1)^n \rightarrow (0,1)^n$ will be called {\em{a power transformation or a power change of coordinates}} if there exists an $n$-dimensional vector $\mathbf r$ with positive rational entries such that 
\begin{equation} \Phi_{\mathbf r}(\mathbf y) = (y_1^{r_1}, \cdots, y_n^{r_n}), \quad \text{ so that } \quad D \Phi_{\mathbf r}(\mathbf y) = \prod_{j=1}^{n} r_j y_j^{r_j-1}. \label{power-tf} \end{equation}
For a vector $\mathbf r$ with $r_1 = r_2 =\cdots = r_n = r > 0$, we write $\Phi_{\mathbf r} = \Phi_r$ by a slight abuse of notation, but consistent with (\ref{power-map}). 

Note that even though generalized blow-downs and power transformations need not be coordinate transformations individually, it is easy to construct a composition of these maps that is. 
\end{enumerate} 
We will refer to the examples listed above as {\em{elementary transformations}}. All the generalized coordinate transformations considered in this paper will be constructed as compositions of these. 

\begin{definition} 
Let $V$ be as in Definition \ref{def-coord-tf}. A (generalized) coordinate transformation $\sigma:(0,1)^n \rightarrow V$ is said to be {\em{admissible}} if it is a finite composition of the elementary transformations.  
\end{definition} 

\subsubsection{Remarks} \label{remarks} 
In view of the discussion in \S \ref{sec-preliminaries} on the non-closure of the class of fractional power series, the composition of two arbitrary generalized coordinate systems need not be well-defined. Even if defined it need not be a generalized coordinate system. The following two results provide situations in which composition is meaningful. 
\begin{lemma} \label{lemma-gencoords-comp}
Let $\sigma : (0,1)^n \rightarrow V$ and $\varphi : (0,1)^n \rightarrow W \subseteq (0,1)^n$ be two generalized coordinate systems.
\begin{enumerate}[(a)]
\item If $\varphi = \Phi_{\mathbf r}$ for any $\mathbf r$ with positive entries, then $\sigma \circ \varphi$ is a generalized coordinate system.
\item If $\sigma$ is real-analytic, then $\sigma \circ \varphi$ is a generalized coordinate system.  
\end{enumerate} 
\end{lemma} 
\begin{proof}
Part (a) is obvious. For part (b), we only need to verify that $\sigma \circ \varphi$ has fractional power series entries. Recalling that $\varphi = \tau \circ \Phi_{1/R}$ for some real-analytic function $\tau$, we define a new real-analytic function $\tau_0 = \sigma \circ \tau$ and observe that $\sigma \circ \varphi = \tau_0 \circ \Phi_{1/R}$, whence the result follows. 
\end{proof} 
\begin{lemma} \label{lemma-coords-comp2} 
Given any two generalized coordinate systems $\sigma : (0,1)^n \rightarrow V$ and $\varphi : (0,1)^n \rightarrow W \subseteq (0,1)^n$, there exists a power transformation $\Phi_{\mathbf R}$ such that $\sigma \circ \Phi_{\mathbf R} \circ \varphi$ is a generalized coordinate system. Further if the Jacobian $J(\mathbf y) = \det(D \sigma \circ \Phi_{\mathbf R} \circ \varphi(\mathbf y))$ of this generalized coordinate system is of the form 
\[ J(\mathbf y) = (\text{unit}) \prod_{j=1}^{n} y_j^{\frac{1}{r_j} -1} \quad \text{ for some } \mathbf r = (r_1, \cdots, r_n) \text{ with } r_j > 0, 1 \leq j \leq n, \] 
then $\sigma \circ \Phi_{\mathbf R} \circ \varphi \circ \Phi_{\mathbf r}$ is a coordinate system.  
\end{lemma} 
\begin{proof}
If the entries of $\mathbf R$ are chosen large enough so that $\sigma \circ \Phi_{\mathbf R}$ is a vector-valued real-analytic function, the first statement follows from Lemma \ref{lemma-gencoords-comp}(b). 

For the second part, we need to verify that the Jacobian of $\sigma \circ \Phi_{\mathbf R} \circ \varphi \circ \Phi_{\mathbf r}$ is a unit. Set $\mathbf x = \sigma \circ \Phi_{\mathbf R} \circ \varphi(\mathbf y)$ and $\mathbf y = \Phi_{\mathbf r}(\mathbf t)$; then $t_j = y_j^{1/r_j}$, $dt_j = (1/r_j) y_j^{1/r_j-1} dy_j$, hence
\[ d\mathbf x = (\text{unit}) \prod_{j=1}^{n} y_j^{\frac{1}{r_j}-1} \, d\mathbf y = (\text{unit}) \, d\mathbf t,  \]    
which is the desired conclusion. 
\end{proof}
\subsection{Horns} Next we introduce the model domains on which the simplifications of the roots will take place. Some of the terminology adopted for this is borrowed from \cite{Par01}, but we encourage the reader to identify the few minor differences. 
\begin{definition} \label{def-horn} For $n=1$, a {\em{horn}} in $(0, \infty)$ is simply the open interval $(0, 1)$. A {\em{horn}} in $\mathbb R^n_{>0}$ for $n \geq 2$ is a subset of the positive orthant with one of the following two possible structures: Let $f$ be a fractional power series on $(0,1)^{n-1}$. 
\begin{enumerate}[1.]
\item An \emph{$n$-dimensional horn} $W$ \emph{adjacent to} $f$ (an \emph{adjacent horn}, for short) is a  set $W \subseteq \mathbb R^n_{>0}$ such that
\begin{equation} W = \bigl\{\mathbf y = (\mathbf y', y_n) \in \mathbb R^n_{> 0} : 0 < \kappa \left(y_n - f(\mathbf y') \right) < g(\mathbf y'), \; \mathbf y' \in (0, 1)^{n-1}  \bigr\}, \label{def-solidhorn} \end{equation}
where $\kappa$ is either $+1$ or $-1$ and $g$ is fractional normal crossings on $(0, 1)^{n-1}$.  
\item An \emph{$n$-dimensional horn} $W$ \emph{separated from} $f$ (a \emph{distant horn}, for short) is a set $W \subseteq \mathbb R^n_{>0}$ such that
\begin{equation} W = \bigl\{(\mathbf y', y_n) \in \mathbb R^n_{>0} : g_1(\mathbf y') < y_n - f(\mathbf y') < g_2(\mathbf y'), \; \mathbf y' \in (0, 1)^{n-1} \bigr\}, \label{def-hollowhorn} \end{equation}
where
\begin{itemize} 
\item $g_1$ is a fractional monomial, and
\item $g_2$ is either a nonzero constant or a fractional monomial  
\end{itemize} 
with the property that $g_1 g_2 > 0$ and 
\begin{equation} g_1(\mathbf y') = a y_k^{\mu} g_2(\mathbf y'), \quad a \in \mathbb R \setminus \{0\} \label{g1g2-relation} \end{equation} for some index $1 \leq k \leq n-1$ and some constant $\mu > 0$.      
\end{enumerate} 
\end{definition}  
Clearly, the notions of adjacent and distant horns are relative to certain fractional power series, since the horn (\ref{def-hollowhorn}) which is separated from $f$ is adjacent to $f - g_1$. Nonetheless by a slight abuse of notation we will continue to refer to the horns simply as adjacent or distant, the underlying function $f$ being clear from the context. 

The next lemma is an easy fact about distant horns. 
\begin{lemma} \label{decomp-distant-horn}
Let $W$ be a set of the form (\ref{def-hollowhorn}) where $f, g_1, g_2$ satisfy the conditions in Definition \ref{def-horn} except possibly (\ref{g1g2-relation}). Assume instead that $g_1 g_2^{-1}$ is a fractional monomial on $(0,1)^{n-1}$, not necessarily of the form specified by (\ref{g1g2-relation}). Then $W$ can be covered by at most $n-1$ distant horns.  
\end{lemma} 
\begin{proof}
After a permutation of variables if necessary, we can assume that there is an index $k \in \{1, \cdots, n-1\}$ such that 
\[ g_1(\mathbf y') = a y_1^{\mu_1} \cdots y_k^{\mu_k} g_2(\mathbf y'), \quad \text{ where } \quad \mu_i > 0 \text{ for all } 1 \leq i \leq k, \]
and $a \ne 0$. Setting $h_1 = g_1$, $h_{k+1} = g_2/2$, and $h_i = y_i^{\mu_i} \cdots y_k^{\mu_k} g_2(\mathbf y')$ for $2 \leq i \leq k$ we observe that the distant horns \[ W_i = \bigl\{(\mathbf y', y_n) : h_i(\mathbf y') < y_n - f(\mathbf y') < 2h_{i+1}(\mathbf y'), \, \mathbf y' \in (0,1)^{n-1} \bigr\}, \, 1 \leq i \leq k, \] cover $W$, as claimed.    
\end{proof} 
\subsection{Tower of horns and preferred coordinates} \label{subsec-adm-pref} 
Let $\pi' : \mathbb R^n \rightarrow \mathbb R^{n-1}$ denote the projection of an $n$-dimensional vector onto its first $(n-1)$ coordinates, i.e., $\pi'(\mathbf x) = \mathbf x'$. 
\begin{definition} \label{def-prepared-cusp}
A {\em{tower of horns}} $V$ (henceforth abbreviated as {\em{tower}}) is a set defined inductively on dimension as follows: 
For $n=1$, $V$ is a one-dimensional horn. 

For $n \geq 2$, a set $V \subseteq \mathbb R^n_{>0}$ in $\mathbf x$-space is said to be an $n$-dimensional tower if $V' = \pi'(V)$ is itself an admissible generalized coordinate image of an $(n-1)$-dimensional tower, and if there exists a system of coordinates $\psi'$ on $V'$, namely $\mathbf x' = \psi'(\mathbf w')$, such that $V$ expressed in $(\mathbf w', x_n)$ variables is an $n$-dimensional horn. More precisely, $V$ is an $n$-dimensional tower if there exist 
\begin{itemize}
\item an $(n-1)$-dimensional tower $\widehat{V} \subseteq (0,1)^{n-1}$, 
\item a set $V' = \Psi'(\widehat{V}) \subseteq \mathbb R^{n-1}$ that is the image of $\widehat{V}$ under an admissible generalized coordinate transformation $\Psi'$, and 
\item an admissible system of coordinates $\psi':(0,1)^{n-1} \rightarrow V'$ on $V'$ (not necessarily related to $\Psi'$), 
\end{itemize} 
such that $V$ takes one of the following two forms:
\begin{align} V &= \left\{(\mathbf x', x_n) : 0 < \kappa \left(x_n - f(\mathbf x') \right) < g(\mathbf x'), \; \mathbf x' \in V' \right\} \quad \text{ or } \label{cusp1} \\ V &= \left\{ (\mathbf x', x_n) : g_1(\mathbf x') < x_n - f(\mathbf x') < g_2(\mathbf x'), \; \mathbf x' \in V'  \right\}. \label{cusp2} \end{align}
Here $\kappa \in \{ \pm 1 \}$ and the set 
\begin{equation} W = \{(\mathbf w', x_n) : \mathbf x' = \psi'(\mathbf w'), \; (\mathbf x', x_n) \in V,\, \mathbf w' \in (0,1)^{n-1} \} \label{cusp-W}\end{equation}
is an adjacent (respectively distant) horn in $\mathbb R^n$ if $V$ is given by (\ref{cusp1}) (respectively (\ref{cusp2})). The functions $f$, $g$, $g_1$ and $g_2$ defined on $V' \subseteq \mathbb R^{n-1}$ will be referred to as the defining functions of the tower.  
\end{definition} 

In the remainder of the section, we will attach to every tower $V$ a special class of admissible coordinatizations that respects the structure of $V$ and which will be critical in the monomialization of the roots. Fix a vector $\mathbf r' = (r_1, \cdots, r_{n-1})$, $r_j > 0$, $1 \leq j \leq n-1$. We will construct a system of coordinates on $V$ for every such $\mathbf r'$.

Suppose first that the tower $V$ is of the form (\ref{cusp1}) so that $W$ defined in (\ref{cusp-W}) is an adjacent horn. In view of (\ref{def-solidhorn}) and Definition \ref{def-prepared-cusp}, $\psi':(0,1)^{n-1} \rightarrow V'$ is an admissible coordinate transformation such that $f \circ \psi'$ is a fractional power series and $g \circ \psi'$ is fractional normal crossings on $(0,1)^{n-1}$. Since the class of fractional power series and the class of fractional normal crossings are both preserved under power transformations, we deduce that $f \circ \psi' \circ \Phi'_{\mathbf r'}$ is a fractional power series and $g \circ \psi' \circ \Phi'_{\mathbf r'}$ is fractional normal crossings as well. Here $\Phi'_{\mathbf r'}$ is defined as in (\ref{power-tf}), but in $(n-1)$ dimensions. Let $\pmb{\kappa}' = {\pmb{\kappa}}'(\mathbf r') = (\kappa_1, \cdots, \kappa_{n-1}) \geq \mathbf 0$ be the multi-exponent such that
\[ g \circ \psi' \circ \Phi'_{\mathbf r'}(\mathbf w') = (\text{unit}) (\mathbf w')^{\pmb{\kappa}'}, \quad \pmb{\kappa}' > 0.  \] We define a coordinate transformation $\varphi$ on $V$ as $\varphi = \varphi_1 \circ \varphi_2 \circ \varphi_3 \circ \varphi_4$, where 
\begin{equation} \label{solid-cusp-coord}
\begin{aligned} \varphi_1^{-1} &: \mathbf x \mapsto \mathbf w = ( \bigl(\psi' \circ \Phi'_{\mathbf r'} \bigr)^{-1}(\mathbf x'), x_n), \\ \varphi_2^{-1} &: \mathbf w \mapsto \mathbf u = \left(w_1, w_2, \cdots, w_{n-1}, w_n - f \circ \psi' \circ \Phi'_{\mathbf r'}(\mathbf w') \right), \\ \varphi_3^{-1} &: \mathbf u \mapsto \mathbf v = \bigl( u_1, \cdots, u_{n-1}, \frac{u_n}{g \circ \psi' \circ \Phi'_{\mathbf r'}(\mathbf u')} \bigr), \\ \varphi_4^{-1}&: \mathbf v \mapsto \mathbf y =   (v_1^{\kappa_1+r_1}, \cdots, v_{n-1}^{\kappa_{n-1}+r_{n-1}}, v_n). \end{aligned} \end{equation} 
We ask the reader to verify that $\varphi$ is a vector-valued fractional power series on $(0,1)^n$. We also observe that $\varphi_1$ is admissible, $\varphi_2$ is a shift, $\varphi_4$ is a power transformation and $\varphi_3$ is a composition of power transformations, generalized blow-downs and scalings by units. Further
\begin{align*}
d\mathbf x = (\text{unit})\, \mathbf w'^{\mathbf r' - \mathbf 1'} d\mathbf w &= (\text{unit}) \, \mathbf u'^{\mathbf r' - \mathbf 1'} d\mathbf u = (\text{unit}) (\mathbf v')^{\pmb{\kappa}'+{\mathbf r'} - \mathbf 1'} d\mathbf v, \\ \text{ and } d\mathbf y &=  \Bigl[ \prod_{j=1}^{n-1}(\kappa_j + r_j)\Bigr] (\mathbf v')^{\pmb{\kappa}'+{\mathbf r'}}\, d \mathbf v, 
\end{align*}   
verifying that $\varphi$ is indeed an admissible system of coordinates on $V$.  
These coordinate changes were designed so that in the new coordinates an adjacent horn takes the form
\[
V=\Bigl\{0<\kappa y_n<1,\, \mathbf y'\in  \bigl[\psi'\circ \Phi'_{\mathbf r'} \circ \Phi'_{\mathbf 1/(\pmb{\kappa}' + \mathbf r')}\bigr]^{-1}(V')\Bigr\}.
\]

Suppose next that $V$ is of the form (\ref{cusp2}), so that $W$ defined in (\ref{cusp-W}) is a distant horn. In view of the definition (\ref{def-hollowhorn}), this implies that $f \circ \psi'$ (and hence $f \circ \psi' \circ \Phi'_{\mathbf r'}$) is a fractional power series, $g_1 \circ \psi' \circ \Phi'_{\mathbf r'}$ is a fractional monomial, and $g_2 \circ \psi' \circ \Phi'_{\mathbf r'}$ is either a nonzero constant or a fractional monomial of the same sign as $g_1 \circ \psi' \circ \Phi'_{\mathbf r'}$. Let  
\[ g_2 \circ \psi' \circ \Phi'_{\mathbf r'}(\mathbf w') = a_2 \, (\mathbf w')^{\pmb{\kappa}'}, \quad \pmb{\kappa}' \geq 0, \, a_2 \in \mathbb R \setminus \{0\}. \] By a permutation of the variables $\mathbf w'$ if necessary, we may assume that  
 \[ g_1 \circ \psi' \circ \Phi'_{\mathbf r'}(\mathbf w') = a w_{n-1}^{\mu} \, \bigl[g_2 \circ \psi' \circ \Phi'_{\mathbf r'}(\mathbf w') \bigr]. \]
The system of coordinates $\varphi$ on $V$ is now defined by $\varphi = \varphi_1 \circ \varphi_2 \circ \varphi_3 \circ \varphi_4$, where $\varphi_1$ and $\varphi_2$ are the same as in (\ref{solid-cusp-coord}), while 
\begin{equation} \label{hollow-cusp-coord}
\begin{aligned}
\varphi_3^{-1} &: \mathbf u \mapsto \mathbf v = \left(u_1, \cdots, u_{n-2}, \frac{g_1 \circ \psi' \circ \Phi'_{\mathbf r'}(\mathbf u')}{u_n}, \frac{u_n}{g_2 \circ \psi' \circ \Phi'_{\mathbf r'}(\mathbf u')} \right), \\ 
\varphi_4^{-1} &: \mathbf v \mapsto \mathbf y = \left( v_1^{\kappa_1+r_1}, \cdots, v_{n-2}^{\kappa_{n-2}+r_{n-2}}, v_{n-1}^{\frac{\kappa_{n-1}+r_{n-1}}{\mu}}, v_n^{\frac{\kappa_{n-1}+r_{n-1}}{\mu} + 1}\right).  
\end{aligned} \end{equation} 
Thus $\varphi$ is a vector-valued fractional power series, and each $\varphi_i$, $i= 1, 2,3,4$ is again a composition of the elementary transformations described in \S \ref{examples}. Combining the Jacobians of the various components, we find that 
\begin{align*} d\mathbf x &= (\text{unit}) \, \mathbf w'^{\mathbf r' - \mathbf 1'}d\mathbf w = (\text{unit}) \, \mathbf u'^{\mathbf r' - \mathbf 1'}  d\mathbf u, \\ &= \frac{a_2}{a^{(\kappa_{n-1}+ r_{n-1})/\mu} \mu} v_1^{\kappa_1 + r_1-1} \cdots v_{n-2}^{\kappa_{n-2} + r_{n-2}-1} v_{n-1}^{\frac{\kappa_{n-1}+r_{n-1}}{\mu}-1} v_n^{\frac{\kappa_{n-1}+r_{n-1}}{\mu}} d\mathbf v \\ &= \frac{a_2}{a^{(\kappa_{n-1}+r_{n-1})/\mu} \mu} \Bigl[\Bigl( \frac{\kappa_{n-1}+r_{n-1}}{\mu}\Bigr) \Bigl(\frac{\kappa_{n-1}+r_{n-1}}{\mu} + 1 \Bigr) \prod_{j=1}^{n-2} (\kappa_j+r_j) \Bigr]^{-1} \, d\mathbf y, \end{align*}
so that the Jacobian of $\varphi$ is indeed a unit, and $\varphi$ is an admissible coordinate system on $V$ as claimed. In the revised coordinates, the distant horn takes the form 
\[ V = \Bigl\{ 0 < \kappa v_{n-1} < 1, 0 < \kappa v_n < 1, \; \psi' \circ \Phi_{\mathbf r'} \bigl(v_1, \cdots, v_{n-2}, \left(a^{-1} v_{n-1}v_{n} \right)^{\frac{1}{\mu}} \bigr) \in V' \Bigr\}. \]
\begin{definition}
Let $V$ be an $n$-dimensional tower. A choice of coordinates $\varphi$ on $V$ of the form described above will be called a {\em{preferred system of coordinates}} for $V$. The pair $(V,\varphi)$ will be referred to as a {\em{prepared tower}}. 
\end{definition} 
Any preferred system of coordinates on a tower has a nice property that we record here for future use.  
\begin{lemma} \label{lemma-fnc}
Let $(V, \varphi)$ be an $n$-dimensional prepared tower. 
If $h$ is a function on $V$ that is fractional normal crossings in the variables $\mathbf u = (\mathbf u', u_n)$, where \[\mathbf u' = \psi'^{-1}(\mathbf x') \quad \text{ and } \quad u_n = x_n - f(\mathbf x') \] are as in (\ref{solid-cusp-coord}), then $h$ converts to fractional normal crossings in any of the preferred system of coordinates. 
\end{lemma} 
\begin{proof}
If $V$ is such that $W$ given in (\ref{cusp-W}) is an adjacent horn, the change of variables $\varphi_3$ in (\ref{solid-cusp-coord}) implies that \[u_i = v_i \; \text{ for } 1 \leq i \leq n-1, \quad \text{ while } \quad u_n = g \circ \psi'(\mathbf v') v_n.\] 
Thus each entry of $\mathbf u$ is fractional normal crossings in $\mathbf v$. We now point out two properties of units that are easy to verify:
\begin{itemize} 
\item A fractional power series unit continues to be one after being raised to an arbitrary power.
\item A fractional power series unit in $\mathbf u$, when evaluated at $\mathbf u = \varphi_3(\mathbf v)$ transforms to a fractional power series unit in $\mathbf v$.  
\end{itemize} 
These two facts imply that $h$ is fractional normal crossings in the variables $\mathbf v$. Since the preferred coordinate system $\mathbf y$ is obtained from $\mathbf v$ by a power transformation the result follows.

If $W$ is a distant horn, then by (\ref{hollow-cusp-coord}), \[u_i = v_i \; \text{ for } 1 \leq i \leq n-2, \quad u_{n-1} = (a^{-1}v_{n-1}v_n)^{1/\mu}, \quad u_n = \left(g_2 \circ \psi'\circ\Phi_{{\mathbf r'}}(\mathbf u')\right) v_n. \] Here each entry of $\mathbf u$ is a fractional monomial in terms of $\mathbf v$, and the conclusion is even easier to deduce than the previous case.  
\end{proof} 
\subsection{Constructibility and Complexity}
We conclude this section with two definitions concerning the structural complexity of sets and functions. In the next section, they will be used to establish the degree of computational effectiveness of the various steps of the resolution algorithm.      
\begin{definition}
Let $F$ be a real-analytic function with the property that $F = (\text{unit})G$ in a neighborhood of the origin in $\mathbb R^{n+1}$, where \[G(\mathbf x, x_{n+1}) = x_{n+1}^d + \sum_{\nu=1}^{d} c_{\nu}(\mathbf x) x_{n+1}^{d-\nu} \] is a Weierstrass polynomial in $x_{n+1}$ with real-analytic coefficients $\{c_{\nu}(\mathbf x): 1 \leq \nu \leq d \}$. An $n$-variate real-analytic  function $\Lambda_{n}(\mathbf x)$ is said to be {\em{constructible (with respect to $F$)}} if there exists a $d$-variate polynomial $Q$ (possibly depending on $F$) such that $\Lambda_n(\mathbf x) = Q(c_1(\mathbf x), \cdots, c_{d}(\mathbf x))$.

The definition extends inductively to functions in $\leq n$ variables. For $k < n$, let $\Lambda_k$ and $\Lambda_{k+1}$ be real-analytic functions in $k$ and $(k+1)$ variables respectively. If $\Lambda_{k}$ is constructible with respect to $\Lambda_{k+1}$ which in turn is constructible with respect to $F$, we say that $\Lambda_k$ is constructible with respect to $F$.  
\end{definition} 
Any of the coefficients $c_{\nu}(\mathbf x)$ is of course constructible with respect to $F$. So is the discriminant $\Delta_G$ if $G$ is irreducible. On the other hand, a root of $F$ or the root of a function that is constructible with respect to $F$ is not in general constructible.  
\begin{definition}
Given an $(n+1)$-variate real-analytic function $F$, a tower of dimension $\leq (n+1)$ will be called {\em{constructible (with respect to $F$)}} if its defining functions are constructible with respect to $F$.  

A $k$-dimensional tower ($k \leq n$) that is not necessarily constructible is said to be of {\em{low resolution complexity relative to $F$}} if there exists a $k$-variate constructible function $\Lambda_k$ such that the defining functions of the tower can be completely specified in terms of the roots of $\Lambda_k$. In other words, the description of a tower of low resolution complexity only requires knowledge of a constructible function of $k < n+1$ variables. However, note that the defining functions of the tower need not be constructible with respect to $\Lambda_k$. 
\end{definition} 
For example, let $F$ be an irreducible Weierstrass polynomial. Suppose that $\rho_1, \rho_2$ are two roots of the constructible function $\Lambda_{n} = \Delta_F$ such that $\rho_1$, $\rho_2$ are fractional power series and $\rho_2 - \rho_1$ is non-negative fractional normal crossings. Consider the $n$-dimensional tower $V$ given by (\ref{cusp1}) where $f = \rho_1$, $g = \rho_2 - \rho_1$, $V'=(0,1)^{n-1}$, $\psi' =$ identity. Then $V$ is of low resolution complexity relative to $F$.

\section{A resolution of singularities algorithm} \label{sec-resolution}
The goal of this section is to prove the following theorem and its corollary. 
\subsection{The main results}
\begin{theorem} \label{mainthm-resolution}
Let $F$ be a (possibly complex-valued) real-analytic function defined on a small open neighborhood of the origin in $\mathbb R^{n+1}$, $F(\mathbf 0, 0) = 0$. By an orthogonal linear coordinate transformation if necessary, we may ensure $F(\mathbf 0, x_{n+1}) \not\equiv 0$. Then there exist 
\begin{enumerate}[1.]
\item a small constant $\epsilon > 0$;
\item a non-negative integer $\beta_{n+1}$; 
\item a real-analytic unit $u_0$ on $(-\epsilon, \epsilon)^{n+1}$; 
\item an auxiliary real-analytic function $\Lambda : (-\epsilon, \epsilon)^n \rightarrow \mathbb C$ that is constructible with respect to $F$; and
\item a finite collection $\{(V, \varphi_V) : V \in \mathcal V \}$ of $n$-dimensional prepared towers that are in general non-constructible with respect to $\Lambda$ but are of low resolution complexity relative to $F$, 
\end{enumerate}
which satisfy the following properties:
\begin{enumerate}[(a)]
\item The set $(-\epsilon, \epsilon)^n \setminus \bigcup \{V : V \in \mathcal V \}$ has dimension $< n$. \label{parta} 
\item On each $V \in \mathcal V$, the auxiliary function $\Lambda$ is fractional normal crossings in the new coordinates $\mathbf y = \varphi_V^{-1}(\mathbf x)$. In other words, $\Lambda \circ \varphi_V$ is fractional normal crossings on $(0,1)^n$. \label{partb}
\item On each $V \in \mathcal V$, the roots of $F$ as well as their differences are fractional normal crossings in $\mathbf y$. Specifically, there exist \label{partc} 
\begin{itemize} 
\item a positive integer $N_V$ and a finite totally ordered collection of non-negative multi-indices $\{ \pmb{\gamma}_i(V); 1 \leq i \leq N_V \} \subseteq \mathbb Q^n$,  \[\pmb{\gamma}_i(V) \leq \pmb{\gamma}_{i+1}(V)  \text{ for all } 1 \leq i \leq N_V, \] 
\item units $\{u_i : 1 \leq i \leq N_V \}$ on $(0,1)^n$  
\end{itemize}
such that 
\begin{itemize} 
\item for any (i,j) with $i \ne j$, $1 \leq i, j \leq N_V$, the difference $u_i(\mathbf y) \mathbf y^{\pmb{\gamma}_i} - u_j(\mathbf y) \mathbf y^{\pmb{\gamma}_j}$ is either identically zero or fractional normal crossings on $(0,1)^n$, and 
\item the function $F$ admits the factorization 
\begin{equation} \frac{F(\mathbf x, x_{n+1})}{u_0(\mathbf x, x_{n+1})} =  x_{n+1}^{\beta_{n+1}} \prod_{i = 1}^{N_V} \left( x_{n+1} - u_i(\mathbf y) \mathbf y^{\pmb{\gamma}_i} \right) \label{F-factid} \end{equation}
for $\mathbf y \in (0,1)^n$, $\mathbf x = \varphi_V(\mathbf y) \in V$, $(\mathbf x, x_{n+1}) \in V \times (-\epsilon , \epsilon)$.
\end{itemize}   
\item The Newton polyhedron NP$(F;\Phi_V)$ is defined by a monotone edge path, where $\Phi_V :  (0,1)^n \times (-\epsilon, \epsilon) \rightarrow V \times (- \epsilon, \epsilon)$ denotes the coordinate transformation \label{partd} 
\[ (\mathbf y, y_{n+1}) = \Phi_V^{-1}(\mathbf x, x_{n+1}) = (\varphi_V^{-1}(\mathbf x), x_{n+1}). \]
\end{enumerate}  
\end{theorem} 

We will see in \S \ref{sec-proof-mainthm1} that a factorization (\ref{F-factid}) alone, even with ordered leading exponents, does not suffice to specify the critical integrability index. One requires a finer analysis of the roots of $F$, where the real parts of the roots also have the fractional normal crossings structure and are ``well-ordered'', amongst themselves and with relation to the actual roots. That this can be done, essentially with the same methodology, is the content of the next corollary.      
\begin{corollary}\label{corollary-mainthm-resolution} 
Let $F$ be an $(n+1)$ variate (not necessarily analytic) function, $U \subseteq \mathbb R^n$ an open set with $\mathbf 0 \in \overline{U}$ and $\Psi : (0,1)^n \rightarrow U$ a coordinate transformation on $U$ such that 
\[ F(\Psi(\mathbf u), x_{n+1}) = (\text{unit}) x_{n+1}^{\beta_{n+1}} \prod_{i=1}^{N} \left(x_{n+1} - \varrho_i\left(\mathbf u\right)\right), \]
where $\{ \varrho_i : 1 \leq i \leq N \}$ is a collection of absolutely and uniformly convergent fractional power series on $(-1,1)^n$. Then there exists a finite collection of set-coordinate pairs $\{(V, \varphi) : V \in \mathcal V \}$ with the following properties: 
\begin{enumerate}[(a)]
\item The set $U \setminus \bigcup \{V : V \in \mathcal V \}$ has dimension $< n$.
\item Each set $V$ is the admissible generalized coordinate image of an $n$-dimensional prepared tower. 
\item $\varphi : (0,1)^n \rightarrow V$ is an admissible system of coordinates on $V$ such that each of the quantities \label{corollary-partc}
\[ \{ \varrho_i \}, \{\text{Re}(\varrho_i) \}, \{ \varrho_i - \varrho_j, i \ne j \}, \{ \varrho_i - \text{Re}(\varrho_j) \}, \{ \text{Re}(\varrho_i) - \text{Re}(\varrho_j) : i \ne j \} \] 
is either identically zero or fractional normal crossings in these coordinates. 
\end{enumerate}    
\end{corollary}
\subsection{Construction of $\Lambda$} \label{Construction-Lambda}
Since every real-analytic function in a neighborhood of the origin in $\mathbb R^{n+1}$ is the restriction of a holomorphic function near the origin in $\mathbb C^{n+1}$, we invoke the Weierstrass preparation theorem (Theorem \ref{W-prep} after a nonsingular linear change of coordinates and a shrinking of domain if necessary) to find a nonvanishing holomorphic function $u_0$ such that $F$ is of the form 
\begin{equation} \frac{F(\mathbf x, x_{n+1})}{u_0(\mathbf x, x_{n+1})} = G(\mathbf x, x_{n+1}) = x_{n+1}^d + \sum_{\nu=1}^{d} c_{\nu}(\mathbf x) x_{n+1}^{d- \nu},   \label{W-poly}\end{equation} 
where the coefficients $c_\nu$ are holomorphic in $\mathbf x$, with $c_\nu(\mathbf 0) = 0$. Without loss of generality (after a scaling in each coordinate if needed), we may assume that (\ref{W-poly}) holds for $(\mathbf x, x_{n+1}) \in (-2,2)^{n+1}$ and that the coefficients $c_{\nu}$ are bounded on $(-2,2)^n$. By the unique factorization (\ref{polynomial-factors}) mentioned in \S \ref{subsec-factor} and Lemma \ref{lemma-Wpoly-factor}, any such $G$ can be represented in the form  
\begin{equation} \begin{aligned} G(\mathbf x, x_{n+1}) &= x_{n+1}^{\beta_{n+1}} \prod_{\ell=1}^{L} \left[G_\ell(\mathbf x, x_{n+1}) \right]^{m_\ell}, \text{ with } \\ c(\mathbf x) &:= \prod_{\ell=1}^{L} \left[G_\ell(\mathbf x, 0) \right]^{m_\ell} \not\equiv 0, \end{aligned} \label{cGl}\end{equation} 
where $\beta_{n+1}$ is a non-negative integer, the functions $G_\ell$ are distinct irreducible Weierstrass polynomials in $x_{n+1}$, and the exponents $m_\ell$ are positive integers. By Lemma \ref{disc-red-lemma}, the discriminant of $P = G_1 G_2 \cdots G_L$ (considering $P$ as a polynomial in $x_{n+1}$ with coefficients depending on $\mathbf x$) is not an identically vanishing function of $\mathbf x$. The auxiliary function $\Lambda$ used in Theorem \ref{mainthm-resolution} is defined as follows, \begin{equation} \Lambda(\mathbf x) := c(\mathbf x) \Delta_P(\mathbf x) \not\equiv 0.  \label{def-Lambda}  \end{equation}    
It is clear that $\Lambda$ is real-analytic and constructible with respect to $F$.  

\subsection{Ingredients of the proof} 
As indicated in the introduction, the proof of Theorem \ref{mainthm-resolution} uses induction on dimension and is based iteratively on two main principles; namely, 
\begin{itemize}
\item a factorization of the form (\ref{F-factid}) for $\Lambda$ essentially implies that $\Lambda$ can be monomialized, after decomposition of $(-\epsilon, \epsilon)^{n}$ into a finite number of subsets and assigning suitable coordinates to each subset,   
\item  monomialization of $\Lambda$ implies (\ref{F-factid}) for $F$. 
\end{itemize}
The precise statements of these steps constitute the main propositions in this subsection. See Proposition \ref{lemma-block2} for the first step and Proposition \ref{lemma-block1} for the second.   
\begin{proposition}\label{lemma-block1}
Let $F$, $G$ and $\Lambda$ be as in (\ref{W-poly}), (\ref{cGl}) and (\ref{def-Lambda}). Let $V \subseteq (0,1)^n$ be an open set whose closure contains the origin.
 Suppose that $\sigma :(0,1)^n \rightarrow V$ is a generalized coordinate transformation on $V$ such that 
the function $\mathbf y \mapsto \Lambda \circ \sigma(\mathbf y)$ is fractional normal crossings on $(0,1)^n$.  

Then there exist exponents $\{\pmb{\gamma}_i : 1 \leq i \leq d \}$ with $\mathbf 0 < \pmb{\gamma}_i \leq \pmb{\gamma}_{i+1}$ and units $\{ u_i : 1 \leq i \leq N \}$ on $(0,1)^n$ such that 
\begin{equation} \begin{aligned} G(\sigma(\mathbf y), x_{n+1}) &= \frac{F \left( \sigma(\mathbf y), x_{n+1}\right)}{u_0 \left( \sigma(\mathbf y), x_{n+1}\right)} \\ &=  x_{n+1}^{\beta_{n+1}} \prod_{i=1}^{N} \left(x_{n+1} - r_i(\mathbf y) \right), \quad r_i(\mathbf y) = u_i(\mathbf y) \mathbf y^{\pmb{\gamma}_{i}}.  \end{aligned}  \label{F-sigma-factor}   \end{equation}  
Further all the differences $\{r_i - r_{i'} : i \ne i' \}$ are either identically zero or fractional normal crossings on $(0,1)^n$, and the Newton polyhedron of $F(\sigma(\mathbf y), x_{n+1})$ in the coordinates $(\mathbf y, x_{n+1})$ is defined by a monotone edge path. 
\end{proposition} 
\begin{proof}
The fact that the roots of $F$ and their differences are fractional normal crossings in $\mathbf y$ is essentially a consequence of the Jung-Abhyankar theorem applied to polynomials with fractional power series coefficients, but we furnish the details. 
Let $\Phi_R$ be the power map defined in (\ref{power-map}), with $R$ chosen sufficiently large so that $\sigma \circ \Phi_R$ is a vector-valued real-analytic function whose power series expansion converges absolutely and uniformly on an open parallelepiped containing $[-1,1]^n$ centered at the origin. Let us keep in mind that $\Phi_R([0,1]^n) = [0,1]^n$ and $\sigma \circ \Phi_R((0,1)^n) = \sigma((0,1)^n) = V$. Set  
\begin{align*} \tilde{F}(\mathbf y, x_{n+1}) &:= F \left(\sigma \circ \Phi_R (\mathbf y), x_{n+1}\right), \\ \tilde{G}(\mathbf y, x_{n+1}) &:= G \left(\sigma \circ \Phi_R (\mathbf y), x_{n+1}\right), \\ \tilde{G}_{\ell}(\mathbf y, x_{n+1}) &:= G_{\ell} \left(\sigma \circ \Phi_R(\mathbf y), x_{n+1} \right), \end{align*}
so that $\tilde{F}$ is real-analytic and $\tilde{G}$, $\tilde{G}_{\ell}$ are Weierstrass polynomials defined near the origin in $\mathbb R^{n+1}$.  
We will show that $\tilde{G}$ satisfies the assumptions and hence the conclusions of Lemma \ref{JA-cor}.

Since $F = (\text{unit})G$ on $(-2,2)^{n+1}$ by (\ref{W-poly}), we can find an open neighborhood $U$ of the origin in $\mathbb R^n$ containing $[0,1]^n$ such that $\sigma \circ \Phi_R(U) \subseteq (-2,2)^n$ and $\tilde{F} = (\text{unit}) \tilde{G}$ on $U \times (-2, 2)$. Further the factorization of $G$ in (\ref{cGl}) implies that $\tilde{G}$ factorizes as 
\[ \tilde{G}(\mathbf y, x_{n+1}) = x_{n+1}^{\beta_{n+1}} \prod_{\ell=1}^{L} \left[\tilde{G}_{\ell}(\mathbf y, x_{n+1}) \right]^{m_{\ell}} \quad \text{ on } \quad U \times \mathbb R. \]   
In addition, we define 
\begin{align*}  \tilde{c}(\mathbf y) &:= \prod_{\ell=1}^{L} \left[ \tilde{G}_{\ell}(\mathbf y, 0)\right]^{m_{\ell}}, \\ \text{ so that } \quad c(\mathbf y) &=  \prod_{\ell=1}^{L} \bigl[ {G}_{\ell}(\sigma \circ \Phi_R(\mathbf y), 0)\bigr]^{m_{\ell}} = c(\sigma \circ \Phi_R(\mathbf y)). \end{align*} Since $c \not\equiv 0$, we observe that $\tilde{c}$ is a nontrivial real-analytic function on $U$. 

As in Lemma \ref{JA-cor}, we set $\tilde{P} := \tilde{G}_1 \cdots \tilde{G}_{L}$ so that $\tilde{P}$ is yet another polynomial with real-analytic coefficients on $U$. Further $\tilde{P}(\mathbf y, x_{n+1}) = P\left(\sigma \circ \Phi_R\left(\mathbf y\right), x_{n+1}\right) $ has the same degree as $P$ as a polynomial in $x_{n+1}$, with coefficients equal to the corresponding coefficients of $P$ evaluated at $\sigma \circ \Phi_R(\mathbf y)$. Since the discriminant is a fixed polynomial of the coefficients, it follows that  
\[ \Delta_{\tilde{P}}(\mathbf y) = \Delta_{P} \left( \sigma \circ \Phi_R(\mathbf y) \right). \]
We recall that $\Delta_{P} \not\equiv 0$, hence conclude $\Delta_{\tilde{P}}$ is $\not\equiv 0$ and is real-analytic on $U$.

Finally we set \begin{equation*} \tilde{\Lambda}(\mathbf y) := \tilde{c}(\mathbf y) \Delta_{\tilde{P}}(\mathbf y) = \Lambda (\sigma \circ \Phi_R(\mathbf y)). \end{equation*} By our hypothesis on $\Lambda$ and for $R$ chosen sufficiently large, the function $\tilde{\Lambda}$ is normal crossings in $\mathbf y$, therefore so are $\tilde{c}$ and $\Delta_{\tilde{P}}$. 
Thus we have verified all the hypotheses of Lemma \ref{JA-cor} for $\tilde{G}$. By the lemma therefore, all the nontrivial roots of $\tilde{G}$ as well as their differences are fractional normal crossings in $\mathbf y$ on $(0,1)^n$, and hence (\ref{F-sigma-factor}) holds with $G$ replaced by $\widetilde{G}$. Replacing $\mathbf y$ by $\Phi_{1/R}(\mathbf y)$ and observing that the class of fractional normal crossings on $(0,1)^n$ is preserved under $\Phi_{1/R}$, we arrive at the relation (\ref{F-sigma-factor}), establishing en route that $r_i - r_j$ is either identically zero or fractional normal crossings for all $i \ne j$.

To show that the exponents $\{ \pmb{\gamma}_i \}$ form a totally ordered set, we invoke Lemma \ref{monomial-ordering-lemma}. Since we have just now shown that for any $i \ne j$ with $r_i \not\equiv r_j$ \[ r_i(\mathbf y) - r_j(\mathbf y) = u_i(\mathbf y) \mathbf y^{\pmb{\gamma}_i} - u_j(\mathbf y) \mathbf y^{\pmb{\gamma}_j} \]
is fractional normal crossings, the lemma implies that either $\pmb{\gamma}_i \leq \pmb{\gamma}_j$ or $\pmb{\gamma}_j \leq \pmb{\gamma}_i$. The statement concerning the Newton polyhedron follows from Lemma \ref{mep-lemma3}.      
\end{proof}

Note that the conclusions of Theorem \ref{mainthm-resolution} and Proposition \ref{lemma-block1} yield the same statements in terms of the factorization of $F$. If we wish to use the proposition to prove the theorem, we need to decompose an open neighborhood of the origin in $\mathbb R^n$ into a finite number of $n$-dimensional prepared towers and reduce $\Lambda$ to fractional normal crossings on each tower in the preferred coordinates. An inductive argument given in \S\S \ref{subsec-mainthm-resolution-proof} shows that $\Lambda$ admits a factorization of the form (\ref{F-factid}), in dimension $n$ of course. The following proposition partially bridges the gap between these two statements. Namely, we show that if $F$ obeys a more refined version of the factorization (\ref{F-factid}), then $F$ can be expressed as fractional normal crossings in carefully chosen coordinates and after suitable decompositions. In the sequel, we will apply this result with $F$ replaced by $\Lambda$.      
\begin{proposition}\label{lemma-block2}
Let $V$ be an admissible generalized coordinate image of an $n$-dimensional tower. Suppose that $F$ is a function (not necessarily real-analytic) admitting the factorization (\ref{F-factid}) on the set $V$ equipped with a system of coordinates $\varphi$, where 
\begin{enumerate}[(i)]
\item the roots of $F$ satisfy all the conditions specified in part (\ref{partc}) of Theorem \ref{mainthm-resolution}; namely each of the roots and their differences is either identically zero or fractional normal crossings in the variables $\mathbf y = \varphi^{-1}(\mathbf x) \in (0,1)^n$, $\mathbf x \in V$. \label{part1}
\item Additionally, assume that each of the functions in $\{ \text{Re}(u_i(\mathbf y)) \mathbf y^{\pmb{\gamma}_i}:1 \leq i \leq N_V \}$ and \[\bigl\{ u_i(\mathbf y) \mathbf y^{\pmb{\gamma}_i} - \text{Re}(u_j(\mathbf y)) \mathbf y^{\pmb{\gamma}_j}, \text{Re}(u_i(\mathbf y)) \mathbf y^{\pmb{\gamma}_i} - \text{Re}(u_j(\mathbf y)) \mathbf y^{\pmb{\gamma}_j}  : 1 \leq i, j \leq N_V \bigr\} \] is either identically zero or fractional normal crossings on $(0,1)^n$. \label{part2}  
\end{enumerate}
Then there is a decomposition of $V \times (-1, 1)$ (excluding possibly a subset of dimension $< (n+1)$) into a finite number of $(n+1)$-dimensional prepared towers on each of which $F$ is fractional normal crossings in any preferred system of coordinates.  
\end{proposition} 
We will first present the proof of Proposition \ref{lemma-block2} in two special cases which already capture the main ideas (see Lemmas \ref{imaginary-roots} and \ref{real-roots} below). The proof in the general case involves additional technicalities and is given in \S\S \ref{subsec-proof-lemma-block2}. In order to state the lemmas we need to set up some notation highlighting the finer structures of the roots in (\ref{F-factid}). Let us fix the set $V \subseteq \mathbb R^n_{>0}$ and the coordinate system $\varphi$ on $V$ for which (\ref{F-factid}) holds. By a slight abuse of notation, we rewrite (\ref{F-factid}) as  
\begin{equation} F(\mathbf x, x_{n+1}) = u_0(\mathbf x, x_{n+1}) x_{n+1}^{\beta_{n+1}} \prod_{i=1}^{M} \left( x_{n+1} - r_i(\mathbf y) \right)^{n_i}, \; r_i(\mathbf y) = u_i(\mathbf y) \mathbf y^{\pmb{\gamma}_i} \label{F-factid2} \end{equation} 
where each $n_i \geq 1$ is an integer with $n_1 + n_2 + \cdots + n_M = N_V$, and the roots $r_i$ are distinct. Let us denote by $\{\pmb{\alpha}_\ell : 1 \leq \ell \leq L \}$ the distinct elements of the set of exponents $\{ \pmb{\gamma}_i : 1 \leq i \leq M\}$, which we order as follows, 
\begin{equation} \pmb{\alpha}_1 < \pmb{\alpha}_2 < \cdots < \pmb{\alpha}_L. \label{alpha-ordering} \end{equation}
We also define 
\[ \mathcal L_{\ell} := \{i : \pmb{\gamma}_i = \pmb{\alpha}_{\ell}, 1 \leq i \leq M \}. \] 
Thus if $i, j \in \mathcal L_{\ell}$ for some $\ell$, $i \ne j$, then $u_i \not\equiv u_j$; specifically, $u_i - u_j$ is either a unit or fractional normal crossings. The nature of the leading order coefficients of the units $u_i$ dictate the cases considered in the next two lemmas. Specifically, Lemma \ref{imaginary-roots} and its corollaries deal with situations where none of the factors of $F$ have any cancellation. In Lemma \ref{imaginary-roots}, this is achieved by assuming that the dominant components of the roots are all purely imaginary. In Corollary \ref{corollary-imroots}, the lack of cancellation is due to the fact that the dominant parts of the roots are real but do not have the same sign as $x_{n+1}$. The other end of the spectrum, namely where all the roots are real and there could be cancellation, is handled in Lemma \ref{real-roots}.  
\begin{lemma} \label{imaginary-roots}
Let $V$ and $\varphi$ be as in Proposition \ref{lemma-block2}. Assume that $F$ obeys the factorization (\ref{F-factid2}) with the ordering (\ref{alpha-ordering}) of exponents and that Im$(u_i)$ is a unit on $(0,1)^n$ for every $1 \leq i \leq M$. Then the conclusion of Proposition \ref{lemma-block2} holds.  
\end{lemma}   
\begin{proof} 
We will need to use the two constants $C_0$ and $c_0$ defined as follows,  
\begin{align*}  C_0 &:= \sup \{ |u_i(\mathbf y)| : \mathbf y \in (0,1)^n, 1 \leq i \leq M \} > 0, \\ c_0 &:= \inf \{\left| \text{Im}(u_i(\mathbf y)) \right| : \mathbf y \in (0,1)^n, 1 \leq i \leq M \} > 0.  \end{align*}
Fix any choice of positive constants $\{p_{\ell} : 0 \leq \ell \leq L-1\}$ and $\{q_{\ell} : 1 \leq \ell \leq L \}$ with the property that $q_{\ell} > p_{\ell-1} > C_0$ for all $1 \leq \ell \leq L-1$. Define sets $\widetilde{V}_{\ell, \kappa}$, $0 \leq \ell \leq L$, $\kappa \in \{ \pm 1\}$, where 
\begin{equation} \label{Vtilde-lkappa}
\begin{aligned}
\widetilde{V}_{L, \kappa} &:= \left\{(\mathbf x, x_{n+1}) : 0 < \kappa x_{n+1} <  q_{L} \mathbf y^{\pmb{\alpha}_L} \right\}, \\
\widetilde{V}_{\ell, \kappa} &:= \left\{(\mathbf x, x_{n+1}) : p_{\ell} \mathbf y^{\pmb{\alpha}_{\ell+1}} < \kappa x_{n+1} <  q_{\ell} \mathbf y^{\pmb{\alpha}_\ell}  \right\}, \quad 1 \leq \ell \leq L-1, \\
\widetilde{V}_{0, \kappa}  &:= \left\{(\mathbf x, x_{n+1}) : p_0 \mathbf y^{\pmb{\alpha}_1} < \kappa x_{n+1} < 1  \right\}, \text{ so that } \end{aligned} \end{equation}  
\[\bigcup_{\begin{subarray}{c} 0 \leq \ell \leq L \\ \kappa = \pm 1 \end{subarray}} \widetilde{V}_{\ell, \kappa} = \bigl[V \times (-1, 1) \bigr] \setminus \{ x_{n+1} = 0 \}. \]  
We observe that $\widetilde{V}_{L, \kappa}$ is itself a tower which converts to a horn adjacent to the constant zero function when expressed in the variables $(\mathbf y, x_{n+1})$. This need not be the case for $\widetilde{V}_{\ell, \kappa}$ with $\ell \leq L-1$. However, the ordering (\ref{alpha-ordering}) implies, by Lemma \ref{decomp-distant-horn}, that each set $\{ \widetilde{V}_{\ell, \kappa} : 0 \leq \ell \leq L-1\}$ can be decomposed as a finite union of $(n+1)$-dimensional towers, each of which converts to a horn separated from 0 in the $(\mathbf y, x_{n+1})$ variables. We intend to show that on each of these towers, $F$ is fractional normal crossings when expressed in {\em{any}} preferred coordinate system.

Fix $1 \leq \ell \leq L-1$. Fix a tower arising from $\widetilde{V}_{\ell, \kappa}$ and a preferred system of coordinates on this tower. Recall from (\ref{F-factid2}) that $F$ is the product of a real-analytic unit $u_0$ with factors of the form $(x_{n+1} - r_i(\mathbf y))$. Since $u_0$ converts to a fractional power series unit in any of the preferred coordinates, it suffices to show that each of the factors $(x_{n+1} - r_i(\mathbf y))$ transforms to fractional normal crossings. More precisely, we write 
\[ x_{n+1} - r_i(\mathbf y) = \begin{cases} r_i(\mathbf y) \left[ \frac{x_{n+1}}{r_i(\mathbf y)} - 1 \right] &\text{ if } i \in \mathcal L_k, k \leq \ell,\\  x_{n+1} \left[ 1 - \frac{r_i(\mathbf y)}{x_{n+1}} \right] &\text{ if } i \in \mathcal L_k, k > \ell. \end{cases}   \] 
By the construction of preferred coordinates in \S \ref{subsec-adm-pref}, 
the functions \begin{equation} \left\{\frac{x_{n+1}}{r_i(\mathbf y)} - 1 : i \in \mathcal L_k, k \leq \ell \right\} \quad \text{ and } \quad \left\{1 - \frac{r_i(\mathbf y)}{x_{n+1}} : i \in \mathcal L_k, k > \ell \right\} \label{show-units} \end{equation}  are fractional power series in the preferred system of coordinates. We will show momentarily that they are units. Assuming this for now and combining this with (\ref{F-factid2}) we arrive at the expression
\begin{align*} 
F(\mathbf x, x_{n+1}) & = u_0(\mathbf x, x_{n+1}) \Biggl( x_{n+1}^{\beta_{n+1}} \prod_{\begin{subarray}{c} i \in \mathcal L_k \\ k \leq \ell\end{subarray}} r_i(\mathbf y) \left[ \frac{x_{n+1}}{r_i(\mathbf y)} - 1 \Biggr] \right) \Biggl( \prod_{\begin{subarray}{c} i \in \mathcal L_k \\ k > \ell  \end{subarray}} x_{n+1} \left[ 1 - \frac{r_i(\mathbf y)}{x_{n+1}} \right] \Biggr) \\ 
&= u(\mathbf x, x_{n+1})\Bigl[ x_{n+1}^{\beta_{n+1}} \Bigl( \prod_{\begin{subarray}{c} i \in \mathcal L_k \\ k \leq \ell \end{subarray}} r_i(\mathbf y) \Bigr) \Bigl(\prod_{\begin{subarray}{c} i \in \mathcal L_k \\ k > \ell \end{subarray}} x_{n+1} \Bigr) \Bigr].
\end{align*} 
The expression within the box above is fractional normal crossings in $(\mathbf y, x_{n+1})$, hence by Lemma \ref{lemma-fnc} it continues to be so in any set of preferred coordinates. The function $u(\mathbf x, x_{n+1})$ is by assumption a fractional power series unit in preferred coordinates, which concludes the proof of the lemma for $1 \leq \ell \leq L-1$. The endpoint cases $\ell = 0$ and $\ell = L$ are easier and left to the reader.  

It remains to show that the functions in (\ref{show-units}) are units, i.e., bounded from above and below by positive constants in absolute value. Accordingly we estimate on $\widetilde{V}_{\ell, \kappa}$ and for $i \in \mathcal L_k$, 
\[ \begin{aligned} |x_{n+1} - r_i(\mathbf y)| &\leq |x_{n+1}| + C_0 |\mathbf y^{\pmb{\alpha}_k}| \\ &\leq \begin{cases} q_{\ell} |\mathbf y^{\pmb{\alpha}_{\ell}}| + C_0 |\mathbf y^{\pmb{\alpha}_{k}}| \leq (q_{\ell} + 2C_0) |\mathbf y^{\pmb{\alpha}_k}| &\text{ if } k \leq \ell \\  |x_{n+1}| + C_0 |\mathbf y^{\pmb{\alpha}_{\ell+1}}|\leq  \left(1 + \frac{C_0}{p_{\ell}} \right) |x_{n+1}| &\text{ if } k > \ell.  \end{cases} \end{aligned} \]   
On the other hand, 
\[ \begin{aligned}  |x_{n+1} - r_i(\mathbf y)| &\geq \frac{1}{2} \Bigl[ \bigl|x_{n+1} - \text{Re}(r_i(\mathbf y)) \bigr| + \bigl| \text{Im}(r_i(\mathbf y)) \bigr| \Bigr] \\ & \geq \left\{ \begin{aligned} &\frac{1}{2}\bigl| \text{Im}(r_i(\mathbf y)) \bigr| = \frac{1}{2} |\text{Im}(u_i(\mathbf y))| \mathbf y^{\pmb{\alpha}_k} &\text{ if } k \leq \ell, \\ &\frac{1}{2}\bigl|x_{n+1} - \text{Re}(r_i(\mathbf y)) \bigr| \geq \frac{1}{2} \left[ |x_{n+1}| - C_0 |\mathbf y^{\pmb{\alpha}_{\ell+1}}| \right]   &\text{ if } k > \ell,   \end{aligned} \right\} \\ &\geq \left\{ \begin{aligned} &\frac{c_0}{2} |\mathbf y^{\pmb{\alpha}_k}| &\text{ if } k \leq \ell, \\ &  \frac{1}{2} \left(1 - \frac{C_0}{p_{\ell}} \right) |x_{n+1}|  &\text{ if } k > \ell.  \end{aligned} \right\} \end{aligned}  \] 
Thus 
\begin{align*}
\frac{c_0}{2C_0} \leq &\left|\frac{x_{n+1}}{r_i(\mathbf y)} - 1 \right| \leq \frac{q_{\ell} + 2C_0}{c_0} \text{ for } k \leq \ell, \text{ and } \\ 
\frac{1}{2} \left(1 - \frac{C_0}{p_{\ell}} \right) \leq &\left| 1 - \frac{r_i(\mathbf y)}{x_{n+1}}\right| \leq  \left( 1 + \frac{C_0}{p_{\ell}}\right) \text{ for } k > \ell, 
\end{align*} 
which justifies the claim. 
\end{proof}
The following corollary is an analogue of Lemma \ref{imaginary-roots} when the region to be decomposed is a certain subset of $V \times (-1,1)$. 
\begin{corollary} \label{corollary-imroots-0}
Let $V$ and $\varphi$ be as in Proposition \ref{lemma-block2}. Suppose that $F$ obeys the factorization (\ref{F-factid2}) with the ordering (\ref{alpha-ordering}) on an $(n+1)$-dimensional tower $\widetilde{V}_{\kappa} \subseteq V \times (-2,2)$ of the form 
\begin{equation} \widetilde{V}_{\kappa} = \left\{(\mathbf x, x_{n+1}) : 0 < \kappa x_{n+1} < g(\mathbf y), \; \mathbf x = \varphi(\mathbf y) \in V \right\}, \; \kappa = \pm 1 \label{V-tilde-kappa}\end{equation} 
where the roots $r_i$ satisfy the hypotheses of Lemma \ref{imaginary-roots} and $g(\mathbf y) = (\text{unit}) \mathbf y^{\pmb{\beta}}$ is fractional normal crossings on $(0,1)^n$ such that $\{ \pmb{\beta}, \pmb{\alpha}_{1}, \cdots, \pmb{\alpha}_L \}$ is an ordered set of exponents. Then $\widetilde{V}_{\kappa}$ can be decomposed into a finite number of prepared towers on each of which $F$ is fractional normal crossings in any system of preferred coordinates.   
\end{corollary} 
\begin{proof}
Since the combined set of exponents for $r_i$ and $g$ is ordered, there exists a unique index $L_0 \in \{1, \cdots, L\}$ such that \[ \pmb{\alpha}_{\ell} < \pmb{\beta} \text{ for } \ell < L_0, \quad \text{ and } \quad \pmb{\alpha}_{\ell} \geq \pmb{\beta} \text{ for } \ell \geq L_0. \]
We also define positive constants $d_0$ and $D_0$ as follows, \begin{align*} 4d_0 &:= \inf\{|g(\mathbf y)\mathbf y^{-\pmb{\beta}}| : \mathbf y \in (0,1)^n \}, \\ D_0 &:= \sup\{|g(\mathbf y)\mathbf y^{-\pmb{\beta}}| : \mathbf y \in (0,1)^n \}. \end{align*} 

Let us recall the definition of the sets $\widetilde{V}_{\ell, \kappa}$ given in (\ref{Vtilde-lkappa}). Setting \[p_{\ell} = 2C_0 \text{ for } \ell \geq L_0, \quad q_{\ell} = \begin{cases} 3C_0 &\text{ if } \ell \geq L_0+1, \\ 3d_0 &\text{ if } \ell = L_0, \; \pmb{\alpha}_{L_0} = \pmb{\beta}, \\ 3C_0 &\text{ if } \ell = L_0, \;\pmb{\alpha}_{L_0} > \pmb{\beta}, \end{cases} \] we find that the collection $\{ \widetilde{V}_{\ell, \kappa} : \ell \geq L_0 \}$ combined with the two sets $\{ V_{\kappa}^{\sharp}, V_{\kappa}^{\ast}  \}$ form a finite cover for $\widetilde{V}_{\kappa}$, where \begin{align*}
V_{\kappa}^{\sharp} &:= \begin{cases}\left\{(\mathbf x, x_{n+1}) : 2C_0 \mathbf y^{\pmb{\alpha}_{L_0}}< \kappa x_{n+1} < 3d_0 \mathbf y^{\pmb{\beta}}   \right\} &\text{ if } \pmb{\alpha}_{L_0} > \pmb{\beta}, \\ \emptyset &\text{ if } \pmb{\alpha}_{L_0} = \pmb{\beta},  \end{cases} \\
V_{\kappa}^{\ast} &:= \left\{(\mathbf x, x_{n+1}) : 2d_0 \mathbf y^{\pmb{\beta}} < \kappa x_{n+1} < g(\mathbf y) \right\}. 
\end{align*} 
The analysis in Lemma \ref{imaginary-roots} can be repeated verbatim to show that each of the sets $\{\widetilde{V}_{\ell, \kappa} : \ell \geq L_0\}$ and $V_{\kappa}^{\sharp}$ permit a decomposition into a finite number of $(n+1)$-dimensional prepared towers on which $F$ can be reduced to fractional normal crossings for any choice of preferred coordinates. It therefore remains to consider the set $V_{\kappa}^{\ast}$. 

Since $(g(\mathbf y) - 2d_0 \mathbf y^{\pmb{\beta}})$ is fractional normal crossings, we may view $V_{\kappa}^{\ast}$ as a tower which in $(\mathbf y, x_{n+1})$ coordinates is a horn adjacent to $2d_0 \mathbf y^{\pmb{\beta}}$. Let us recall from \S \ref{subsec-adm-pref} the construction of preferred coordinates for such a tower. Setting \begin{equation} \label{y-v} y_i = v_i \text{ for } 1 \leq i \leq n, \quad \text{ and } \quad x_{n+1} = 2d_0 \mathbf y^{\pmb{\beta}} + \left(g(\mathbf y) - 2d_0 \mathbf y^{\pmb{\beta}} \right) v_{n+1},  \end{equation} it suffices to show, as in the proof of Lemma \ref{imaginary-roots}, that each of the factors $(x_{n+1} - r_i(\mathbf y))$ converts to fractional normal crossings in the coordinates $(\mathbf v, v_{n+1})$. In view of Lemma \ref{lemma-fnc}, it is enough to establish that each of these factors is a fractional monomial in $\mathbf y$, multiplied with a bounded nonvanishing function that takes the form of a fractional power series in the variables $(\mathbf v, v_{n+1})$. More specifically, we will show that the functions 
\begin{equation} \left\{ \frac{x_{n+1}}{r_i(\mathbf y)}-1 : i \in \mathcal L_{\ell}, \ell \leq L_0-1 \right\} \quad \text{and} \quad \left\{ \frac{x_{n+1} - r_i(\mathbf y)}{\mathbf y^{\pmb{\beta}}} : i \in \mathcal L_{\ell}, \ell \geq L_0 \right\} \label{show-units2} \end{equation}  
are units in $(\mathbf v, v_{n+1})$. That these are fractional power series in $(\mathbf v, v_{n+1})$ is an easy consequence of (\ref{y-v}), so we concentrate on the bounds.

On $V_{\kappa}^{\ast}$ and for $i \in \mathcal L_{\ell}$ with $\ell \leq L_0-1$, 
\begin{equation} \label{est-belowL_0}
\begin{aligned}
|x_{n+1} - r_i(\mathbf y)| &\leq 2d_0 \mathbf y^{\pmb{\beta}} + C_0 \mathbf y^{\pmb{\alpha}_{\ell}} \leq c_0^{-1}(C_0 + 2d_0)|r_i(\mathbf y)|, \text{ whereas } \\
|x_{n+1} - r_i(\mathbf y)| &\geq |\text{Im}(r_i)(\mathbf y)| \geq \frac{c_0}{2} \mathbf y^{\pmb{\alpha}_{\ell}} \geq \frac{c_0}{2C_0} |r_i(\mathbf y)|.  
\end{aligned} \end{equation} 
On the other hand, for $i \in \mathcal L_{\ell}$ and $\ell \geq L_0$, 
\begin{equation} \label{est-aboveL_0}
\begin{aligned}
|x_{n+1} - r_i(\mathbf y)| &\leq |x_{n+1}| + |r_i(\mathbf y)| \leq (D_0 + C_0) \mathbf y^{\pmb{\beta}}, \text{ while } \\
|x_{n+1} - r_i(\mathbf y)| &\geq \frac{1}{2} \left[ |x_{n+1} - \text{Re}(r_i(\mathbf y))| + |\text{Im}(r_i(\mathbf y))| \right] \\ 
&\geq \left\{ \begin{aligned} &\frac{1}{2} \left( |x_{n+1}| - |\text{Re}(r_i(\mathbf y))| \right)  &\text{ if } \mathbf y^{\pmb{\alpha}_{\ell} - \pmb{\beta}} C_0 < d_0, \\ &\frac{1}{2} |\text{Im}(r_i(\mathbf y))| &\text{ if } \mathbf y^{\pmb{\alpha}_{\ell} - \pmb{\beta}} C_0 \geq d_0  \end{aligned} \right\} \\ 
&\geq \left\{ \begin{aligned} &\frac{1}{2} \left( 2d_0 \mathbf y^{\pmb{\beta}} - C_0 \mathbf y^{\pmb{\alpha}_{\ell}}\right) &\text{ if } \mathbf y^{\pmb{\alpha}_{\ell} - \pmb{\beta}} C_0 < d_0, \\ &\frac{c_0}{2} \mathbf y^{\pmb{\alpha}_{\ell}} &\text{ if } \mathbf y^{\pmb{\alpha}_{\ell} - \pmb{\beta}} C_0 \geq d_0 \end{aligned} \right\} \\
& \geq \min \left( \frac{d_0}{2}, \frac{c_0 d_0}{2C_0} \right) \mathbf y^{\pmb{\beta}}.
\end{aligned} 
\end{equation} 
Combining the estimates in (\ref{est-belowL_0}) and (\ref{est-aboveL_0}) we find that the functions in (\ref{show-units2}) are indeed units, which completes the proof of the corollary.  
\end{proof} 
\begin{corollary} \label{corollary-imroots} 
Let $V$ and $\varphi$ be as in Proposition \ref{lemma-block2}. Let $F$ obey a factorization of the form (\ref{F-factid2}) with the ordering (\ref{alpha-ordering}), where for each $1 \leq i \leq M$, 
\begin{itemize}
\item either Im$(u_i)$ is a unit, 
\item or Re$(u_i)$ is negative (respectively positive).
\end{itemize}   
Then there exists a decomposition of $V \times (0,1)$ (respectively $V \times (-1,0)$) into a finite number of $(n+1)$-dimensional prepared towers on each of which $F$ is fractional normal crossings in the preferred systems of coordinates. 

The same conclusion holds if in the preceding sentence $V \times (0,1)$ (respectively $V \times (-1,0)$) is replaced by the $(n+1)$-dimensional tower $\widetilde{V}_1$ (respectively $\widetilde{V}_{-1}$) given in (\ref{V-tilde-kappa}), where $g(\mathbf y) = (\text{unit}) \mathbf y^{\pmb{\beta}}$ is fractional normal crossings such that $\{\pmb{\beta}, \pmb{\alpha}_1, \cdots, \pmb{\alpha}_L \}$ is ordered.
\end{corollary} 
\begin{proof}
Let us concentrate on obtaining a cover for $V \times (0,1)$, the argument for $V \times (-1,0)$ being identical. We will show that even though the imaginary parts of the functions $u_i$ need no longer be units, the $(n+1)$-dimensional prepared towers obtained from the domains $\widetilde{V}_{\ell, +1}$ defined in (\ref{Vtilde-lkappa}) continue to suffice. Since $F$ is a fractional power series in any set of preferred coordinates on these towers, the desired conclusion will follow from Lemma \ref{imaginary-roots} provided we show the following:  there exists a function $H$ satisfying the hypotheses of this lemma for which $F/H$ is a bounded function that is also bounded away from zero
on $V \times (0,1)$.  

Define the index set $\mathbb I^{\ast} := \{i : \text{Im}(u_i) \text{ is a unit} \}$. By the hypothesis, Re$(u_i)$ is negative and a unit for $i \not\in \mathbb I^{\ast}$. The last condition means that 
\begin{equation} C_1 := \sup \left\{\left| \frac{\text{Im}(r_i(\mathbf y))}{\text{Re}(r_i(\mathbf y))} \right| : \mathbf y \in (0,1)^n,\, i \not\in \mathbb I^{\ast} \right\} < \infty. \label{def-C1} \end{equation}  
We also define for $i \not\in \mathbb I^{\ast}$ a (purely imaginary) function $\varrho_i$ such that \[ \text{Re}(\varrho_i) \equiv 0 \quad \text{ and } \quad \text{Im}(\varrho_i) = \text{Re}(r_i). \]
For such $i$ and for $x_{n+1} > 0$,     
\begin{align*}
|x_{n+1} - r_i(\mathbf y)|^2 &= |x_{n+1} - \text{Re}(r_i(\mathbf y))|^2 + |\text{Im}(r_i(\mathbf y))|^2 \\ &= \bigl(x_{n+1} + \bigl| \text{Re}(r_i(\mathbf y))\bigr| \bigr)^2 + |\text{Im}(r_i(\mathbf y))|^2, \text{ so that } \\ 
|x_{n+1} - r_i(\mathbf y))| &\geq |x_{n+1} - \varrho_i(\mathbf y)|, \text{ and } \\  |x_{n+1} - r_i(\mathbf y))|^2  &\leq 2\Bigl[x_{n+1}^2 + \bigl(\text{Re}(r_i(\mathbf y)) \bigr)^2 \Bigr] + |\text{Im}(r_i(\mathbf y))|^2 \\ &\leq (C_1^2 +2) |x_{n+1} - \varrho_i(\mathbf y)|^2. 
\end{align*} 
It follows from the calculations above that 
\begin{align*} 1 \leq \left|  F(\mathbf x, x_{n+1}) \bigl[H(\mathbf x, x_{n+1}) \bigr]^{-1} \right| \leq (C_1^2 +2)^{\frac{N_V}{2}}, \text{ where } \\  
H(\mathbf x, x_{n+1}) = \Bigl[ \prod_{i \in \mathbb I^{\ast}} \bigl(x_{n+1} - r_i(\mathbf y) \bigr) \Bigr] \Bigl[ \prod_{i \not\in \mathbb I^{\ast}} \bigl(x_{n+1} - \varrho_i(\mathbf y) \bigr) \Bigr]
\end{align*}
satisfies the hypotheses of Lemma \ref{imaginary-roots}. The first conclusion of the corollary is now a consequence of this lemma applied to $H$. The statement for $\widetilde{V}_1$ follows from Corollary \ref{corollary-imroots-0} applied to $H$.   
\end{proof} 
\begin{lemma} \label{real-roots} 
Let $V$ and $\varphi$ be as in Proposition \ref{lemma-block2}. Suppose that $F$ obeys the factorization (\ref{F-factid2}), where all the roots $r_i$ are real fractional normal crossings whose differences are also fractional normal crossings. Then the conclusion of Proposition \ref{lemma-block2} holds.   
\end{lemma} 
\begin{proof}
Since the roots are real and the differences of the roots are fractional normal crossings on $(0,1)^n$, the set of roots is pointwise strictly ordered and each root is of a fixed sign on $(0,1)^n$. Namely there exist subsets of distinct indices $\mathcal I = \{ i_1,, i_2, \cdots i_{M_1}\}$ and $\mathcal J = \{ j_1, \cdots, j_{M_2} \}$, $M_1 + M_2 = M$ such that 
\begin{align*} 
\{1, 2 \cdots, M \} &= \mathcal I \cup \mathcal J , \text{ and } \\  
r_{j_{M_2}} < r_{j_{M_2-1}} < \cdots < r_{j_1} <  &0  < r_{i_1} < \cdots < r_{i_{M_1}} \; \text{ on } (0,1)^n. 
\end{align*}  
Set $i_0 = j_0 = 0$, $r_0 \equiv 0$ and consider the regions
\begin{align}
&\bigl\{ (\mathbf x, x_{n+1}) : 0 < x_{n+1} < \frac{r_{i_1}(\mathbf y)}{2}, \, \mathbf x = \varphi(\mathbf y) \in V \bigr\}, \label{region-main1}  \\ 
&\bigl\{ (\mathbf x, x_{n+1}) : \frac{r_{j_1}(\mathbf y)}{2} < x_{n+1} < 0, \, \mathbf x = \varphi(\mathbf y) \in V  \bigr\}, \label{region-main2} \\ 
&\bigl\{(\mathbf x, x_{n+1}) : r_{i_k}(\mathbf y) < x_{n+1} < \frac{1}{2}(r_{i_k} + r_{i_{k+1}})(\mathbf y), \, \mathbf x \in V \bigr\}, \; 1 \leq k \leq M_1-1, \label{region3} \\ 
& \bigl\{(\mathbf x, x_{n+1}) : \frac{1}{2}(r_{i_k} + r_{i_{k-1}})(\mathbf y) < x_{n+1} < r_{i_{k}}(\mathbf y), \, \mathbf x \in V \bigr\}, \; 1 \leq k \leq M_1, \label{region4} \\ 
&\bigl\{(\mathbf x, x_{n+1}) : \frac{1}{2}(r_{j_k} + r_{j_{k+1}})(\mathbf y)  < x_{n+1} < r_{j_k}(\mathbf y), \, \mathbf y \in V \bigr\}, \; 1 \leq k \leq M_2-1, \label{region5} \\
&\bigl\{(\mathbf x, x_{n+1}) : r_{j_{k}}(\mathbf y)  < x_{n+1} < \frac{1}{2}(r_{j_k} + r_{j_{k-1}})(\mathbf y), \, \mathbf y \in V \bigr\}, \; 1 \leq k \leq M_2, \label{region6} \\
&\bigl\{ (\mathbf x, x_{n+1}) : r_{i_{M_1}}(\mathbf y) < x_{n+1} < 1, \, \mathbf y \in V  \bigr\}, \label{region7} \\
&\bigl\{ (\mathbf x, x_{n+1}) : -1  < x_{n+1} < r_{j_{M_2}}(\mathbf y), \, \mathbf y \in V  \bigr\} \label{region8},
\end{align} 
whose union covers $V \times (-1,1)$, except possibly a set of dimension $<n$. We claim that each of these regions is an $(n+1)$-dimensional tower. Indeed, it is clear that in the coordinates $(\mathbf y, x_{n+1})$, the regions (\ref{region-main1}) and (\ref{region-main2}) are horns adjacent to zero, whereas (\ref{region3}) and (\ref{region4}) are horns adjacent to $r_{i_k}$. Further, the projection of any of domains above onto the first $n$ coordinates is $V$, an admissible generalized coordinate image of an $n$-dimensional tower, which justifies the claim.

Next we argue that (\ref{region3}) can be treated the same way as (\ref{region-main1}), and (\ref{region4}) the same way as (\ref{region-main2}). Indeed by a change of variable 
\begin{equation}(\mathbf y, x_{n+1}) \mapsto (\mathbf y, y_{n+1}) \quad \text{ where } \quad y_{n+1} = x_{n+1} - r_{i_k}(\mathbf y), \label{changeofvar-reduction} \end{equation} we find that the towers (\ref{region3}) and (\ref{region4}) reduce to the form (\ref{region-main1}) and (\ref{region-main2}) respectively, and that $F$, by virtue of our assumption, transforms to a function which continues to admit a factorization of the form (\ref{F-factid2}) in these new coordinates. The regions (\ref{region5})--(\ref{region8}) can also be reduced to either (\ref{region-main1}) or (\ref{region-main2}) by similar shift transformations. In order to complete the proof of the lemma it therefore suffices to cover each of the regions (\ref{region-main1}) and (\ref{region-main2}) by a finite number of prepared towers on each of which $F$ is fractional normal crossings in any preferred system of coordinates.  

With this in mind, we estimate each factor in (\ref{F-factid2}) as follows. 
On the region (\ref{region-main1}) and for $i \in \mathcal I$, 
\begin{equation} 
\frac{r_i(\mathbf y)}{2} \leq r_i(\mathbf y) - x_{n+1} = \bigl| x_{n+1} - r_i(\mathbf y) \bigr| 
\leq \frac{r_{i_1}(\mathbf y)}{2} + r_i(\mathbf y) \leq \frac{3}{2} r_i(\mathbf y), 
\label{realroots-est1} \end{equation}
 while for $i \in \mathcal J$ we have \begin{equation} r_i < 0 \quad \text{ and hence } \quad
 |x_{n+1} - r_i(\mathbf y)| = x_{n+1} - r_i(\mathbf y). \label{realroots-est2}  \end{equation} 
In view of (\ref{F-factid2}), (\ref{realroots-est1}) and (\ref{realroots-est2}), we define a function $H = H_1 H_2$, where
\[ H_1(\mathbf x, x_{n+1}) = \prod_{i \in \mathcal I} r_i(\mathbf y),\quad H_2(\mathbf x, x_{n+1}) = x_{n+1}^{\beta_{n+1}}  \Bigl[ \prod_{i \in \mathcal J} \bigl( x_{n+1} - r_i(\mathbf y \bigr) \Bigr]. \]
We observe that $H_2$ satisfies the hypothesis and hence the conclusion of Corollary \ref{corollary-imroots}. Thus the tower (\ref{region-main1}) can be decomposed into a finite number of prepared towers on each of which $H_2$ is expressible as fractional normal crossings in the preferred coordinates. Since $H_1$ is already fractional normal crossings in $\mathbf y$, by Lemma \ref{lemma-fnc} it continues to be so in the preferred coordinates on each such prepared tower. Thus the desired conclusion of the lemma holds for the function $H$. By (\ref{realroots-est1}) and (\ref{realroots-est2}), $F H^{-1}$ is a unit in the preferred coordinates, completing the proof for (\ref{region-main1}). A similar calculation can be verified on the region (\ref{region-main2}) and is left to the reader. 
\end{proof} 
\subsection{Proof of Proposition \ref{lemma-block2}} \label{subsec-proof-lemma-block2} 
The general proof is essentially a combination of the techniques used to prove Lemmas \ref{imaginary-roots} and \ref{real-roots}. Assume that $F$ obeys the factorization (\ref{F-factid2}), with the hypotheses that the sets $\mathcal A_1, \mathcal A_2$ and the difference sets $\mathcal A_i - \mathcal A_j$, $i, j = 1,2$ all consist of fractional normal crossings and possibly the constant zero function, where $\mathcal A_1 = \{r_i : 1 \leq i \leq M \}$ and $\mathcal A_2 = \{\text{Re}(r_i) : 1 \leq i \leq M \}$. Following the proof of Corollary \ref{corollary-imroots} we define the index set
\[ \mathbb I^{\ast} = \left\{1 \leq i \leq M : \text{Im}(u_i) \text{ is a unit on } (0,1)^n \right\}, \]
and the constant $C_1$ as in (\ref{def-C1}). 
For $i \not\in \mathbb I^{\ast}$, we know that Re$(r_i) \not\equiv 0$, since otherwise Im$(u_i) = u_i$ would be a unit. Further Re$(r_i)$ is fractional normal crossings, so it does not change sign on $(0,1)^n$. Accordingly as in Lemma \ref{real-roots}, we decompose $\{1, \cdots, M \} \setminus \mathbb I^{\ast} = \mathcal I \cup \mathcal J$, where Re$(r_i) > 0$ for $i \in \mathcal I$ and Re$(r_i) < 0$ for $i \in \mathcal J$. Since the nontrivial elements of $\mathcal A_2 - \mathcal A_2$ are fractional normal crossings, we can order the distinct elements in $\{ \text{Re}(r_i) : i \not\in \mathbb I^{\ast} \}$. Thus there are index sets $\mathcal I' = \{ i_1, \cdots, i_{M_1} \} \subseteq \mathcal I$, $\mathcal J' = \{j_1, \cdots, j_{M_2} \} \subseteq \mathcal J$, such that
\[ \begin{aligned}  \{\text{Re}(r_i&): i \in \mathcal I \} = \{\text{Re}(r_i) : i \in \mathcal I' \}, \\ \{\text{Re}(r_j&): j \in \mathcal J \} = \{\text{Re}(r_j) : j \in \mathcal J' \}, \\ \text{Re}(r_{i}) &\not\equiv \text{Re}(r_{i'}) \text{ for } i, i' \in \mathcal I', i \ne i', \\ \text{Re}(r_{j}) &\not\equiv \text{Re}(r_{j'}) \text{ for } j, j' \in \mathcal J', j \ne j', \text{ and } \\  \text{Re}(r_{j_{M_2}}) < &\cdots < \text{Re}(r_{j_1}) < 0 < \text{Re}(r_{i_1}) < \cdots < \text{Re}(r_{i_{M_1}}). \end{aligned} \]   
We also consider as in Lemma \ref{real-roots} the regions (\ref{region-main1})--(\ref{region8}), but with $r_i$ replaced by Re$(r_i)$. We observe that by part (\ref{part2}) of the hypotheses of the proposition, the change of variables \begin{equation} \label{changeofvar-reduction2} (\mathbf y, x_{n+1}) \mapsto (\mathbf y, y_{n+1}), \quad y_{n+1} = x_{n+1} - \text{Re}(r_{i_k})(\mathbf y) \end{equation}  converts $F$ to a form that continues to obey the assumptions of Proposition \ref{lemma-block2}, while mapping (\ref{region3}) and (\ref{region4}) to regions of the form (\ref{region-main1}) and (\ref{region-main2}) respectively. Similar arguments hold for the other regions as well. We therefore reduce to considering only the regions (\ref{region-main1}) and (\ref{region-main2}).

On (\ref{region-main1}) and for $i \in \mathcal I$, the estimate (\ref{realroots-est1}) can be used to obtain    
\[ \frac{\text{Re}(r_{i}(\mathbf y))}{2} \leq |x_{n+1} - \text{Re}(r_{i})(\mathbf y)| \leq \frac{3}{2} \text{Re}(r_i(\mathbf y)), \] 
so that 
\begin{equation} 
\begin{aligned} 
|x_{n+1} - r_i(\mathbf y)|^2 &= |x_{n+1} - \text{Re}(r_i(\mathbf y))|^2 + |\text{Im}(r_i(\mathbf y))|^2 \\ &\geq  |x_{n+1} - \text{Re}(r_i(\mathbf y))|^2 \geq \Bigl[\frac{\text{Re}(r_{i}(\mathbf y))}{2} \Bigr]^2. \end{aligned} 
\label{general-est1} \end{equation} 
On the other hand, in conjunction with (\ref{def-C1}), we also arrive at the estimate
\begin{equation} \begin{aligned} |x_{n+1} - r_i(\mathbf y)|^2 &= |x_{n+1} - \text{Re}(r_i(\mathbf y))|^2 + |\text{Im}(r_i(\mathbf y))|^2 \\ &\leq \left(\frac{9}{4} + C_1^2 \right) \bigl[ \text{Re}(r_i(\mathbf y)) \bigr]^2. \end{aligned} \label{general-est2} \end{equation}
Define a function $H = H_1 H_2$, where 
\[ H_1(\mathbf x, x_{n+1}) = \prod_{i \in \mathcal I} \text{Re}(r_i(\mathbf y)),\quad H_2(\mathbf x, x_{n+1}) = x_{n+1}^{\beta_{n+1}}  \Bigl[ \prod_{i \in \mathcal J \cup \mathbb I^{\ast}} \bigl( x_{n+1} - r_i(\mathbf y \bigr) \Bigr]. \]  
The estimates (\ref{general-est1}) and (\ref{general-est2}) imply that $FH^{-1}$ is a unit when expressed in any set of preferred coordinates. Since $H_1$ is fractional normal crossings by assumption (ii) and $H_2$ satisfies the hypotheses of Corollary \ref{corollary-imroots} on the tower (\ref{region-main1}), we are done.  

\subsection{Proof of Theorem \ref{mainthm-resolution}} \label{subsec-mainthm-resolution-proof}
The proof is by induction on $n$. For the initializing step $n=1$, 
the properties listed in parts (\ref{parta})--(\ref{partd}) of Theorem \ref{mainthm-resolution} may be verified directly from the Newton-Puiseux lemma involving the Puiseux factorization of a bivariate real-analytic function. 

For the inductive step, we assume that the conclusions of the theorem hold for all real-analytic functions in $\leq n$ variables. Let $F$ be a real-analytic function near the origin in $\mathbb R^{n+1}$, for which we construct the auxiliary function $\Lambda$ as in \S \ref{Construction-Lambda}. Since $\Lambda$ is an $n$-variate real-analytic function, we apply the induction hypothesis on $\Lambda$, after a nonsingular linear coordinate transformation in $\mathbb R^n$ if necessary so as to ensure $\Lambda(\mathbf 0', x_n) \not\equiv 0$. By the previous step of the induction, there exists a small constant $\epsilon > 0$ and a finite collection of $(n-1)$-dimensional prepared towers whose union covers $(-\epsilon, \epsilon)^{n-1}$ except possibly a lower dimensional subset, and on each of which $\Lambda$ admits a factorization of the form (\ref{F-factid}) in preferred coordinates. Fix such a tower $U'$, and let $\Psi'$ be the system of preferred coordinates on $U'$. We write 
\[ \Lambda(\mathbf x', x_n) = (\text{unit}) \prod_{i=1}^{N'} \bigl(x_n - \lambda_i(\mathbf u') \bigr), \quad \mathbf x'= \Psi'(\mathbf u'), \; \mathbf u' \in (0,1)^{n-1}, \]
where the roots $\lambda_i$ of $\Lambda$ are fractional normal crossings, and hence Re$(\lambda_i)$ are fractional power series. Let $\Xi_0(\mathbf u', x_n)$ be a monic polynomial in $x_n$ with real-analytic coefficients in $\mathbf u'$ such that $\{ \lambda_i \}$, $\{ \text{Re}(\lambda_i) \}$ and the coordinate functions $\{ u_1, \cdots, u_{n-1} \}$ are roots of $\Xi_0$. The existence of $\Xi_0$ is ensured by Lemma \ref{lemma-integral-closure}. 

Since $\Psi'$ is a vector-valued fractional power series by definition, there exists a large integer $R$ such that $\Psi' \circ \Phi'_R$ is a vector-valued regular power series, where $\Phi_R'$ is the power map defined in (\ref{power-map}) but in $(n-1)$ dimensions. Set $\mathbf u' = \Phi_R'(\mathbf w')$. It is important to note that $\Psi' \circ \Phi_R'(\mathbf w')$ is no longer a coordinate system on the tower, though it is a generalized one, since 
\begin{equation} d\mathbf x' = (\text{unit}) d\mathbf u' = (\text{unit}) \mathbf w'^{R-1} d\mathbf w'. \label{Jac-chase}\end{equation}  
By the induction hypothesis again but now applied to the $n$-variate real-analytic function $(\mathbf w', x_n) \mapsto \Xi_0(\Phi_R'(\mathbf w'), x_n)$, we can decompose the orthant $(0,1)^{n-1}$ in $\mathbf w'$-space into a finite number of prepared towers $\{(W', \eta_{W'}) : W' \in \mathcal W' \}$ such that 
\[ \Xi_0(\Phi_R'(\mathbf w'), x_n) = \prod_{i} \bigl(x_n - \xi_i(\mathbf t') \bigr), \quad \mathbf w' = \eta_{W'}(\mathbf t') \in W', \; \mathbf t' \in (0,1)^{n-1}, \] 
where the roots $\xi_i$ and their differences are either trivial or fractional normal crossings. 

Since the collection $\{ \xi_i \}$ contains the monomials $\{ w_j^R : 1 \leq j \leq n-1 \}$ expressed in the new coordinates $\mathbf t'$, we conclude that the coordinate functions $w_j$ are fractional normal crossings in these new variables. Thus there exist $r_j > 0$, $1 \leq j \leq n-1$ such that 
\[ \mathbf w'^{R-1} = (\text{unit}) \prod_{j=1}^{n-1} t_j^{r_j- 1}. \] Since $\eta_{W'}$ is a coordinate system on $W'$ by our induction assumption, $d\mathbf w'= (\text{unit}) d \mathbf t'$, hence (\ref{Jac-chase}) implies that \[ d\mathbf x' = (\text{unit}) \prod_{j=1}^{n-1} t_j^{r_j- 1} d\mathbf t'. \]
Finally, setting $\mathbf y' = \Phi'_{\mathbf r'}(\mathbf t') = (t_1^{r_1}, \cdots, t_{n-1}^{r_{n-1}})$ we find that \begin{equation} d\mathbf x' = (\text{unit}) d\mathbf y'. \label{coord-check} \end{equation} Recalling that $\Psi' \circ \Phi_R'$ is a regular power series, $\eta_{W'}$ is a fractional power series and $\Phi'_{\mathbf r'}$ is a power transformation, we conclude (see Lemma \ref{lemma-coords-comp2}) that $\varphi' = \Psi' \circ \Phi_R' \circ \eta_{W'} \circ (\Phi'_{\mathbf r'})^{-1}$ is indeed a vector-valued fractional power series, hence in view of (\ref{coord-check}) an admissible coordinate system onto its image.       

Further, the collection $\{ \xi_i \}$ contains all roots of $\Lambda$ as well as their real parts, therefore we find that the hypotheses of Proposition \ref{lemma-block2} are satisfied with $n$, $F$, $V$ and $\varphi$ replaced by $n-1$, $\Lambda$, $V' = \Psi'\circ \Phi_R'(W')$ and $\varphi' =  \Psi' \circ \Phi_R' \circ \eta_{W'} \circ (\Phi_{\mathbf r'}')^{-1}$ respectively. It follows from Proposition \ref{lemma-block2} therefore that $V' \times (-1,1)$ can be decomposed into a finite number of $n$-dimensional prepared towers $\{(V, \varphi_V) : V \in \mathcal V \}$ on each of which $\Lambda$ is fractional normal crossings in any preferred system of coordinates. The construction given in the proof of Proposition \ref{lemma-block2} shows that the defining functions of these towers depend on the roots of $\Lambda$ but not on the roots of $F$. Thus the towers are of low resolution complexity with respect to $F$ but in general non-constructible with respect to $\Lambda$. On each such $V$ the hypotheses of Proposition \ref{lemma-block1} hold with $\sigma$ replaced by $\varphi_V$. Invoking this proposition yields the conclusions (\ref{parta})--(\ref{partd}) of Theorem \ref{mainthm-resolution} in dimension $(n+1)$, completing the induction.   
\qed

\subsection{Proof of Corollary \ref{corollary-mainthm-resolution}}
The proof is essentially a repetition of an argument given in the proof of Theorem \ref{mainthm-resolution}, so we only sketch the details. Fix $U$ and let $\{\varrho_i \}$ be the fractional power series roots of $F$ expressed in the coordinates $\mathbf u$, where $\mathbf x = \Psi(\mathbf u)$. Since Re$(\varrho_i)$ is a fractional power series for each $i$, there exists by Lemma \ref{lemma-integral-closure} a monic polynomial $F_0(\mathbf u, x_{n+1})$ in $x_{n+1}$ with real-analytic coefficients in $\mathbf u$ whose roots are $\{ \varrho_i \}$, $\{ \text{Re}(\varrho_i) \}$ and the coordinate functions $\{ u_1, \cdots, u_n \}$. 

Let $R$ be a large positive integer such that $\varphi \circ \Phi_R$ is a regular power series, and set $\mathbf u = \Phi_R(\mathbf w)$. By Theorem \ref{mainthm-resolution} applied to $F_0(\Phi_R(\mathbf w), x_{n+1})$, we decompose $(0,1)^n$ into a finite number of $n$-dimensional prepared towers $\{(W, \eta_W) : W \in \mathcal W \}$ such that 
\[ F_0(\Phi_R(\mathbf w), x_{n+1}) = \prod_{i} (x_{n+1} - \omega_i(\mathbf t)), \quad \mathbf w = \eta_W(\mathbf t) \in W, \; \mathbf t \in (0,1)^n \]
where the roots $\{ \omega_i \}$ and their differences are fractional normal crossings in $\mathbf t$. Since \[ \mathcal A = \{ \varrho_i \circ \Phi_R \circ \eta_W, \;\text{Re}(\varrho_i) \circ \Phi_R \circ \eta_W  \} \subseteq \{ \omega_i \}, \] every nontrivial element of $\mathcal A$ and $\mathcal A - \mathcal A$ must be fractional normal crossings. By setting $\mathbf y = \Phi_{\mathbf r}(\mathbf t) = (t_1^{r_1}, \cdots, t_n^{r_n})$, where $\mathbf r$ is a vector of positive entries judiciously chosen so that $d \mathbf x = (\text{unit}) d\mathbf y$, we find that the desired conclusion is satisfied with \begin{equation} V = \Psi \circ \Phi_R(W) \subseteq U \qquad \text{ and } \qquad \varphi = \Psi \circ \Phi_R \circ \eta_W \circ (\Phi_{\mathbf r})^{-1}.  \label{set-coord-pair}\end{equation}      
\qed

{\bf{Remark: }} Summarizing the discussion in this section, we obtain the following statement. Given a real-analytic function $F$ of $(n+1)$ variables with $F(\mathbf 0, 0) = 0$, one can obtain a decomposition of $(-\epsilon, \epsilon)^{n}$ into finitely many sets $\{V\}$ and construct for each set an associated coordinate system $\varphi$ such that the roots of $F(\varphi(\mathbf y), x_{n+1})$ have the fine structure stated in part (\ref{corollary-partc}) of Corollary \ref{corollary-mainthm-resolution}. This involves a two-step procedure. 
\begin{itemize}
\item In the first step, one monomializes $\Lambda = \Lambda_F$, obtaining a decomposition of $(- \epsilon, \epsilon)^{n}$ into a finite number of prepared towers $\{ U \in \mathcal U\}$ such that the conclusions of Theorem \ref{mainthm-resolution}, in particular the factorization (\ref{F-factid}), hold for $F$ on each $U$ and in preferred coordinates. 
\item In the second step, each tower $U$ gives rise to an $(n+1)$-variate real-analytic function $F_0 = F_0(U)$. Monomializing $H_U(\mathbf u) := \Lambda_{F_0(U)}(\mathbf u)$ (i.e. a second application of Theorem \ref{mainthm-resolution}, on $F_0(U)$) leads to a finer decomposition of $U$ and the desired representation of $F$. The sets $\{V : V \in \mathcal V(U) \}$ in the final decomposition are explicitly determined, as seen in the proofs of Corollary \ref{corollary-mainthm-resolution} and Proposition \ref{lemma-block2}, and of low resolution complexity with respect to $F$ and $F_0(U)$. 
\end{itemize}
Note that if $(V, \varphi)$ is a set-coordinate pair obtained at the end of this process (as in (\ref{set-coord-pair})), and if $r$ is any real root of $F(\varphi(\mathbf y), x_{n+1})$ with respect to $x_{n+1}$, then the coordinate system $\Phi = \Phi(\varphi, V, r)$ lies in $\mathcal C$, where $\Phi$ and $\mathcal C$ are as in the statement of Theorem \ref{main-thm1}. In \S \ref{sec-proof-mainthm1}, the construction above will be used to obtain the finite collection $\mathcal C^{\ast}$ of coordinate systems mentioned in this theorem. See also remark \ref{remark-Cstar} following Theorem \ref{main-thm1}.

\section{Examples} \label{sec-example}
We pause briefly to elucidate the resolution of singularities algorithm described in the previous section in the context of a few concrete functions $F$.  

\subsection{Example 1} Let \[ F(x_1, x_2, x_3) = x_3^2 + a x_2^{2N} x_3 + b x_1^{4M}, \]
where $a$, $b$ are positive constants, $M, N$ are positive integers. Then in the notation of (\ref{cGl}) and (\ref{def-Lambda}), 
\[ c(\mathbf x) = 4 b x_1^{4M}, \quad \Delta_P(\mathbf x) = (a x_2^{2N})^2 - 4b x_1^{4M}, \] 
so that 
\[ \Lambda(\mathbf x) = a^2 b x_1^{4M} \left(x_2^{4N} - \frac{4 b}{a^2} x_1^{4M} \right). \]    
We aim to verify the conclusions of Theorem \ref{mainthm-resolution} for this $F$. Namely, we will introduce a finite collection of local coordinates on $(x_1, x_2)$ space in which $\Lambda$ is fractional normal crossings, then establish that the roots of $F$ in $x_3$ as well as their differences are fractional normal crossings in these coordinates. We will also specify the Newton {\atxt polyhedra} of $F$ in these coordinates and compute the Newton  {\atxt exponents}.  

We decompose the positive orthant in $(x_1, x_2)$ space into four regions. Defining $c$ by the identity $c^{4N} = 4b/a^2$, we set    
\begin{align*}
V_1 &= \left\{(x_1, x_2) : 0 < x_2 < \frac{c}{2} x_1^{\frac{M}{N}} \right\},  \\
V_2 &= \left\{(x_1, x_2) : \frac{c}{2} x_1^{\frac{M}{N}} < x_2 < c x_1^{\frac{M}{N}} \right\}, \\
V_3 &= \left\{(x_1, x_2) : c x_1^{\frac{M}{N}} < x_2 < \frac{3c}{2} x_1^{\frac{M}{N}} \right\}, \\
V_4 &= \left\{(x_1, x_2) : \frac{3c}{2} x_1^{\frac{M}{N}} < x_2 < 1 \right\}. \\ 
\end{align*} 
We describe the choice of coordinates on the domains $V_1$, $V_3$ and $V_4$, the treatment for $V_2$ being similar to $V_3$.

On $V_1$, the discriminant of $F$ is negative, so there are two non-real roots that are complex conjugates of each other. Making the admissible change of variables 
\[ (x_1, x_2) = \varphi_{V_1}(y_1, y_2) \qquad \text{ where } \qquad x_1= y_1^{\frac{N}{M+N}}, x_2 = \frac{c}{2} y_1^{\frac{M}{M+N}} y_2, \]
we note that the hypotheses of Proposition \ref{lemma-block1} are met, and that the two roots $r_1$, $r_2$ of $F$ in these coordinates are 
\begin{align*}
x_3 &= -\frac{a}{2} x_2^{2N} \pm \frac{1}{2} \sqrt{a^2 x_2^{4N} - 4b x_1^{4M}} \\ &= - \frac{\sqrt{b}}{2^{2N}} y_1^{\frac{2MN}{M+N}} y_2^{2N} \pm i \sqrt{b} y_1^{\frac{2MN}{M+N}} \sqrt{1 - \frac{y_2^{4N}}{2^{4N}}} \\ &= \pm i \sqrt{b} y_1^{\frac{2MN}{M+N}} \left[ \sqrt{1 - \frac{y_2^{4N}}{2^{4N}}} \pm i \frac{y_2^{2N}}{2^{2N}} \right]. 
\end{align*}
Thus $r_1$, $r_2$ and $r_1-r_2$ are fractional normal crossings in the coordinates $(y_1, y_2)$, as claimed in Proposition \ref{lemma-block1} and Theorem \ref{mainthm-resolution}. In fact, if $\mathcal A = \{ r_i, \text{Re}(r_i) :  i=1,2 \}$, then every element of $\mathcal A$ and $\mathcal A - \mathcal A$ is either identically zero or fractional normal crossings. Thus the coordinate transformation $\Phi_{V_1}$ given by $\Phi_{V_1}(\mathbf y, y_{n+1}) = (\varphi_{V_1}(y_1, y_2), x_{3})$ lies in the coordinate class $\mathcal C(F)$ defined in the statement of Theorem \ref{main-thm1}.  Moreover, the Newton polyhedron of $F$ in the system of coordinates $\Phi$ is given by the monotone edge path connecting the points $(0,0,2)$ and $(4MN/(M+N), 0, 0)$. Thus $\delta_0(F;\Phi_{V_1}) = \frac{1}{2} + \frac{1}{4N} + \frac{1}{4M}$. 
   
On $V_3$, both roots are real and distinct. In the coordinates \[ (x_1, x_2) = \varphi_{V_2}(y_1, y_2), \qquad x_1 = y_1^{\frac{N}{M+N}}, x_2 = c y_1^{\frac{M}{M+N}} \left(1 + \frac{3y_2}{2} \right)  \] 
these roots take the form 
\[ x_3 = \sqrt{b} y_1^{\frac{2MN}{M+N}} \left[ - \left( 1 + \frac{3y_2}{2}\right)^{2N} \pm \sqrt{\left( 1 + \frac{3y_2}{2}\right)^{4N} - 1} \right], \]
which are again fractional normal crossings. Further, if $\Phi_{V_3}(\mathbf y) = (\varphi_{V_2}(y_1, y_2), x_3)$, then $\Phi_{V_3} \in \mathcal C(F)$, and the Newton diagram and Newton {\atxt exponent}  of $F$ under $\Phi_{V_3}$ remain the same as for $V_1$. 

On $V_4$, the roots continue to be real and distinct. We set  
\[ (x_1, x_2) = \varphi_{V_4}(y_1, y_2) \quad \text{ where } \quad x_1 = \left( \frac{2}{3c} \right)^{\frac{N}{M}} y_1 y_2^{\frac{N}{M+N}}, \qquad x_2 = y_2^{\frac{M}{M+N}}. \]
The roots now are 
\begin{align*} r_1 &= -\frac{a}{2} y_2^{\frac{2MN}{M+N}} \left[1 + \sqrt{1 - \left( \frac{2}{3}\right)^{4N} y_1^{4M}} \right], \\
r_2 &= -\frac{\frac{a}{2} \left( \frac{2}{3} \right)^{4N} }{1 + \sqrt{1 - \left( \frac{2}{3}\right)^{4N} y_1^{4M} }} y_1^{4M} y_2^{\frac{2MN}{M+N}},  
\end{align*}   
which are again fractional normal crossings but with distinct leading exponents, unlike the previous two cases. Set as before $\Phi_{V_4}(y_1, y_2, y_3) = (\varphi_{V_4}(y_1, y_2), x_3)$, so that $\Phi_{V_4} \in \mathcal C(F)$. The Newton diagram of $F$ in coordinates $\Phi_{V_4}$ is defined by the monotone edge path obtained by joining the points $(0,0,2)$, $(0, 2MN/(M+N), 1)$ and $(4M, 4MN/(M+N), 0)$. Even though the monotone edge path now consists of two distinct segments, a straightforward calculation shows that $\delta_0(F;\Phi_{V_4}) = \frac{1}{2} + \frac{1}{4N} + \frac{1}{4M}$, as before.   

\subsection{Example 2} The previous example had the simplified feature that the coordinate changes $\varphi_{V_i}$ which converted the roots $\{r_i\}$ and their differences $\{ r_i - r_j \}$ into fractional normal crossings had the same effect on $\{ \text{Re}(r_i) \}$ and $\{ \text{Re}(r_i) - \text{Re}(r_j)\}$ as well. This need not always be the case, and a finer decomposition of the sets $V_i$ as specified by Corollary \ref{corollary-mainthm-resolution} may in general be necessary to meet these additional restrictions, as the following example shows. Let
\[F(x_1, x_2, x_3) = (x_3-x_1)(x_3-x_1 + x_1^2 x_2 - x_1 x_2^2 + i x_1 x_2),\]
so that the two roots $r_1, r_2$ satisfy 
\[ r_1 = x_1, \quad r_2 = x_1 - x_1^2 x_2 + x_1 x_2^2 - i x_1 x_2, \quad r_2 - r_1 = x_1^2 x_2 - x_1 x_2^2 + i x_1 x_2. \] 
Each of the functions above is fractional normal crossings, but 
\[ \text{Re}(r_2) - \text{Re}(r_1) = x_1x_2^2 - x_1^2 x_2 \]
is not. Thus the ambient coordinate system $(x_1, x_2)$ satisfies the conclusion of Theorem \ref{mainthm-resolution} but does not lie in $\mathcal C(F)$.  

However, by decomposing the positive orthant into sets $\{ U_i : 1 \leq i \leq 4 \}$ equipped with coordinates $\{ \varphi_{U_i} : 1 \leq i \leq 4 \}$, where 
\begin{align*}
U_1 &= \Bigl\{(x_1, x_2) : 0 < x_2 < \frac{1}{2} x_1 \Bigr\}, \quad \varphi_{U_1}(y_1, y_2) =  \Bigl(\sqrt{y_1}, \frac{1}{2} y_2 \sqrt{y_1} \Bigr), \\
U_2 &= \Bigl\{(x_1, x_2) : \frac{1}{2}x_1 < x_2 < x_1 \Bigr\}, \quad \varphi_{U_2}(y_1, y_2) = \Bigl(\sqrt{y_1}, \sqrt{y_1} (1 - \frac{y_2}{2})  \Bigr), \\
U_3 &= \Bigl\{(x_1, x_2) : x_1 < x_2 < \frac{3}{2}x_1 \Bigr\}, \quad \varphi_{U_3}(y_1, y_2) = \Bigl(\sqrt{y_1}, \sqrt{y_1} \bigl(1 + \frac{3y_2}{2} \bigr) \Bigr), \\
U_4 &= \Bigl\{(x_1, x_2) : \frac{3}{2}x_1 < x_2 < 1 \Bigr\}, \quad \varphi_{U_4}(y_1, y_2) = \Bigl( \frac{2}{3}y_1 \sqrt{y_2}, \sqrt{y_2} \Bigr),    
\end{align*}
we note that the conclusions of Corollary \ref{corollary-mainthm-resolution} hold in these new coordinates. Therefore 
\[ \Phi_{U_i}(y_1, y_2, y_3) = (\varphi_{U_i}(y_1, y_2), x_3) \in \mathcal C(F), \quad 1 \leq i \leq 4. \] 
   
\section{Proof of Theorems \ref{main-thm1} and \ref{main-thm2}} \label{sec-proof-mainthm1} 
The main results in this section are the following two propositions. The first gives an algebraic expression for the Newton {\atxt exponent}  of a fractional power series with a factorization of the form (\ref{F-factid2}). The second proposition uses this algebraic representation to compute the critical integrability index of a model class of functions whose roots are fractional power series. Together with Theorem \ref{mainthm-resolution}, they supply the essential ingredients of the proof of Theorem \ref{main-thm1}, given at the end of this section. 
\begin{proposition}
Suppose that $V \subseteq (0,1)^n$ and that $\varphi$ is a coordinate system on $V$ such that the roots of $F(\varphi(\mathbf y), x_{n+1})$ with respect to $x_{n+1}$ are fractional normal crossings and their differences are also fractional normal crossings. In particular, assume that $F$ admits the factorization (\ref{F-factid}) and hence (\ref{F-factid2}), with the exponents $\{ \pmb{\alpha}_{\ell} = (\alpha_{\ell}(1), \cdots, \alpha_{\ell}(n)): 1 \leq \ell \leq L\}$ ordered as in (\ref{alpha-ordering}). If $\Phi$ is the coordinate transformation given by $(\mathbf x, x_{n+1}) = \Phi(\mathbf y, y_{n+1}) = (\varphi(\mathbf y), y_{n+1})$, then 
\[ \delta_0(F;\Phi) = \min \left[ \frac{\alpha_{\ell}(j) + 1}{A_{\ell}(j) + \alpha_{\ell}(j) B_{\ell}} : 1 \leq j \leq n, 1 \leq \ell \leq L \right],  \] 
where   
\begin{equation}
\mathbf A_{\ell} = \sum_{k \leq \ell} \pmb{\alpha}_k \sum_{i \in \mathcal L_k} n_i = (A_{\ell}(1), \cdots, A_{\ell}(n)), \qquad  B_{\ell} = \sum_{k > \ell} \sum_{i \in \mathcal L_k} n_i + \beta_{n+1} \label{AlBl-new}
\end{equation}
are as in (\ref{AlBl-old}). \label{prop-step1-thm1}
\end{proposition} 
\begin{proof}
By Lemma \ref{mep-lemma3}, NP$(F;\Phi)$ is defined by the monotone edge path joining the points $\{(\mathbf A_{\ell}, B_{\ell}) : 1 \leq \ell \leq L \}$. By Lemma \ref{lemma-nd=gen}, $\delta_0(F, \Phi)$ equals the generalized Newton {\atxt exponent}  of $F$ in coordinates $\Phi$. But a straightforward calculation shows that the diagonal in $\mathbb R^2$ intersects the line joining $(A_{\ell-1}(j), B_{\ell-1})$ and $(A_{\ell}(j), B_{\ell})$ at a point each of whose coordinates is $(A_{\ell}(j) + \alpha_{\ell}(j)B_{\ell})/(\alpha_{\ell}(j)+1)$. Thus the $j$th projected Newton {\atxt exponent}  $\delta_j(F, \Phi)$ is given by 
\[ \delta_j(F;\Phi) = \min \left[ \frac{\alpha_{\ell}(j) + 1}{A_{\ell}(j) + \alpha_{\ell}(j) B_{\ell}} : 1 \leq \ell \leq L \right], \]
completing the proof.   
\end{proof}
\begin{proposition}
Suppose that $G^{\ast}$ is a polynomial in $x_{n+1}$ with fractional power series coefficients, defined on $(-1,1)^{n} \times \mathbb R$ and admitting the factorization
\begin{equation} G^{\ast}(\mathbf x, x_{n+1}) =  \prod_{i=1}^{M} \left( x_{n+1} - \varrho_i(\mathbf x) \right)^{n_i}, \label{fps-factid} \end{equation}  
where the roots $\{ \varrho_i \}$ are distinct fractional power series. Let $\mathcal C(G^{\ast})$ be the class of coordinates defined as in the statement of Theorem \ref{main-thm1}.   
Then 
\begin{equation} \int_{(-1,1)^{n+1}} |G^{\ast}(\mathbf x, x_{n+1})|^{- \delta} d\mathbf x \, d\mathbf x_{n+1} < \infty \label{locint-condition} \end{equation} 
if and only if 
\begin{equation} \label{if-implication} \delta < \inf \left\{ \delta_0(G^{\ast}, \Phi) : \Phi \in \mathcal C(G^{\ast}) \right\}. \end{equation}    
In fact there exists a finite subcollection $\mathcal C_0(G^{\ast}) \subseteq \mathcal C(G^{\ast})$ such that the infimum above equals $\min \{ \delta_0(G^{\ast}; \Phi) : \Phi \in \mathcal C_0(G^{\ast}) \}$. \label{prop-step2-thm1}
\end{proposition} 
\begin{proof}
The proof has certain similarities with the bivariate case, specifically with parts of the analysis carried out in \cite{PhSt97}, \cite{PhSteStu99}. We will need to use the sets (i.e., unions of towers) described in \S \ref{sec-resolution} in the proof of Proposition \ref{lemma-block2}. However, the preferred coordinates on the towers will be irrelevant for the present analysis. We will primarily work with the ambient coordinate system $\varphi$ of that proposition, which in the present context will arise from $\Phi(\varphi, V, r) \in \mathcal C(G^{\ast})$.   

Suppose that (\ref{locint-condition}) holds for some $\delta > 0$. We intend to show that $\delta < \delta_0(G^{\ast}, \Phi)$ for any $\Phi = \Phi(\varphi, V, r) \in \mathcal C(G^{\ast})$. Fix such a coordinate system $\Phi$, and set $F^{\ast} := G^{\ast} \circ \Phi$. The convergence of the integral in (\ref{locint-condition}) implies that \begin{equation} \int_{(0,1)^n \times (-\epsilon,\epsilon)} |F^{\ast}(\mathbf y, y_{n+1})|^{-\delta} d\mathbf y \, dy_{n+1} < \infty \label{locint-condition-implication} \end{equation} for some small $\epsilon > 0$. Without loss of generality, by scaling if necessary we may assume that $\epsilon = 1$. It follows from (\ref{fps-factid}) and the definition of $\mathcal C(G^{\ast})$ that the function $F^{\ast}$ which is given by
\[ F^{\ast}(\mathbf y, y_{n+1}) =   \prod_{i=1}^{M} \bigl(y_{n+1} - r_i(\mathbf y) \bigr)^{n_i}, \quad  r_i(\mathbf y) = \varrho_i(\varphi(\mathbf y)) - r(\mathbf y) \] admits a factorization of the form (\ref{F-factid2}) with the roots $\{r_i\}$ obeying the hypotheses of Proposition \ref{lemma-block2}. Denoting by $\{ \pmb{\alpha}_{\ell} \}$ the distinct leading exponents of the roots $\{r_i \}$ and assuming the ordering (\ref{alpha-ordering}), we can therefore define sets $\widetilde{W}_{\ell, \kappa}$ similar to (\ref{Vtilde-lkappa}), namely,
\[ \widetilde{W}_{\ell, \kappa} = \{(\mathbf y, y_{n+1}) :  \frac{1}{2}\mathbf y^{\pmb{\alpha}_{\ell+1}} < \kappa y_{n+1} <  \mathbf y^{\pmb{\alpha}_{\ell}}, \; \mathbf y \in (0,1)^n \},  \; 1 \leq \ell \leq L-1, \]
with appropriate modifications at the endpoints $\ell = 0, L$. Then $\widetilde{W}_{\ell, \kappa} \subseteq (0,1)^n \times (-1,1)$ for every $0 \leq \ell \leq L$ and $\kappa = \pm 1$. The finiteness of the integral in (\ref{locint-condition-implication}) then leads to the estimate  
\begin{multline} \int_{\widetilde{W}_{\ell, \kappa}} \left| F^{\ast}(\mathbf y, y_{n+1})\right|^{-\delta} d\mathbf y \, dy_{n+1} \\ \leq \int_{(0,1)^n \times (- 1, 1)} |F^{\ast}(\mathbf y, y_{n+1})|^{-\delta} \, d\mathbf y \, dy_{n+1} < \infty. \label{finiteness} \end{multline} 
But on $\widetilde{W}_{\ell, \kappa}$, there is a constant $C > 0$ such that
\begin{align*} 
|y_{n+1} - r_i(\mathbf y)| &\leq |y_{n+1}| + |r_i(\mathbf y)| \\  
&\leq \begin{cases}  C \mathbf y^{\pmb{\alpha}_{k}} &\text{ if } k \leq \ell, \; i \in \mathcal L_k, \\ C |y_{n+1}| &\text{ if } k \geq \ell + 1, \; i \in \mathcal L_k, \end{cases}
\end{align*}  
which implies 
\begin{equation} |F^{\ast}(\mathbf y, y_{n+1})| \leq C^N \mathbf y^{\mathbf A_{\ell}} |y_{n+1}|^{B_{\ell}}, \quad N = n_1 + \cdots + n_M. \label{Fstar-upper} \end{equation} 
It follows from (\ref{finiteness}) and (\ref{Fstar-upper}) that 
\[ \int_{\widetilde{W}_{\ell, \kappa}} \mathbf y^{-\mathbf A_{\ell} \delta} |y_{n+1}|^{-B_\ell \delta} \, d\mathbf y \, dy_{n+1} \leq (C)^{N \delta} \int_{\widetilde{W}_{\ell, \kappa}} \left| F^{\ast}(\mathbf y, y_{n+1})\right|^{-\delta} d\mathbf y \, dy_{n+1}  < \infty.  \]
Thus
\begin{align*}
\infty > \int_{\widetilde{W}_{\ell, \kappa}} \mathbf y^{-\mathbf A_{\ell} \delta} & |y_{n+1}|^{-B_\ell \delta} \, d\mathbf y \, dy_{n+1} \\ &= \int_{(0, 1)^n} \mathbf y^{-\mathbf A_{\ell} \delta} \int_{\frac{1}{2}\mathbf y^{\pmb{\alpha}_{\ell+1}}}^{\mathbf y^{\pmb{\alpha}_{\ell}}} |y_{n+1}|^{-B_{\ell} \delta} dy_{n+1} \, d\mathbf y \\  & \geq  \int_{(0, 1)^n}    \mathbf y^{-\mathbf A_{\ell} \delta} \int_{ \frac{1}{2}\mathbf y^{\pmb{\alpha}_{\ell}}}^{ \mathbf y^{\pmb{\alpha}_{\ell}}} |x_{n+1}|^{-B_{\ell} \delta} dy_{n+1} \, d\mathbf y \\
&\geq \left\{ \begin{aligned} &C^{-1}\int_{(0, 1)^n} \mathbf y^{-\mathbf A_{\ell} \delta} \mathbf y^{\pmb{\alpha}_{\ell}(1 - B_{\ell} \delta)} \, d\mathbf y &\text{ if } 1 - B_{\ell} \delta \ne 0, \\ 
& C^{-1}\int_{(0, 1)^n} \mathbf y^{-\mathbf A_{\ell} \delta} \ln \left(\frac{1}{|\mathbf y|} \right) \, d\mathbf y &\text{ if } 1 - B_{\ell} \delta = 0. \end{aligned} \right\}  
\end{align*}  
The convergence of the integrals in the last line shows that for all $1 \leq \ell \leq L$, 
\[ \delta < \left\{\begin{aligned} & \min \left\{\frac{\alpha_{\ell}(j)+1}{A_{\ell}(j) + \alpha_{\ell}(j)B_{\ell}} : 1 \leq j \leq n \right\} &\text{ if } \delta \ne \frac{1}{B_{\ell}}, \\ 
& \min \left\{\frac{1}{A_{\ell}(j)} : 1 \leq j \leq n \right\} &\text{ if } \delta = \frac{1}{B_{\ell}}.  \end{aligned} \right\} \]
The second situation only occurs if for all $1 \leq j \leq n$, 
\[ \frac{1}{B_{\ell}} < \frac{1}{A_{\ell}(j)} \quad \text{ in which case } \quad \frac{1}{B_{\ell}} < \frac{1 + \alpha}{A_{\ell}(j)+\alpha B_{\ell}} \text{ for any } \alpha > 0. \] 
Combining both cases above and taking the minimum over all $1 \leq \ell \leq L$, we obtain  
\[  \delta < \min \left[\frac{\alpha_{\ell}(j)+1}{A_{\ell}(j) + \alpha_{\ell}(j)B_{\ell}} : 1 \leq j \leq n, 1 \leq \ell \leq L \right]. \] 
In view of Proposition \ref{prop-step1-thm1}, this implies $\delta < \delta_0(F^{\ast}) = \delta_0(G^{\ast};\Phi)$, completing the ``only if'' part of the desired conclusion.  

For the reverse implication, we fix $\delta$ satisfying (\ref{if-implication}). Since $G^{\ast}$ admits the representation (\ref{fps-factid}), there exists by Corollary \ref{corollary-mainthm-resolution} a finite collection of set-coordinate pairs $\{(V, \varphi_V) : V \in \mathcal V \}$ that monomialize the roots of $G^{\ast}$ in the sense of that corollary. Since $\{ V \in \mathcal V \}$ is a finite cover of $(-1,1)^n$ (up to a set of measure 0), it suffices to show that for every $V \in \mathcal V$, 
\begin{equation} \label{only-if-implication}
\int_{V \times (-1,1)} \bigl| G^{\ast}(\mathbf x, x_{n+1})\bigr|^{-\delta} \, d\mathbf x \, dx_{n+1} < \infty.
\end{equation}   

In order to prove (\ref{only-if-implication}), we set $r_i(\mathbf y) = \varrho_i(\varphi_V(\mathbf y))$, so that the hypotheses of Proposition \ref{lemma-block2} hold with $(F, V, \varphi)$ in that proposition replaced by $(G^{\ast}, V, \varphi_V)$. We need a classification of the roots $r_i$ based on their leading exponents and coefficients, and hence define the index sets $\mathbb I^{\ast}$, $\mathcal I$, $\mathcal I'$, $\mathcal J$ and $\mathcal J'$ as in the proof of Proposition \ref{lemma-block2}. We will also need to use the regions (\ref{region-main1})--(\ref{region8}) with $r_i$ replaced by Re$(r_i)$. As already mentioned in the proof of the proposition, these regions cover $V \times (-1, 1)$. Our goal is to show that $|G^{\ast}|^{-\delta}$ is integrable on each of these regions. 

As before, we invoke the change of coordinates (\ref{changeofvar-reduction2}), which belongs to the coordinate class $\mathcal C(G^{\ast})$ to convert the regions (\ref{region3})--(\ref{region8}) to regions of the form (\ref{region-main1}) or (\ref{region-main2}). Since the analysis on these last two domains are very similar, we restrict attention to the region given by (\ref{region-main1}) (which we henceforth denote by $\widetilde{W}$), where we proceed to estimate $G^{\ast}$ from below.  

On $\widetilde{W}$, we use the estimates derived in Proposition \ref{lemma-block2} to conclude that 
\[ G^{\ast} = (\text{unit}) \left[\prod_{i \in \mathcal I} \text{Re}(r_i(\mathbf y))^{n_i} \right] x_{n+1}^{\beta_{n+1}} \left[ \prod_{i \in \mathcal J \cup \mathbb I^{\ast}} \left( x_{n+1} - r_i(\mathbf y) \right)^{n_i} \right]. \]
Let $\ell_0$ be the unique index in $\{1, 2, \cdots, L \}$ such that $r_{i_1} \in \mathcal L_{\ell_0}$. If $i \in \mathcal I$, then $r_{i_1} \leq r_i$, from which it follows that $i \in \mathcal L_k$ for some $k \leq \ell_0$. Define 
\[ \mathbb I_0 := \left\{ i \in \mathcal J \cup \mathbb I^{\ast} : i \in \mathcal L_{\ell}, \ell > \ell_0 \right\}, \qquad \mathbb I_0^c = (\mathcal J \cup \mathbb I^{\ast}) \setminus \mathbb I_0. \]
Then for $i \in \mathbb I_0^{c}$ and $(\mathbf x, x_{n+1})$ in $\widetilde{W}$,
\begin{align*}  |x_{n+1} - r_i(\mathbf y)| &\geq \begin{cases} \left| \text{Im}(r_i(\mathbf y))\right| &\text{ if } i \in \mathbb I_0^c \cap \mathbb I^{\ast}, \\ |x_{n+1}| + |\text{Re}(r_i(\mathbf y))| &\text{ if } i \in \mathbb I_0^c \cap \mathcal J \end{cases}  \\ &\geq C^{-1}|r_i(\mathbf y)|,  \end{align*} 
so that $G^{\ast}$ admits the following lower bound on $\widetilde{W}$:  
\begin{equation} \label{lowerF} |G^{\ast}| \geq C^{-1} x_{n+1}^{\beta_{n+1}} \Biggl[ \prod_{i \in \mathcal I \cup \mathbb I_0^c} |r_i(\mathbf y)|^{n_i} \Biggr] \Biggl[ \prod_{i \in \mathbb I_0} |x_{n+1} - r_i(\mathbf y) |^{n_i} \Biggr]. \end{equation}
It remains to estimate the ``non-cancelling'' factors in the above expression corresponding to $i \in \mathbb I_0$. For this we cover $\widetilde{W}$ with sets of the form  
\begin{align*}
\widetilde{V}_{\ell} &= \left\{(\mathbf x, x_{n+1}) \in \widetilde{W} : p_{\ell} \mathbf y^{\pmb{\alpha}_{\ell+1}} < x_{n+1} < q_{\ell} \mathbf y^{\pmb{\alpha}_{\ell}} \right\}, \quad \ell_0 \leq \ell \leq L-1, \\ 
\widetilde{V}_L &= \left\{(\mathbf x, x_{n+1}) \in \widetilde{W} : 0 < x_{n+1} < q_L \mathbf y^{\pmb{\alpha}_L} \right\}, 
\end{align*}        
where the constants $p_{\ell}, q_{\ell}$ are chosen as in Lemma \ref{imaginary-roots}. Then the same kind of estimation as in Lemma \ref{imaginary-roots} and Corollary \ref{corollary-imroots} yields that on $\widetilde{V}_{\ell}$ and for $i \in \mathbb I_0$,  
\begin{align*} |x_{n+1} - r_i(\mathbf y)| &\geq C^{-1}\max \left(|x_{n+1}|, |r_i(\mathbf y)| \right) \\ &\geq C^{-1} \begin{cases} |x_{n+1}| &\text{ if } i \in \mathcal L_k, \, k \geq \ell + 1, \\ |r_i(\mathbf y)| &\text{ if } i \in \mathcal L_k, \, k \leq \ell. \end{cases} \end{align*}  
In conjunction with (\ref{lowerF}) this implies that on $\widetilde{V}_{\ell}$, 
\begin{align*} |G^{\ast}| &\geq C^{-1} x_{n+1}^{\beta_{n+1}}\Bigl[ \prod_{i \in \mathcal I \cup \mathbb I_0^c} |r_i(\mathbf y)|^{n_i} \Bigr] \Biggl[ \prod_{\begin{subarray}{c}i \in \mathbb I_0 \cap \mathcal L_k \\ k \leq \ell  \end{subarray}} |r_i(\mathbf y)|^{n_i} \Biggr] \Biggl[ \prod_{\begin{subarray}{c}i \in \mathbb I_0 \cap \mathcal L_k \\ k \geq \ell + 1 \end{subarray}} x_{n+1}^{n_i} \Biggr] \\ &\geq C^{-1} \mathbf y^{\mathbf A_{\ell}} x_{n+1}^{B_\ell}, \end{align*} 
where the last step follows from the fact that for $\ell \geq \ell_0$, 
\begin{align*} \bigcup_{k \geq \ell+1} \bigl(\mathbb I_0 \cap \mathcal L_k \bigr) &= \{i : i \in \mathcal L_k, \, k \geq \ell + 1 \} \quad  \text{ and } \\ \Bigl[\bigcup_{k \leq \ell} \bigl( \mathbb I_0 \cap \mathcal L_k  \bigr) \Bigr] \bigcup (\mathcal I \cup \mathbb I_0^c) &= \left\{i : i \in \mathcal L_k, k \leq \ell \right\}.\end{align*}
Therefore 
\begin{align*} \int_{\widetilde{V}_{\ell}} |G^{\ast}|^{-\delta} \, d\mathbf x \, dx_{n+1} &\leq \int_{(0,1)^n} \mathbf y^{-\mathbf A_{\ell} \delta} \int_{C^{-1}\mathbf y^{\pmb{\alpha}_{\ell + 1}}}^{C \mathbf y^{\pmb{\alpha}_{\ell}}} x_{n+1}^{-B_{\ell} \delta} \, dx_{n+1} d\mathbf y \\  
&\leq \left\{ \begin{aligned} & C \int_{(0,1)^n} \mathbf y^{-\mathbf A_{\ell} \delta + \pmb{\alpha}_{\ell}(1 - B_{\ell} \delta)} \, d\mathbf y  &\text{ if } 1 - B_{\ell} \delta > 0, \\ &C \int_{(0,1)^n} \mathbf y^{-\mathbf A_{\ell} \delta + \pmb{\alpha}_{\ell + 1}(1 - B_{\ell} \delta)} \, d\mathbf y  &\text{ if } 1 - B_{\ell} \delta < 0, \\ &C \int_{(0,1)^n} \mathbf y^{-\mathbf A_{\ell} \delta} \, \ln \left( \frac{1}{|\mathbf y|} \right) d\mathbf y  &\text{ if } 1 - B_{\ell} \delta = 0.  \end{aligned}  \right\}
\end{align*} 
We consider the three cases separately. If our chosen $\delta$ satisfies $\delta < 1/B_{\ell}$, then the rightmost integral converges for \[\delta < \min_{1 \leq j \leq n}\frac{\alpha_{\ell}(j) + 1}{A_\ell(j) + \alpha_{\ell}(j) B_{\ell}}, \] and hence for any $\delta$ satisfying (\ref{if-implication}) in view of Proposition \ref{prop-step1-thm1}. In the second case, the easy identity \[\mathbf A_{\ell} + \pmb{\alpha}_{\ell+1}B_{\ell} = \mathbf A_{\ell+1} + \pmb{\alpha}_{\ell+1} B_{\ell+1} \] implies that the integral converges for  
\[ \delta < \min_{1 \leq j \leq n} \frac{\alpha_{\ell + 1}(j) + 1}{A_{\ell+1}(j) + \alpha_{\ell+1}(j) B_{\ell+1}},\] which again is true for any $\delta$ satisfying (\ref{if-implication}). Finally, if $\delta = 1/{B_{\ell}}$ satisfies (\ref{if-implication}), then for all $1 \leq j \leq n$ 
\[ \frac{1}{B_{\ell}} < \frac{1 + \alpha_{\ell}(j)}{A_{\ell}(j) + \alpha_{\ell}(j)B_{\ell}} \quad \text{ and hence } \quad A_{\ell}(j) < B_{\ell},  \]
which implies that the integral in the last case converges. This establishes the reverse implication in the first part of the proposition. 

To establish the last part of the proposition, we define the coordinate class 
\[ \mathcal C_0(G^{\ast};V) = \left\{ \Phi : \begin{aligned} &(\mathbf x, x_{n+1}) = \Phi(\mathbf y, y_{n+1}) = \bigl(\varphi_V(\mathbf y), y_{n+1} + \text{Re}(r(\mathbf y)) \bigr) \\ &\text{where $x_{n+1} = r(\mathbf y)$ is a root of $G^{\ast}(\varphi_V(\mathbf y), x_{n+1})$}, \\ & \text{either } r \equiv r_i \text{ for some } i \in \mathcal I \cup \mathcal J, \text{ or } r \equiv 0\end{aligned}  \right\}, \]
where $\{(V, \varphi_V) : V \in \mathcal V \}$ are the sets and coordinates produced by Corollary \ref{corollary-mainthm-resolution} for $G^{\ast}$. In light of this corollary, $\mathcal C_0(G^{\ast};V)$ is a subclass of $\mathcal C(G^{\ast})$. We also observe that the ``only if'' argument above in fact proves the following stronger statement; namely, $|G^{\ast}|^{-\delta}$ is integrable on $V \times (-1,1)$ for any $\delta < \min \{ \delta_0(G^{\ast};\Phi) : \Phi \in \mathcal C_0(G^{\ast};V) \}$. Setting \[ \mathcal C_0(G^{\ast}) = \bigcup \left\{ \mathcal C_0(G^{\ast};V) : V \in \mathcal V \right\},  \] we find that (\ref{locint-condition}) holds for any $\delta < \min \left\{ \delta_0(G^{\ast};\Phi) : \Phi \in \mathcal C_0(G^{\ast}) \right\}$. Hence by the first part of the proposition \[ \min \left\{\delta_0(G^{\ast};\Phi) : \Phi \in \mathcal C_0(G^{\ast}) \right\} \leq \inf \left\{\delta_0(G^{\ast};\Phi) : \Phi \in \mathcal C(G^{\ast}) \right\}.\] The converse inequality is of course immediate.       
\end{proof}

\subsection{Proof of Theorem \ref{main-thm1}}
Fix a real-analytic function $F$ defined near the origin in $\mathbb R^{n+1}$. Let $\epsilon_0 > 0$ be a small constant such that after an orthogonal linear change of variables if necessary, (\ref{W-poly}) holds on $(-\epsilon_0, \epsilon_0)^{n+1}$.   

Fix a coordinate system $\Phi = \Phi(\varphi, V, r) \in \mathcal C$. Then $\widetilde{F}(\mathbf y, y_{n+1}) = F \circ \Phi(\mathbf y, y_{n+1})$ has roots that are fractional normal crossings. It particular, $\widetilde{F}$ satisfies the hypothesis of Proposition \ref{prop-step2-thm1}. 
Let $\delta > 0$ be such that $|F|^{-\delta}$ is integrable in an open neighborhood $U$ of the origin in $\mathbb R^{n+1}$. We choose $\eta > 0$ sufficiently small so that $\Phi((0,\eta)^{n} \times (-\eta, \eta)) \subseteq U$. Then 
\begin{align*} \infty > \int_{U} \left|F(\mathbf x, x_{n+1}) \right|^{-\delta} & d\mathbf x dx_{n+1} \\ &\geq C^{-1} \int_{(0, \eta)^{n}\times (-\eta, \eta)} \left|F(\Phi(\mathbf y, y_{n+1})) \right|^{-\delta} d\mathbf y \, dy_{n+1} \\ &= C^{-1} \int_{(0, \eta)^{n} \times (-\eta, \eta)} \left|\widetilde{F}(\mathbf y, y_{n+1}) \right|^{-\delta} d\mathbf y \, dy_{n+1}\\ & \geq C_{\eta}^{-1}\int_{(0, 1)^{n} \times (-1,1)} \left|\widetilde{F}_{\eta}(\mathbf y, y_{n+1}) \right|^{-\delta} d\mathbf y \, dy_{n+1}, \end{align*}
where $\widetilde{F}_{\eta}(\mathbf y, y_{n+1}) = \widetilde{F}(\eta \mathbf y, \eta y_{n+1})$ continues to obey the hypotheses of Proposition \ref{prop-step2-thm1}. The proposition then implies $\delta < \delta_0(\widetilde{F}_{\eta}) = \delta_{0}(\widetilde{F}) = \delta_0(F; \Phi)$. Taking the supremum over all $\delta$ in (\ref{cii-def}) and infimum over all $\Phi \in \mathcal C(F)$, we arrive at the conclusion 
\[ \mu_0(F) \leq \inf \left\{ \delta_0(F;\Phi) : \Phi \in \mathcal C \right\}.  \]

For the converse inequality, we invoke Theorem \ref{mainthm-resolution}. Let $\{(V, \varphi_V) : V \in \mathcal V \}$ be the finite collection of $n$-dimensional prepared towers specified in that theorem, so that the following estimate holds: 
\begin{equation} \label{reduce-finite} 
\begin{aligned}
\int_{(-\epsilon_0, \epsilon_0)^{n+1}} |F|^{-\delta} d\mathbf x \, dx_{n+1} &= \sum_{V \in \mathcal V} \int_{V \times (-\epsilon_0, \epsilon_0)} \left| F \right|^{-\delta} \, d\mathbf x \, dx_{n+1} \\ & \leq C \sum_{V \in \mathcal V} \int_{(0,1)^n \times (-\epsilon_0, \epsilon_0)} \left| F_V \right|^{-\delta}\, d\mathbf y\, dx_{n+1}, 
\end{aligned} 
\end{equation} 
where $F_V = F \circ \Phi_V$ and $\Phi_V(\mathbf y, y_{n+1}) = (\varphi_V(\mathbf y), y_{n+1})$. In view of (\ref{F-factid}), we note that $F_V$ satisfies the hypothesis of Proposition \ref{prop-step2-thm1}. Recalling the definition of $\mathcal C_0(F_V)$ from this proposition, we define the coordinate class $\mathcal C^{\ast}$ as follows, 
\[ \mathcal C^{\ast} := \left\{\Phi : \Phi = \Phi_V \circ \Psi, \; \Psi \in \mathcal C_0(F_V), V \in \mathcal V \right\}. \] 
It is easy to see that $\mathcal C^{\ast} \subseteq \mathcal C$.  
Further, invoking Proposition \ref{prop-step2-thm1} we may deduce that the integrals occurring as summands in (\ref{reduce-finite}) converge for   
\[ \delta < \min \left\{\delta_0(F;\Phi) : \Phi \in \mathcal C^{\ast} \right\}. \]
Therefore 
\[ \mu_0(F) \geq \min \left\{ \delta_0(F;\Phi) : \Phi \in \mathcal C^{\ast} \right\} \geq \inf \left\{ \delta_0(F, \Phi) : \Phi \in \mathcal C \right\}. \] 
This proves the converse inequality and also establishes the second statement (\ref{coordclass-finite}) of Theorem \ref{main-thm1}. \qed

\subsection{Proof of Theorem \ref{main-thm2}}
For any $\Phi \in \mathcal C_{\infty}$, $F \circ \Phi$ is a fractional power series with $F \circ \Phi(\mathbf 0, 0) = 0$. Suppose that $\delta < \mu_0(F)$, hence there is an open neighborhood $U$ of the origin such that $|F|^{-\delta} \in L^1(U)$. Fix $\epsilon > 0$ sufficiently small so that $\Phi(-\epsilon, \epsilon)^{n+1} \subseteq U$. Since $\Phi$ has nonvanishing Jacobian, we find that \[ \int_{(-\epsilon, \epsilon)^{n+1}} |F \circ \Phi|^{-\delta} \leq C \int_U |F|^{-\delta} < \infty, \quad \text{ hence } \quad \delta \leq \mu_0(F \circ \Phi) \leq \delta_0(F \circ \Phi),\] where the last inequality follows from Theorem \ref{prop-cii-nd-ndim}. Taking supremum over $\delta$ and infimum over $\Phi \in \mathcal C_{\infty}$, we obtain that $\mu_0(F) \leq \inf \{ \delta_0(F;\Phi) ; \Phi \in \mathcal C_{\infty} \}$. In view of Theorem \ref{main-thm1} however, the converse inequality is obvious. The second conclusion of Theorem \ref{main-thm2} uses a very similar argument and is left to the reader. \qed

\section{Comparison with the bivariate case} \label{sec-bivariate-comparison} 
In view of Theorem \ref{main-thm1} and its bivariate analogue (Theorem 4, \cite{PhSteStu99}), it is natural to ask whether the class $\mathcal C(F)$ can in general be replaced by the smaller class of global analytic coordinate changes. As indicated in the introduction, a counterexample due to Varchenko \cite{Var76} provides a negative answer to this question. The following is a simpler version of this example that suffices to emphasize the necessity of the local coordinate changes.   

\subsection{A counterexample}  Let $F(x_1, x_2, x_3) = x_3^2 - (x_1^2 + x_2^2)$. It is easy to see that if $\Phi$ is any analytic coordinate change, then $\delta_0(F;\Phi) = \frac{3}{2}$. On the other hand, in the local coordinates 
\[ \varphi_V(y_1, y_2) = (\sqrt{y_1}, \sqrt{y_1}y_2) \quad \text{ on } \quad V = \{(x_1, x_2) : 0 < x_2 < x_1 \},  \]
the function $F$ takes the form \[ F(\varphi_V(y_1, y_2), x_3) = x_3^2 - y_1(1 + y_2^2). \]
Setting $\Phi_V(y_1, y_2, y_3) = (\varphi_V(y_1, y_2), y_3 + \sqrt{y_1(1+y_2^2)})$, we find that \[ F \circ \Phi_V(\mathbf y) = y_3^2 + 2 y_3 \sqrt{y_1(1+y_2^2)}, \]    
and that the Newton {\atxt polyhedron} of $F \circ \Phi_V$ is defined by the monotone edge path joining $(0,0,2)$ and $(\frac{1}{2}, 0, 1)$. In particular $\delta_0(F;\Phi_V) = 1 < \frac{3}{2}$. A straightforward calculation shows that in fact $\mu_0(F) = 1$.  

\subsection{Identification of an adapted coordinate system}
We devote the rest of this section to studying finer aspects of the critical integrability problem in $(n+1)$ dimensions, $n \geq 2$. We begin by recalling two definitions, one from the introduction and the other from \cite{Var76}, both pertinent to the bivariate situation. Let $\mathcal C_{\omega}$ denote the class of analytic coordinate transformations in $\mathbb R^2$ of the form (\ref{phong-stein-adapted}). Given a bivariate real-analytic function $F$, we say that an analytic coordinate system $\Phi \in \mathcal C_{\omega}$ is {\em{adapted}} to $F$ if $\delta_0(F, \Phi) = \mu_0(F)$. An attractive feature of two dimensions is that the Newton polyhedron NP$(F;\Phi)$ contains information as to whether $\Phi$ is adapted. More precisely, let $\kappa_{\ell}$ denote the cardinality of the largest subset of $\mathcal L_{\ell}$ consisting of all roots with the same leading coefficient. It is proved in \cite{PhSteStu99} that for a bivariate real-analytic function $F$ and for any $\Phi \in \mathcal C_{\omega}$, the inequality
\begin{equation}
\kappa_{\ell} \leq \delta_0(F;\Phi)^{-1} \label{non-mainface}
\end{equation}   
holds for every $\ell$ that corresponds to a non-main face of NP$(F;\Phi)$. (We recall that a main face of a Newton polyhedron in $\mathbb R^2$ is one that intersects the bisectrix $x_1 = x_2$.) Moreover, (\ref{non-mainface}) holds for a main face if and only if the ambient coordinate system is adapted.  

We investigate whether similar identifications of adapted coordinate systems based on the Newton polyhedron hold for $(n+1)$-variate functions $F$ with $n \geq 2$. First, a few minor modifications in the definitions are necessary in view of the structure of the Newton polyhedra and the local nature of the coordinates in $\mathcal C(F)$. 
\begin{definition}
\begin{enumerate}[1.]
\item Suppose that a Newton polyhedron NP$(F, \Phi)$ is defined by a monotone edge path consisting of the line segments $\Gamma_\ell$, where \begin{equation} \Gamma_\ell = \{ t (\pmb{\mu}_\ell, \nu_\ell) + (1-t) (\pmb{\mu}_{\ell+1}, \nu_{\ell+1}) : 0 \leq t \leq 1 \}, \quad 1 \leq \ell \leq L, \label{def-Gammaell} \end{equation} with $\pmb{\mu}_\ell < \pmb{\mu}_{\ell+1}$, $\nu_{\ell} > \nu_{\ell+1}$. An edge $\Gamma_\ell$ is said to be a {\em{main edge of NP$(F;\Phi)$}} if for some $1 \leq j \leq n$, 
\[ \pi_j \bigl( \Gamma_\ell \bigr)\cap \{ x_j = x_{n+1} \} = \bigl(\delta_0(F;\Phi)^{-1}, \delta_0(F;\Phi)^{-1}  \bigr). \] 
Recall that $\pi_j(\mathbf x, x_{n+1}) = (x_j, x_{n+1})$. 
\item We say that a coordinate system $\Phi = \Phi(V, \varphi, r) \in \mathcal C(F)$ is {\em{integrability-adapted to $F$}} if 
\[ \sup \Bigl\{ \delta : \int_V |F|^{-\delta} d\mathbf x < \infty \Bigr\} = \delta_0(F, \Phi). \]
\end{enumerate} 
\end{definition}

The same proof as in \cite{PhSteStu99} works essentially verbatim to show that (\ref{non-mainface}) continues to hold for any $\ell$ indexing a non-main edge to the right of a main edge. However, the relation need no longer be true on the left of a main edge, as the following example shows. Let 
\[ F(x_1, x_2, x_3) = (x_3^2 - x_1)(x_3^3 - x_2), \quad \text{ and } \quad V = \Bigl\{(x_1, x_2) : 0 < x_2 < \frac{1}{2} x_1^{\frac{3}{2}}\Bigr\}, \]
so that the change of variable $\Phi = \Phi(V, \varphi, r)$ lies in $\mathcal C(F)$, with $\varphi$ given by 
\[ \varphi(y_1, y_2) = \bigl(y_1^{\frac{2}{5}}, y_1^{\frac{3}{5}} y_2 \bigr) \quad \text{ and } \quad r \equiv 0. \]
It follows that NP$(F;\Phi)$ is defined by the monotone edge path joining $(0,0,5)$, $(\frac{2}{5},0,3)$ and $(1,1,0)$, and that $\delta_0(F;\Phi) = \frac{6}{5}$. The main edge of NP$(F;\Phi)$ is the line segment $\Gamma_2$ joining $(2/5, 0, 3)$ and $(1, 1, 0)$.  The non-main edge $\Gamma_1$ lying to the left of $\Gamma_2$ corresponds to $\mathcal L_1$, the collection of roots (in $y_3$) of $F \circ \Phi$ with leading exponent $(1/5, 0)$. Since the leading coefficients of such roots of $F$ are distinct, $\kappa_1 = 1 > \delta_{0}(F;\Phi)^{-1} = 5/6$.         

In spite of this departure from the bivariate situation, we are able to provide a sufficient condition on the Newton polyhedron that ensures that an ambient coordinate system is integrability-adapted. We need the following notation. Let $F \circ \Phi$ admit a factorization of the form (\ref{F-factid2}), so that the Newton polyhedron NP$(F;\Phi)$ is defined by a monotone edge path $\Gamma$ joining the points $\{(\mathbf A_{\ell}, B_{\ell}) : 0 \leq \ell \leq L \}$, with $(\mathbf A_{\ell}, B_{\ell})$ as in (\ref{AlBl-new}). In other words, $\Gamma$ is the union of the line segments $\Gamma_{\ell}$ joining $(\mathbf A_{\ell-1}, B_{\ell-1})$ and $(\mathbf A_{\ell}, B_{\ell})$. Let $\Gamma_{\ell}'$ denote the infinite line in $\mathbb R^{n+1}$ of which $\Gamma_{\ell}$ is a segment. We define $\delta_{\ell}(j;\Phi)$ as follows, 
\[ \bigl( \delta_{\ell}(j;\Phi)^{-1}, \delta_{\ell}(j;\Phi)^{-1} \bigr) := \pi_j \left(\Gamma_{\ell}' \right) \cap \{ x_j = x_{n+1}\}. \]
It follows from Lemma \ref{mep-lemma3} and Proposition \ref{prop-step1-thm1} that the direction vector of $\Gamma_{\ell}$ is $(\pmb{\alpha}_{\ell}, -1)$ and that  
\begin{equation} \label{ND-alt}\delta_{0}(F;\Phi) = \min \left\{\delta_{\ell}(j;\Phi) : 1 \leq \ell \leq L, \; 1 \leq j \leq n  \right\}.  \end{equation}    
\begin{proposition}
Suppose that $\Phi \in \mathcal C(F)$, and that $\delta_{\ell}(j;\Phi)$ is defined as above for NP$(F;\Phi)$. Assume that 
\begin{equation} \kappa_{\ell} \leq \min_{1 \leq j \leq n}\delta_\ell(j;\Phi)^{-1} \quad \text{ for every } 1 \leq \ell \leq L, \label{adapted-suff}\end{equation} 
where $\kappa_{\ell}$ is the largest cardinality of roots of $F \circ \Phi$ with leading exponent $\pmb{\alpha}_{\ell}$ sharing a common leading term. Then $\Phi$ is integrability-adapted to $F$.    
\end{proposition} 
\begin{proof}
Without loss of generality, replacing $F \circ \Phi$ by $F$ if necessary, we may assume that $V = (0,1)^n$, $\Phi = $ identity, and that $F$ admits the factorization (\ref{F-factid2}) with $\mathbf y = \mathbf x$ and with the ordering (\ref{alpha-ordering}), where every element of $\mathcal A = \{r_i, \text{Re}(r_i) : 1 \leq i \leq M \}$ and $\mathcal A - \mathcal A$ is either trivial or fractional normal crossings. In view of these properties, we find that $F$ satisfies the hypotheses of Proposition \ref{prop-step2-thm1}, and the finite subcollection $\mathcal C_0(F)$ defined in this proposition takes the form 
\[ \mathcal C_0(F) = \{\Psi : (\mathbf x, x_{n+1}) = \Psi(\mathbf y, y_{n+1}) = (\mathbf y, y_{n+1} + \text{Re}(r_i(\mathbf y))), i \in \mathcal I \cup \mathcal J \}. \]     
By Proposition \ref{prop-step2-thm1}, we know that 
\[ \int_{(0,1)^n \times (-1,1)} |F|^{-\delta} \, d\mathbf x \, dx_{n+1} \quad \text{ if and only if } \quad \delta < \min \left\{ \delta_0(F;\Psi) : \Psi \in \mathcal C_0(F) \right\}. \] 
Thus it suffices to show that \begin{equation} \label{toshow-revised} \delta_0(F;\Psi) \geq \delta_0(F) \quad \text{ for every } \Psi \in \mathcal C_0(F). \end{equation}  

Fix $\Psi_0 \in \mathcal C_0(F)$ so that $\Psi_0(\mathbf y, y_{n+1}) = (\mathbf y, y_{n+1} + \text{Re}(r_{i_0}(\mathbf y)))$ for some $i_0 \in \mathcal I \cup \mathcal J$. Let $\ell$ be the unique index such that $i_0 \in \mathcal L_{\ell}$, and \[ {\overline{\pmb{\alpha}}}_1 < {\overline{\pmb{\alpha}}}_2 < \cdots < {\overline{\pmb{\alpha}}}_R \]
the collection of distinct multi-indices that occur as leading exponents of the roots of $F \circ \Psi$. We will compare NP$(F;\Psi)$ with NP$(F) =$ NP$(F;\Phi)$, and show that  \begin{equation}  R \geq \ell, \quad  \begin{cases} \delta_{k}(j ; \Psi) = \delta_{k}(j; \Phi) &\text{ for } 1 \leq k < \ell, \\  \delta_{k}(j ; \Psi) \geq \delta_{\ell}(j; \Phi) &\text{ for } \ell \leq k \leq R. \end{cases} \label{toshow-revised2}\end{equation} 
In view of (\ref{ND-alt}), this establishes (\ref{toshow-revised}). 

We set about proving (\ref{toshow-revised2}). Since $i_0 \in \mathcal I \cup \mathcal J$, the fractional power series $r_{i_0}$ and Re$(r_{i_0})$ share the same leading exponent. We set $m_k$ (respectively $\overline{m}_k$) to be the number of roots of $F$ (respectively $F \circ \Psi$) with leading exponent $\pmb{\alpha}_k$ (respectively $\overline{\pmb{\alpha}}_k$), counting multiplicity. Further let $\overline{\gamma}_0 \geq 0$ be the multiplicity of the constant function $y_{n+1} \equiv 0$ as a root of $F \circ \Psi$. Thus $\overline{\gamma}_0 > 0$ if and only if $r_{i_0}$ is a real root of $F$. Since the roots of $F \circ \Psi$ are $\{ r_i - \text{Re}(r_{i_0}) : 1 \leq i \leq M \}$, it easily follows that 
\[ \overline{\pmb{\alpha}}_k = \pmb{\alpha}_k \text{ and } m_k = \overline{m}_k \text{ for } k < \ell, \quad \text{ so that } \quad \overline{m}_{\ell} + \cdots + \overline{m}_R + \overline{\gamma}_0 = B_{\ell-1} > 0.   \]
Thus $R \geq \ell$ and $\delta_k(j; \Psi) = \delta_k(j)$ for $1 \leq k < \ell$, $1 \leq j \leq n$, proving the first two claims in (\ref{toshow-revised2}). Further $\overline{\pmb{\alpha}}_{\ell} = \pmb{\alpha}_{\ell}$ if $\ell < L$, and $\overline{\pmb{\alpha}}_{\ell} \geq \pmb{\alpha}_{\ell}$ if $\ell = L$. 

Let us assume first that $\ell < L$. The last assertion in the preceding paragraph then implies that $\delta_{\ell}(j; \Psi) = \delta_{\ell}(j)$. The possibly new exponents $\{ \overline{\pmb{\alpha}}_{\ell+1}, \cdots, \overline{\pmb{\alpha}}_{R}\}$ originate from the roots $r_i$ that share the same leading coefficient as Re$(r_{i_0})$, and hence 
\[ \overline{m}_{\ell+1} + \cdots + \overline{m}_R + \overline{\gamma}_0 \leq \kappa_{\ell}. \]   
Now an easy calculation shows that
\begin{align*}  \delta_{\ell+1}(j;\Psi) &= \frac{1 + \overline{\alpha}_{\ell + 1}(j)}{A_{\ell-1}(j) + \overline{m}_{\ell} \alpha_{\ell}(j) + \overline{\alpha}_{\ell + 1}(j)(\overline{m}_{\ell+1} + \cdots + \overline{m}_R + \overline{\gamma}_0)}  \\  &= f(\overline{\alpha}_{\ell+1}(j)), \quad \text{ where } \\ 
f(\alpha) &:= \frac{1 + \alpha}{A_{\ell-1}(j) + \overline{m}_{\ell} \alpha_{\ell}(j) + \alpha (\overline{m}_{\ell+1} + \cdots + \overline{m}_R + \overline{\gamma}_0)}, \quad \alpha > 0.  
\end{align*}  
We observe that $f(\alpha_{\ell}(j)) = \delta_{\ell}(j)$ and that $f$ is a monotone function. Moreover by (\ref{adapted-suff}), 
\[ f(\alpha_{\ell}(j)) = \delta_{\ell}(j) \leq \kappa_{\ell}^{-1} \leq \frac{1}{\overline{m}_{\ell+1} + \cdots + \overline{m}_R+ \overline{\gamma}_0} = f(\infty), \] 
so that $f$ is in fact monotone increasing. Thus  
\[ \delta_{\ell}(j) = f(\alpha_{\ell}(j)) \leq f(\overline{\alpha}_{\ell+1}(j)) = \delta_{\ell+1}(j;\Psi) \leq f(\infty). \] 
A similar inductive argument based on $k \geq \ell+1$ shows that 
\[\delta_0(F;\Phi) \leq \delta_{\ell}(j) \leq \delta_{\ell+1}(j;\Psi) \leq \cdots \leq \delta_k(j;\Psi) \quad \text{ for all } k \geq \ell+1, \]
completing the proof of (\ref{toshow-revised2}). The treatment of the case $\ell = L$ involves minor modifications to the above argument and is left to the reader. 
\end{proof}

\section{Appendix: Proof of Theorem \ref{thm-JA-complex}} \label{sec-appendix} 
The essential ingredients of the proof are already in \cite{Suss90}, \cite{Par01}. Namely, we first define the root functions of $F$ as locally holomorphic functions on a dense subset of $\prod_{j=1}^{n} U_j$ and use the discriminantal condition (\ref{disc-norm-cross}) to extend them globally. 

\paragraph{Local root functions.}
For each $1 \leq j \leq n$, we define 
\[ \mathbb H_j = \left\{\mathbf z \in \mathbb C^n : z_j = 0 \right\} \quad \text{ and set } \quad \mathbb H = \bigcup_{j=1}^{n} \mathbb H_j. \]
Since $\mathbb H_j$ is of real codimension 2, $\mathcal O \setminus \mathbb H$ is an open connected dense subset of $\mathcal O$. If $\widehat{\mathbf z} \in \mathcal O \setminus \mathbb H$, then the discriminant condition (\ref{disc-norm-cross}) implies that $\Delta_F(\widehat{\mathbf z}) \ne 0$, so $F(\widehat{\mathbf z}, z_{n+1})$ considered as a polynomial in $z_{n+1}$ has $d$ distinct complex roots depending on $\widehat{\mathbf z}$. Let us denote these by $w_1(\widehat{\mathbf z}), \cdots, w_d(\widehat{\mathbf z})$. Since by (\ref{disc-norm-cross})
\[ \frac{\partial F}{\partial z_{n+1}}(\widehat{\mathbf z}, w_i(\widehat{\mathbf z})) \ne 0 \quad \text{ for } \widehat{\mathbf z} \in \mathcal O \setminus \mathbb H, \; 1 \leq i \leq d, \]     
the implicit function theorem implies that there exist some open neighborhood $\mathcal O_{\widehat{\mathbf z}}$ of $\widehat{\mathbf z}$ contained in $\mathcal O \setminus \mathbb H$ and $d$ holomorphic functions $\mathbf z \mapsto w_i(\mathbf z)$ on $\mathcal O_{\widehat{\mathbf z}}$ such that 
\[ F(\mathbf z, w_i(\mathbf z)) = 0 \quad \text{ and } \quad w_i(\mathbf z) \ne w_j(\mathbf z) \text{ for } z \in \mathcal O_{\widehat{\mathbf z}}. \] 

\paragraph{Definition of $s$.} The labeling of the set of roots on each $\mathcal O_{\widehat{\mathbf z}}$ is at the moment completely arbitrary and quite possibly globally inconsistent on $\mathcal O \setminus \mathbb H$. We now set about trying to arrange them globally. The following observation is a key step in this process. 

Let $\gamma:[0,1] \rightarrow \mathcal O \setminus \mathbb H$ be a continuous curve such that $\gamma(0) = \gamma(1) = \widehat{\mathbf z}$. We claim that $\gamma$ induces a permutation $\pi_{\gamma}$ of the roots, as follows. For each $1 \leq i \leq d$, one can define continuous functions $\zeta_i[\gamma]:[0,1] \rightarrow \mathbb C$ such that 
\begin{equation} \zeta_i[\gamma](0) = w_i(\widehat{\mathbf z}) \quad \text{ and }  \quad F(\gamma(t), \zeta_i[\gamma](t)) =0. \label{def-zetai} \end{equation} 
The condition $\Delta_F \ne 0$ on $\mathcal O \setminus \mathbb H$ implies that the functions $\zeta_i$ are uniquely defined. Each number $\zeta_i(1)$ is a root of $F(\widehat{\mathbf z}, \cdot)$, and hence of the form \[\zeta_i(1) = w_{\pi_{\gamma}(i)}(\widehat{\mathbf z}) \text{ for some permutation $\pi_{\gamma}$ of $\{ 1, 2, \cdots, d \}$.} \]
We leave the reader to verify that $\gamma \mapsto \pi_{\gamma}$ is a group homomorphism between the fundamental group of $\mathcal O$ and the permutation group $\mathbb S_d$.   
We choose $s$ to be an integer such that 
\begin{equation} \pi_{\gamma}^s = \underbrace{\pi_{\gamma} \circ \cdots \circ \pi_{\gamma}}_{s \text{ times}} = \text{identity} \quad \text{ for all } \widehat{\mathbf z} \in \mathcal O \setminus \mathbb H \text{ and all curves $\gamma$.} \label{choice-s} \end{equation}  Clearly $s=2$ for $d=1$ and $s = d! = \text{order of the permutation group } \mathbb S_d$ for $d \geq 2$ will suffice.    

\paragraph{Local lifting.} In this step we transfer the local root functions from $\mathcal O$ to $\prod_{j=1}^{n} U_j$, where $\{ U_j \}$ is the collection of simply connected open neighborhoods of the origin described in the statement of the theorem. Given an arbitrary $\widehat{\pmb{\xi}} \in \prod_{j=1}^{n} U_j$, we set $\widehat{\mathbf z} = \Phi_s(\widehat{\pmb{\xi}})$. Recalling that $\{ w_i : 1 \leq i \leq d \}$ is a set of holomorphic root functions of $F$ on $\mathcal O_{\widehat{\mathbf z}}$, we set 
\[ W_i = w_i \circ \Phi_s \quad \text{ on } \Phi_s^{-1}(\mathcal O_{\widehat{\mathbf z}}), \quad \text{ and } \quad \mathcal W_{\widehat{\pmb{\xi}}} = \{W_1, \cdots, W_d \}, \]
so that $\mathcal W_{\widehat{\pmb{\xi}}}$ is a collection of $d$ locally holomorphic functions on $\Phi_s^{-1}(\mathcal O_{\widehat{\mathbf z}})$ satisfying
\begin{equation} F\bigl(\Phi_s(\pmb{\xi}), W_i(\pmb{\xi}) \bigr) = 0 \quad \text{ for } \pmb{\xi} \in \Phi_s^{-1}(\mathcal O_{\widehat{\mathbf z}}). \label{local-lifting} \end{equation}        

\paragraph{Global definition of the roots}
We will now define the global root functions $b_i$ on $\prod_{j=1}^{n}U_j$. Fix a base point $\pmb{\xi}^{\ast} \in \prod_{j=1}^{n} U_j \setminus \mathbb H$ and fix a labeling of the roots in $\mathcal W_{\pmb{\xi}^{\ast}} = \{b_1, \cdots, b_d \}$ that will be used for the remainder of the proof. For any other point $\widehat{\pmb{\xi}} \in \prod_{j=1}^{n} U_j \setminus \mathbb H$ fixed but arbitrary we choose a continuous curve
\begin{equation} \Gamma: [0,1] \rightarrow \prod_{j=1}^{n}U_j \setminus \mathbb H \text{ such that } \Gamma(0) = \pmb{\xi}^{\ast}, \; \Gamma(1) = \widehat{\pmb{\xi}}. \label{def-Gamma-curve} \end{equation} For every $1 \leq i \leq d$, we define a function $\eta_i[\Gamma]:[0,1] \rightarrow \mathbb C$ by setting $\eta_i[\Gamma](0) = b_i(\pmb{\xi}^{\ast})$ and extending $\eta_i[\Gamma]$ to $[0,1]$ by analytic continuation so that 
\begin{equation} F(\Phi_s(\Gamma(t)), \eta_i[\Gamma](t)) = 0, \quad 0\leq t \leq 1. \label{lifted-root} \end{equation} 
We claim that 
\begin{equation}
(\eta_1[\Gamma], \cdots, \eta_d[\Gamma])(1) \text{ is independent of $\Gamma$ satisfying (\ref{def-Gamma-curve}).} \label{JA-claim}
\end{equation}  

Assuming the claim for the moment, we can define a global function \[b_i:\prod_{j=1}^{n} U_j \setminus \mathbb H \rightarrow \mathbb C \text{  by setting }  b_i(\widehat{\pmb{\xi}}) = \eta_i[\Gamma](1). \] The claim implies that each $b_i$ is well-defined on $\prod_{j=1}^{n}U_j \setminus \mathbb H$. It is holomorphic on $\prod_{j=1}^{n} U_j \setminus \mathbb H$, since by (\ref{lifted-root}) and (\ref{local-lifting}) every $b_i$ agrees on $\Phi_s^{-1}(\mathcal O_{\widehat{\mathbf z}})$ with $w_r \circ \Phi_s$ for some $1 \leq r \leq d$. Since the coefficients $c_\nu$ of the monic polynomial $F$ are bounded, so are the roots $b_i$. Therefore $\{b_i : 1 \leq i \leq d \}$ extend as holomorphic functions on all of $\prod_{j=1}^{n}U_j$ which is the main conclusion of Theorem \ref{thm-JA-complex}. Further the product of $b_i(\pmb{\xi}) - b_j(\pmb{\xi})$ for $i \ne j$ is the discriminant $\Delta_F \circ \Phi_s(\pmb{\xi})$ which is normal crossings, thereby implying that each of the differences $b_i - b_j$ is normal crossings as well.   


\paragraph{Proof of (\ref{JA-claim})} It therefore remains to prove the claim (\ref{JA-claim}). By the monodromy theorem, (\ref{JA-claim}) is already known to hold for curves $\Gamma$ and $\Gamma'$ that obey (\ref{def-Gamma-curve}) and are fixed endpoint homotopic to each other. We therefore only need to restrict attention to curves that are non-homotopic.

By a standard reduction, it suffices to choose $\widehat{\pmb{\xi}} = \pmb{\xi}^{\ast}$ and prove the following statement instead. Let $\Gamma^{\ast}:[0,1] \rightarrow \prod_{j=1}^{n} U_j \setminus \mathbb H$ be any closed curve not homotopic to zero, with $\Gamma^{\ast}(0) = \Gamma^{\ast}(1) = \pmb{\xi}^{\ast}$. Let $\tau_{\Gamma^{\ast}}$ denote the permutation of roots of $F(\pmb{\xi}^{\ast}, \cdot)$ induced by $\Gamma^{\ast}$. Namely if $\eta_i[\Gamma^{\ast}]$ is defined using (\ref{lifted-root}), then 
\[ \eta_i[\Gamma^{\ast}](1) = \eta_{j}[\Gamma^{\ast}](0) \quad \text{ for } j = \tau_{\Gamma^\ast}(i). \]
Our goal is to show that \begin{equation} \tau_{\Gamma^{\ast}} = \text{ identity } \label{JA-revised-goal} \end{equation} for all closed curves $\Gamma^{\ast}$ originating at $\pmb{\xi}^{\ast}$ and not homotopic to zero.       

We observe two facts. First, as we have seen in an earlier part of the proof, $\Gamma^{\ast} \mapsto \tau_{\Gamma^{\ast}}$ is a group homomorphism from the fundamental group of $\prod_{j=1}^{n} U_j \setminus \mathbb H$ to the permutation group $\mathbb S_d$. Second, the fundamental group of $\prod_{j=1}^{n} U_j \setminus \mathbb H$ is isomorphic to that of $\prod_{j=1}^{n} (U_j \setminus \{0\})$, which in turn is isomorphic to $\mathbb Z^n$. A set of generators of this fundamental group is given by $\{ \lambda_1, \cdots, \lambda_n \}$, with
\[ \lambda_j(t) = \bigl\{ \bigl(\xi_1^{\ast}, \cdots, \xi_{j-1}^{\ast}, \epsilon_j e^{2 \pi i t}, \xi_{j+1}^{\ast}, \cdots, \xi_n^{\ast} \bigr) : 0 \leq t \leq 1 \bigr\}, \quad 1 \leq j \leq n, \] 
where $\epsilon_j > 0$ is chosen sufficiently small so that $\{|z_j| = \epsilon_j \} \subseteq U_j \setminus \{0\}$. It therefore suffices to verify (\ref{JA-revised-goal}) only for $\Gamma^{\ast} \in \{ \lambda_1, \cdots, \lambda_n \}$.  

To see this, fix $1 \leq j \leq n$. We recall that (\ref{lifted-root}) holds for $\Gamma = \lambda_j$, so comparing with (\ref{def-zetai}) we arrive at the conclusion 
\[ \eta_i[\lambda_j] = \zeta_i[\Phi_s \circ \lambda_j] \quad \text{ and hence } \quad \tau_{\lambda_j}= \pi_{\Phi_s \circ \lambda_j}. \]
But 
\begin{align*} 
\Phi_s \circ \lambda_j(t) &= \bigl\{(z_1^{\ast}, \cdots, z_{j-1}^{\ast}, \epsilon_j^{s} e^{2 \pi i s t}, z_{j+1}^{\ast}, \cdots, z_n^{\ast} ) : 0 \leq t \leq 1 \bigr\}  \\
&= \underbrace{\gamma_j + \cdots + \gamma_j}_{s \text{ times}}, \quad \text{ where } \mathbf z^{\ast} = \Phi_s(\pmb{\xi}^{\ast}), \text{ and } \\ 
\gamma_j &= \bigl\{(z_1^{\ast}, \cdots, z_{j-1}^{\ast}, \epsilon_j^{s} e^{2 \pi i t}, z_{j+1}^{\ast}, \cdots, z_n^{\ast} ) : 0 \leq t \leq 1 \bigr\} \text{ is a curve in $\mathcal O$.}
\end{align*}
But then by (\ref{choice-s}), $\tau_{\lambda_j} = \pi_{\gamma_j}^s =$ identity, as claimed.  

\section{Appendix: Proof of Theorem \ref{prop-cii-nd-ndim}} \label{sec-appendix2}
Without loss of generality and for notational simplicity, it is convenient for this proof to work in dimension $n$ (as opposed to $(n+1)$ as in the rest of the paper). Accordingly, we set \[ F(\mathbf x) = \sum_{\begin{subarray}{c}\pmb{\kappa} > \mathbf 0 \\ \pmb{\kappa} \in \mathbb Z^n/N \end{subarray}} a_{\pmb{\kappa}} \mathbf x^{\pmb{\kappa}},  \]
where the fractional power series converges absolutely and uniformly in a neighborhood $U_0 \supseteq (-1,1)^n$ of the origin in $\mathbb R^n$. In view of the discussion in \S \ref{sec-intro} on sublevel set growth estimates, it is known that $\nu_0(F) = \mu_0(F)$, so it suffices to show that 
\begin{equation} \label{nu-delta} \nu_0(F) \leq \delta_0(F). \end{equation}   
We will prove (\ref{nu-delta}) by ignoring all ``cancellation'' in $F$, i.e., by estimating $F$ pointwise from above by a non-negative function with an isolated zero at the origin. The main ingredients of the proof are certain tools in convex geometry used in \cite{Nagel-Pramanik09} \cite{Nagel-Pramanik-preprint} in the study of monomial polyhedra and polyhedral cones. We point to the relevant parts of these references for results that may be applied verbatim. 

The main steps of the proof are contained in the following sequence of lemmas.  

\begin{lemma} \label{lemma-defI} There exists a finite set of multi-indices $\mathcal I \subseteq \{\pmb{\kappa} : a_{\pmb{\kappa}} \ne 0 \}$ such that 
\[ \{ \pmb{\kappa} : a_{\pmb{\kappa}} \ne 0 \} \subseteq \bigcup_{\pmb{\kappa} \in \mathcal I} \bigl[\pmb{\kappa} + \mathbb R^n_{\geq 0} \bigr]. \] 
\end{lemma} 
\begin{proof}
This is a consequence of Proposition 3.3 in \cite{Nagel-Pramanik-preprint}. More precisely, for every non-empty subset $S \subseteq \{1, 2, \cdots, n \}$ we apply this proposition with $\mathbf q = -\mathbf 1$, and with $\mathcal P$ and $\mathfrak N$ replaced by $\mathcal P_{S}$ and $\mathfrak N_{S}$ respectively, where $\mathcal P_{S} = \{\mathbf e_j : j \in S \}$, and   
\[ \mathfrak N_S = \left\{\pmb{\kappa}N : a_{\pmb{\kappa}} \ne 0, \pmb{\kappa} = (\kappa_1, \cdots, \kappa_n), \kappa_j \ne 0 \text{ if and only if } j \in S \right\} \subseteq \mathbb Z^n. \]  
Here $\{\mathbf e_1, \cdots, \mathbf e_n \}$ denotes the canonical orthonormal basis in $\mathbb R^n$. If $\mathfrak M_S \subseteq \mathfrak N_S$ is the finite subset whose existence is guaranteed by that proposition, then \[ \mathcal I = \bigcup_{S} \{\pmb{\kappa} : \pmb{\kappa}N \in \mathfrak M_S \} \] satisfies the conclusion of the lemma.   
\end{proof} 
\begin{corollary} \label{corollary-defI} 
Let $\mathcal I = \{\mathbf p_1, \cdots, \mathbf p_d \} $ be as in Lemma \ref{lemma-defI}. Then 
\begin{enumerate}[(a)]
\item there exists a constant $C > 0$ such that 
\begin{equation} |F(\mathbf x)| \leq \frac{C}{2} \sum_{j=1}^{d} |\mathbf x^{\mathbf p_j}|, \quad \mathbf x \in U_0. \label{ptwise-upper-bound} \end{equation}  
\item The following set containment holds: $\mathcal E_{\epsilon}(F) \supseteq E_{\epsilon}$, where \begin{equation} E_{\epsilon} := \left\{ \mathbf x \in U_0 : |\mathbf x^{\mathbf p_j}| \leq \frac{\epsilon}{Cd}, j=1, 2, \cdots, d \right\} \label{Eep} \end{equation}
and $\mathcal E_{\epsilon}(F)$ is the sublevel set defined in (\ref{def-nu0}).  
\end{enumerate} 
\end{corollary} 
\begin{proof}
It follows from Lemma \ref{lemma-defI} that $F(\mathbf x) = \sum_{\pmb{\kappa}\in \mathcal I} \mathbf x^{\pmb{\kappa}} F_{\pmb{\kappa}}(\mathbf x)$, where each $F_{\pmb{\kappa}}$ is a fractional power series converging absolutely and uniformly on $U_0$, and hence is bounded therein. The result (\ref{ptwise-upper-bound}) follows. The second conclusion is an easy consequence of the first and is left to the reader.   
\end{proof}
A critical step in the proof of Theorem \ref{prop-cii-nd-ndim} is the following estimate. 
\begin{proposition} \label{prop-Eep-sizebound} 
Let $E_{\epsilon}$ be as in (\ref{Eep}), $0 < \epsilon \leq 1$. Then there exists a constant $C$ depending of $F$ but independent of $\epsilon$ such that 
\begin{equation} \label{Eep-sizebound} |E_{\epsilon}| \geq C^{-1} \epsilon^{\delta_0(F)}. \end{equation}  
\end{proposition} 

\begin{proof}[Conclusion of the proof of Theorem \ref{prop-cii-nd-ndim}]
Let us assume Proposition \ref{prop-cii-nd-ndim} for the moment. Then on one hand Corollary \ref{corollary-defI}(b) combined with (\ref{Eep-sizebound}) implies  
\begin{equation} \label{EepF-lower} |\mathcal E_{\epsilon}(F)| \geq |E_{\epsilon}| \geq C_{\delta} \epsilon^{\delta_0(F)}. \end{equation}
On the other hand, it follows from the definition (\ref{def-nu0}) of $\nu_0(F)$ that 
\begin{equation} \label{EepF-upper} |\mathcal E_{\epsilon}(F)| \leq C_{\nu} \epsilon^{\nu} \end{equation} 
for any $\nu < \nu_0(F)$. Combining (\ref{EepF-lower}) and (\ref{EepF-upper}) and letting $\epsilon \rightarrow 0$, we find that $\nu \leq \delta_0(F)$. Taking supremum over $\nu$ yields the desired conclusion (\ref{nu-delta}).     
\end{proof}

\begin{proof}[Proof of Proposition \ref{prop-Eep-sizebound}]
It remains to prove the proposition. For this we estimate $|E_{\epsilon}|$ as follows, 
\begin{equation} \label{Eepest-step1}
\begin{aligned} |E_{\epsilon}| \geq \int_{E_{\epsilon} \cap (0,1]^n} d\mathbf x &= \int_{\text{Log}(E_{\epsilon} \cap (0,1]^n)} e^{\mathbf y \cdot \mathbf 1} \, d\mathbf y \\ &= \Bigl[ \log\Bigl(\frac{Cd}{\epsilon} \Bigr)\Bigr]^n \int_{\Omega} \exp \left[\log \Bigl(\frac{Cd}{\epsilon} \Bigr)\mathbf y \cdot \mathbf 1 \right] \, d\mathbf y,
\end{aligned}
\end{equation} 
where $\text{Log}(x_1, \cdots, x_n) = (\log x_1, \cdots, \log x_n)$ for $\mathbf x \in (0, \infty)^n$ and hence $\Omega$ is the linear polyhedron given by 
\begin{equation} \label{def-Omega} \begin{aligned} \Omega := &\Bigl[\log \Bigl(\frac{Cd}{\epsilon} \Bigr) \Bigr]^{-1} \text{Log}[E_{\epsilon} \cap (0,1]^n] \\ = &\left\{\mathbf y \in (-\infty, 0]^n : \mathbf y \cdot \mathbf p_j \leq -1, j=1, \cdots, d \right\}. \end{aligned}  \end{equation} 
Clearly, $\Omega$ is nonempty as it contains all vectors whose entries are sufficiently negative. Thus the rightmost integral in (\ref{Eepest-step1}) is strictly positive. 

Sharp size estimates of exponential measures on linear polyhedra (or equivalently, monomial-type measures on monomial polyhedra) have been studied in \cite{Nagel-Pramanik09}. In particular, 
Lemma 4.10 in \cite{Nagel-Pramanik09} implies that for some constant $C_n$ depending only on dimension, 
\begin{equation} \label{Eepest-step2} \int_{\Omega}\exp \left[\log \Bigl(\frac{Cd}{\epsilon} \Bigr)\mathbf y \cdot \mathbf 1 \right] \,  d\mathbf y \geq C_n^{-1} e^{M(\mathbf e)} |\widehat{\Omega}| \end{equation} 
where $M(\mathbf v) := \sup_{\mathbf y \in \Omega} \mathbf y \cdot \mathbf v \text{ for any } \mathbf v \in \mathbb R^n$, 
\[ \begin{aligned}
\mathbf e &= \log \Bigl(\frac{Cd}{\epsilon} \Bigr) \mathbf 1, \quad M(\mathbf e) = \log \Bigl(\frac{Cd}{\epsilon} \Bigr) M(\mathbf 1), \text{ and } \\ \widehat{\Omega} &= \left\{\mathbf y \in \Omega: M(\mathbf e) - 1 \leq \mathbf y \cdot \mathbf e < M(\mathbf e)  \right\} \\ &= \left\{ \mathbf y \in \Omega :  M(\mathbf 1) - \Bigl[\log\Bigl(\frac{Cd}{\epsilon} \Bigr) \Bigr]^{-1} < \mathbf y \cdot \mathbf 1 < M(\mathbf 1)  \right\}. \end{aligned} 
\]
We will prove in Corollaries \ref{corollary-M1est} and \ref{corollary-Omegahat} below that 
\begin{align}
M(\mathbf 1) \geq - \delta_0(F), \quad &\text{hence} \quad M(\mathbf e) \geq \delta_0(F) \log \Bigl( \frac{\epsilon}{Cd}\Bigr)  \text{ and } \label{M1-est} \\
|\widehat{\Omega}| &\geq C^{-1} \left[ \log\Bigl(\frac{Cd}{\epsilon} \Bigr)\right]^{-n}. \label{Omegahat-est}
\end{align}
Substituting these in (\ref{Eepest-step2}) and (\ref{Eepest-step1}) leads to the lower bound in (\ref{Eep-sizebound}), completing the proof.  
\end{proof} 
\begin{lemma} \label{lemma-duality-nd}
The following inequality holds: 
\[ \sup \left\{ \sum_{j=1}^{d} \beta_j : \mathbf 1 = \sum_{j=1}^{d} \beta_j \mathbf p_j, \; \beta_j \geq 0, j=1, \cdots, d \right\} \leq \delta_0(F). \] 
\end{lemma} 
\begin{proof}
Let $(\beta_1, \cdots, \beta_d) \in [0, \infty)^d$ be a vector such that $\mathbf 1 = \sum_{j=1}^{d} \beta_j \mathbf p_j$. Set $\beta_0 = \sum_{j=1}^{d} \beta_j$. Then \[ \beta_0^{-1} \mathbf 1 = \sum_{j=1}^{d} \left( \frac{\beta_j}{\beta_0}\right) \mathbf p_j \in \text{ conv hull}\{\mathbf p_j : j = 1, \cdots, d \} \subseteq \text{NP}(F). \] 
By the definition of Newton {\atxt exponent} , $\beta_0 \leq \delta_0(F)$, proving the lemma. 
\end{proof} 
\begin{corollary} \label{corollary-M1est} 
(\ref{M1-est}) holds.
\end{corollary} 
\begin{proof}
By Lemma 2.7 of \cite{Nagel-Pramanik09} (which is simply the strong duality theorem of linear programming), 
\begin{align*} M(\mathbf 1) &= \inf \left\{ - \sum_{j=1}^{d} \beta_j   : \mathbf 1 = \sum_{j=1}^{d} \beta_j \mathbf p_j, \; \beta_j \geq 0, j=1, \cdots, d \right\}  \\
&= - \sup \left\{\sum_{j=1}^{d} \beta_j : \mathbf 1 = \sum_{j=1}^{d} \beta_j \mathbf p_j, \; \beta_j \geq 0, j=1, \cdots, d \right\} \\ 
&\geq - \delta_0(F), 
\end{align*}   
where the last step follows from Lemma \ref{lemma-duality-nd}. 
\end{proof} 
\begin{lemma} \label{lemma-Omegahat}
Let $\Omega$ be as in (\ref{def-Omega}), and let $\mathbf w \in \Omega$ be a point obeying $\mathbf w \cdot \mathbf 1 = M(\mathbf 1)$. Let $\mathcal Q$ denote the cube 
\[ \mathcal Q := \left\{\mathbf y : w_j - \Bigl[n \log \Bigl(\frac{Cd}{\epsilon} \Bigr) \Bigr]^{-1} \leq y_j \leq w_j, j=1, \cdots, n \right\}.   \]    
Then $\mathcal Q \subseteq \widehat{\Omega}$. 
\end{lemma} 
\begin{proof}
Let $\mathbf y \in \mathcal Q$. Summing up the defining inequalities for $\mathcal Q$, we find that 
\[ M(\mathbf 1) - \Bigl[\log \Bigl(\frac{Cd}{\epsilon} \Bigr) \Bigr]^{-1} = \sum_{j=1}^{n} \left(w_j - \Bigl[n \log \Bigl(\frac{Cd}{\epsilon} \Bigr) \Bigr]^{-1} \right) \leq \mathbf y \cdot \mathbf 1 \leq \sum_{j} w_j = M(\mathbf 1). \] 
On the other hand, for $j=1, \cdots, d$, 
\[ \mathbf y \cdot \mathbf p_j = \mathbf w \cdot \mathbf p_j + (\mathbf y - \mathbf w) \cdot \mathbf p_j \leq \mathbf w \cdot \mathbf p_j \leq -1, \]
where the last inequality uses the fact that $\mathbf p_j \in [0, \infty)^n$ and $\mathbf y - \mathbf w \in (-\infty, 0]^n$. We have therefore verified that $\mathbf y$ satisfies all the defining relations for $\widehat{\Omega}$, which is the desired conclusion.  
\end{proof} 
\begin{corollary} \label{corollary-Omegahat} 
(\ref{Omegahat-est}) holds. 
\end{corollary} 
\begin{proof} 
By Lemma \ref{lemma-Omegahat}, $|\widehat{\Omega}| \geq |\mathcal Q| = n^{-n} [\log(Cd/\epsilon)]^{-n}$. 
\end{proof} 

\section{Acknowledgements}
The third author thanks Professor Detlef M\"uller for valuable comments and Professor Kalle Karu for mentioning a reference for Lemma \ref{lemma-integral-closure}. The second author was partially supported by grants
DMS-0138167 and DMS-0551894 from the  National Science Foundation. The third author was partially supported by grant 22R82900 from the Natural Sciences and Engineering Research 
Council of Canada.

\vskip.2in
\noindent{\sc Department of Mathematics}\\
\noindent{\sc Columbia University}\\
\noindent{\sc  New York, NY 10027}\\
\noindent{\tt{tcollins@math.columbia.edu}}

\vskip.2in

\noindent{\sc Department of Mathematics}\\
\noindent{\sc University of Rochester}\\
\noindent{\sc Rochester, NY 14627}\\
\noindent{\tt{allan@math.rochester.edu}}

\vskip.2in

\noindent{\sc Department of Mathematics}\\
\noindent{\sc University of British Columbia}\\
\noindent{\sc Vancouver, B.C., Canada V6T 1Z2}\\
\noindent{\tt{malabika@math.ubc.ca}}

\end{document}